\documentclass{article}
\usepackage{amsmath,amsthm,amsfonts,amssymb}
\usepackage{mathtools,mathrsfs}
\usepackage[inline]{enumitem}
\usepackage[hidelinks]{hyperref}
\usepackage[space]{accsupp}
\usepackage{graphicx}
\usepackage{openmoji}
\usepackage{transparent}
\usepackage{bussproofs}
\usepackage{setspace}
\usepackage{array,tabularx,caption}
\usepackage{tikz}
\usepackage{caption}
\usepackage{subcaption}
\usepackage{subfiles}
\usepackage{lipsum}
\usepackage{colonequals}
\usepackage{url}
\usepackage{here}

\usetikzlibrary{decorations.pathmorphing}

\tikzset{squig/.style={
  decorate,
  decoration={
    snake,
    pre=lineto,
    post=lineto,
    pre length=0.15cm,
    post length=0.15cm,
    amplitude=1
  }
}}

\tikzset{rel/.style={
  shorten >=4pt, shorten <=4pt,
}}

\theoremstyle{plain}
\newtheorem{theorem}{Theorem}[section]
\newtheorem*{theorem*}{Theorem}
\newtheorem{proposition}[theorem]{Proposition}
\newtheorem*{proposition*}{Proposition}
\newtheorem{lemma}[theorem]{Lemma}
\newtheorem*{lemma*}{Lemma}

\newtheorem*{claim*}{Claim}
\newtheorem{corollary}[theorem]{Corollary}
\newtheorem*{corollary*}{Corollary}

\newtheorem*{fact*}{Fact}

\newtheorem*{guide*}{Guide}

\theoremstyle{definition}
\newtheorem{definition}[theorem]{Definition}

\newtheorem{example}[theorem]{Example}
\newtheorem*{example*}{Example}
\newtheorem{remark}[theorem]{Remark}
\newtheorem*{remark*}{Remark}
\newtheorem{problem}[theorem]{Problem}

\newenvironment{subproof}[1][\proofname]{%
  \begin{proof}[#1]%
}{%
  \end{proof}%
}

\DeclarePairedDelimiterX{\set}[1]{\{}{\}}{\setargs{#1}}
\NewDocumentCommand{\setargs}{>{\SplitArgument{1}{;}}m}
{\setargsaux#1}
\NewDocumentCommand{\setargsaux}{mm}
{\IfNoValueTF{#2}{#1} {#1\nonscript\:\delimsize\vert\allowbreak\nonscript\:\mathopen{}#2}}%

\DeclareSymbolFont{symbolsC}{U}{txsyc}{m}{n}
\DeclareMathSymbol{\lewis}{\mathrel}{symbolsC}{74}

\newcommand*{\imp}{\mathrel{\Rightarrow}}
\newcommand*{\siff}{\mathrel{\Leftrightarrow}}

\newcommand*{\tiff}{\mathrel{\text{iff}}}
\newcommand*{\tand}{\mathrel{\text{\&}}}

\newcommand*{\tor}{\mathrel{\text{or}}}

\newcommand*{\defeq}{\coloneq}
\newcommand*{\defiff}{\mathrel{\;\ratio\Longleftrightarrow\;}}
\newcommand*{\defsiff}{\mathrel{\ratio\Leftrightarrow}}

\newcommand*{\name}[1]{\textsf{#1}}
\newcommand*{\Logic}[1]{\textbf{#1}}
\newcommand*{\Theory}[1]{\mathcal{#1}}
\newcommand*{\Model}[1]{\mathfrak{#1}}

\newcommand*{\FrameClass}[1]{\mathscr{F}_{#1}}

\newcommand*{\Singularity}{w_\bot}

\newcommand*{\MRel}{\sqsubset}
\newcommand*{\MRelRev}{\sqsupset}
\newcommand*{\IRel}{\leq}
\newcommand*{\IRelRev}{\geq}

\newcommand*{\Lang}{\mathscr{L}}
\newcommand*{\PropVar}{\mathord{\mathrm{PropVar}}}
\newcommand*{\PF}{\Lang_p}
\newcommand*{\MFb}{\Lang_\Box}
\newcommand*{\MFd}{\Lang_\Diamond}
\newcommand*{\MF}{\Lang_{\Box\Diamond}}

\newcommand*{\Cl}{\Logic{Cl}}
\newcommand*{\Int}{\Logic{Int}}
\newcommand*{\K}{\Logic{K}}

\newcommand*{\SatE}{\Vdash^e}
\newcommand*{\SatI}{\Vdash^i}
\newcommand*{\SatCK}{\Vdash^{\text{c}}}
\newcommand*{\SatIK}{\Vdash^{\text{h}}}

\newcommand*{\ValidE}{\vDash^e}
\newcommand*{\ValidI}{\vDash^i}
\newcommand*{\ValidCK}{\vDash^{\text{c}}}
\newcommand*{\ValidIK}{\vDash^{\text{h}}}

\newcommand*{\Var}{\mathrm{V}}
\newcommand*{\VarPos}{\mathrm{V}^+}
\newcommand*{\VarNeg}{\mathrm{V}^-}
\newcommand*{\VarAny}{\mathrm{V}^\bullet}
\newcommand*{\VarAnother}{\mathrm{V}^\circ}

\usepackage{subfiles}

\title{Intuitionistic and Constructive Modal Logics for Classical Modal Logicians}
\author{Yuta Sato \footnote{Email: 231x032x@gsuite.kobe-u.ac.jp} \footnote{Graduate School of System Informatics,
Kobe University,
1-1 Rokkodai, Nada, Kobe 657-8501, Japan.}}
\date{}

\begin{document}

\maketitle

\begin{abstract}
This paper gives a survey of propositional modal logic in an intuitionistic setting.
Mainly intended for readers already familiar with classical modal logic,
we discuss why there are two prominent families of such logics, intuitionistic modal logics (IMLs) and constructive modal logics (CMLs),
what motivates their study, and how they compare with each other.
We then take a look at the logics lying between the CML and IML traditions,
and survey various extensions and $\Diamond$/$\Box$-free fragments.
We further record several observations which seem to be implicit or absent in the literature;
we point out a small defect in the original semantics for $\Logic{CK}$ and
prove Lyndon interpolation theorems for $\Logic{CK}$ and $\Logic{WK}$.
\end{abstract}

\section{Introduction} \label{sec:intro}

Classical modal logicians may have a good command of \emph{Kripke semantics}.
It provides intuitive interpretations of modal operators $\Box$ and $\Diamond$,
which, in the classical setting, become naturally interdefinable by de Morgan duality,
and many normal extensions of classical propositional modal logic $\Logic{K}$ can be studied
by considering corresponding conditions on Kripke frames.
They may also feel comfortable with intuitionistic propositional logic $\Logic{Int}$;
its Kripke semantics can be seen as a slight modification of the modal one,
based on preordered Kripke frames and upward-closed valuations,
with the resulting satisfaction relation having a distinct property called \emph{persistency}.

Things become complicated once we consider modal operators in an intuitionistic base, however;
in an intuitionistic version of modal logic,
$\Box$ and $\Diamond$ are no longer interdefinable in general due to the failure of double negation elimination,
and even the apparently innocent question of how Kripke semantics should ensure persistency comes with more than one answer.
Also, in the literature, there seem to be two prominent families of such logics, \emph{intuitionistic} modal logics (IMLs) and \emph{constructive} modal logics (CMLs), which may further confuse newcomers to the subject;
the terms ``intuitionistic'' and ``constructive'' are distinguished here even though they may be used interchangeably in many fields of study,
and the two families adopt different notions of validity, both of which may at first glance seem unnatural from the classical viewpoint.

This survey is written primarily for readers familiar with classical modal logic who want an overview of intuitionistic and constructive modal logics,
answering questions such as how CMLs and IMLs are defined, what motivates their study, and how they compare with each other.
It also takes a look at logics that cannot be categorized as either CMLs or IMLs, which we shall call \emph{borderline logics},
and $\Diamond$/$\Box$-free logics, which largely avoid the tension between them.
In addition, it surveys intuitionistic variants of standard classical modal logics such as $\Logic{S4}$ and $\Logic{GL}$,
and intuitionistic logics with modal operators other than $\Box$ or $\Diamond$.

Although the paper is mainly expository, it also records several observations which seem to be implicit or absent in the literature.
Thus, readers already familiar with the topic may also find the paper useful.
In particular:
\begin{itemize}
  \item We point out a small defect in the original semantics for $\Logic{CK}$ due to Mendler and de Paiva \cite{mendler_constructive_2005}, and we provide a simple fix.
  \item We prove that $\Logic{CK}$ and $\Logic{WK}$ enjoy the Lyndon interpolation property by a straightforward adaptation of Wijesekera's \cite{wijesekera_constructive_1990} proof of Craig interpolation for $\Logic{WK}$.
  \item We collect a number of separation results for the $\Diamond$/$\Box$-free fragments of logics in the region between CMLs and IMLs.
\end{itemize}

This paper is organized as follows.
Section \ref{sec:prelim} is mainly devoted to preliminary definitions and results on modal and intuitionistic logics.
We also introduce Simpson's well-known requirements for a modal logic to be called ``intuitionistic''.
In Section \ref{sec:models}, we introduce birelational Kripke semantics for modal logic in an intuitionistic setting,
and we see that there are several ways to ensure persistency.
We then show that different ways of ensuring persistency may validate different sets of formulas,
which ultimately leads to the reason why CMLs and IMLs part company with each other.
We also take a look at the semantics of $\Diamond$/$\Box$-free logics and a truth-preserving operation that is only possible under the $\Diamond$/$\Box$-freeness.
In Sections \ref{sec:cml} and \ref{sec:iml}, we take a close look at CMLs and IMLs, respectively,
presenting the specialized semantics and several known results in both families of logics.
We will see their respective advantages and disadvantages
and the reasons why different authors have preferred one approach over the other.
Then in Section \ref{sec:other}, we take a look at recent studies on some borderline logics between CMLs and IMLs,
and also provide some separation results on their $\Diamond$/$\Box$-free fragments.
In Section \ref{sec:extensions}, we consider the intuitionistic variants of
common classical modal logics such as $\Logic{KD}$, $\Logic{KT}$, $\Logic{K4}$, $\Logic{GL}$, $\Logic{S4}$, $\Logic{S5}$, and related systems,
and present several results on these extensions of CMLs, IMLs and $\Diamond$/$\Box$-free intuitionistic modal logics.
We will also introduce some related logics, namely Heyting--Lewis logic and propositional lax logic.

\section{Preliminaries} \label{sec:prelim}

Throughout this section, we refer to Chagrov and Zakharyaschev \cite{CZ97} for the basic definitions and properties on modal and intuitionistic logic.

\begin{definition} \label{def:languages-formulas}
  Let $\PropVar$ be the set of all propositional variables, and let:
  \begin{align*}
    \PF  &= \PropVar \cup \set{ \land, \lor, \to, \bot },
      & \MFb &= \PF \cup \set{ \Box }, \\
    \MFd &= \PF \cup \set{ \Diamond },
      & \MF  &= \PF \cup \set{ \Box, \Diamond },
  \end{align*}
  where $\bot$ is nullary, $\Box, \Diamond$ are unary, and $\land, \lor, \to$ are binary.
  We use $\top$ and $\neg \varphi$ as abbreviations for $\bot \to \bot$ and $\varphi \to \bot$, respectively.
  The formulas of each language are built from $\PropVar$ and $\bot$ in the usual inductive way using these connectives.
  For a language $\Lang$, by abuse of notation, we also write $\Lang$ for the set of all formulas in $\Lang$.
\end{definition}

\begin{definition} \label{def:logics-extensions}
  An \emph{$\Lang$-logic} is a set $L \subseteq \Lang$ closed under modus ponens and 
  uniform substitution. The prefix ``$\Lang$-'' may be omitted if no confusion can arise.
  For $\Lang \subseteq \Lang'$,
  an $\Lang'$-logic $L'$ is an \emph{extension} of an $\Lang$-logic $L$ if $L \subseteq L'$,
  and a \emph{conservative extension} if, in addition, $L' \cap \Lang = L$.
\end{definition}

\begin{definition} \label{def:logics-sums}
  For a logic $L$ and a set $X$ of schematic axioms and rules, $L + X$ is the least
  extension of $L$ containing every instance of the axioms in $X$ and closed under the rules in $X$.
  When $X = \set{ \name{X}_1, \dots, \name{X}_n }$ is finite, we may omit set braces to write $L + \name{X}_1 + \dots + \name{X}_n$.
\end{definition}

\begin{definition} \label{def:axiomatization}
  An extension $L'$ of $L$ is said to be \emph{(finitely) axiomatized over $L$} if $L' = L + X$ for some (finite) $X$.
  In particular, a logic $L$ is said to be \emph{(finitely) axiomatized} if it is (finitely) axiomatized over $\emptyset$,
  or, equivalently, a Hilbert-style calculus for $L$ is given by (finitely many) schematic axioms and rules.
\end{definition}

Let $\Int$ and $\Cl$ be intuitionistic and classical propositional logic, respectively,
with their standard finite Hilbert-style axiomatizations.

\begin{definition} \label{def:relations}
  A binary relation $R$ is said to be
  \emph{serial} if $\forall x\, \exists y\, (x R y)$,
  \emph{reflexive} if $\forall x\, (x R x)$,
  \emph{symmetric} if $\forall x,y\, (x R y \imp y R x)$,
  \emph{transitive} if $\forall x,y,z\, (x R y \tand y R z \imp x R z)$,
  and \emph{Euclidean} if $\forall x,y,z\, (x R y \tand x R z \imp y R z)$.
  We also say $R$ is \emph{preordered} if $R$ is both reflexive and transitive.
\end{definition}

\subsection{Modal logics and their semantics}

We shall first briefly summarize the basic definitions and results on classical modal logics.
We shall adopt the formulation of $\Logic{K}$ that is usually seen in the context of nonnormal modal logics,
the reason for which we will see momentarily.

\begin{table}[ht]
  \centering
  \onehalfspacing
  \caption{Basic modal axioms and rules. Names are borrowed from \cite{dalmonte_intuitionistic_2020}.}
  \label{tab:basic-modal-axioms}
  \begin{tabular}{rl|rl}
    $\name{M}_\Box$     & \AxiomC{$\varphi \to \psi$} \UnaryInfC{$\Box \varphi \to \Box \psi$} \DisplayProof &
    $\name{M}_\Diamond$ & \AxiomC{$\varphi \to \psi$} \UnaryInfC{$\Diamond \varphi \to \Diamond \psi$} \DisplayProof \\
    $\name{N}_\Box$      & $\Box \top$ & $\name{N}_\Diamond$ & $\neg \Diamond \bot$ \\
    $\name{C}_\Box$     & $(\Box \varphi \land \Box \psi) \to \Box (\varphi \land \psi)$ &
    $\name{C}_\Diamond$ & $\Diamond(\varphi \lor \psi) \to (\Diamond \varphi \lor \Diamond \psi)$ \\
    $\name{Dual}_\Box$ & $\Diamond \varphi \leftrightarrow \neg \Box \neg \varphi$ &
    $\name{Dual}_\Diamond$ & $\Box \varphi \leftrightarrow \neg \Diamond \neg \varphi$ \\
  \end{tabular}
  \doublespacing
\end{table}

\begin{definition}
  We let $\Logic{K}_\Box \defeq \Cl + \name{M}_\Box + \name{N}_\Box + \name{C}_\Box$
     and $\Logic{K}_\Diamond \defeq \Cl + \name{M}_\Diamond + \name{N}_\Diamond + \name{C}_\Diamond$
  (see Table \ref{tab:basic-modal-axioms} for the definition of axioms and rules).
\end{definition}

\begin{remark}
  The reader may now be wondering why we didn't use the standard rule $\name{Nec}_\Box: \mathord{\vdash}\, \varphi \imp \mathord{\vdash}\, \Box \varphi$
  and the standard axiom $\name{K}_\Box: \Box (\varphi \to \psi) \to (\Box \varphi \to \Box \psi)$.
  First, $(\name{Nec}_\Box, \name{K}_\Box)$ and $(\name{M}_\Box, \name{N}_\Box, \name{C}_\Box)$ can be derived from each other even in an intuitionistic setting (\cite[pp.219--220]{bozic_models_1984}).
  Second, $\name{C}_\Diamond$ and $\name{N}_\Diamond$, the duals of $\name{C}_\Box$ and $\name{N}_\Box$ respectively, become very important when we talk about the differences between IMLs and CMLs.
\end{remark}

\begin{definition}
  We say $\Model{F} = (W, \MRel)$ is a \emph{Kripke frame (k-frame)} if $W$ is a non-empty set and $\MRel$ is a binary relation on $W$,
  and $\Model{M} = (\Model{F}, V)$ is a \emph{Kripke model (k-model)} if $\Model{F}$ is a k-frame and $V: \PropVar \to \mathcal{P}(W)$.
  Such $\Model{M}$ is said to be \emph{based on} $\Model{F}$.
  For $\Model{M} = (W, \MRel, V)$, we define a \emph{satisfaction relation} $\Vdash_\Model{M}$ on $W \times \MF$ by:
  \begin{enumerate}
    \item $x \nVdash_\Model{M} \bot$
    \item $x \Vdash_\Model{M} p \defiff x \in V(p)$
    \item $x \Vdash_\Model{M} \psi_1 \land \psi_2 \defiff (x \Vdash_\Model{M} \psi_1 \tand x \Vdash_\Model{M} \psi_2)$
    \item $x \Vdash_\Model{M} \psi_1 \lor \psi_2 \defiff (x \Vdash_\Model{M} \psi_1 \tor x \Vdash_\Model{M} \psi_2)$
    \item $x \Vdash_\Model{M} \psi_1 \to \psi_2 \defiff (x \Vdash_\Model{M} \psi_1 \imp x \Vdash_\Model{M} \psi_2)$
    \item $x \Vdash_\Model{M} \Box \varphi \defiff \forall y \MRelRev x \, (y \Vdash_\Model{M} \varphi)$
    \item $x \Vdash_\Model{M} \Diamond \varphi \defiff \exists y \MRelRev x\, (y \Vdash_\Model{M} \varphi)$
  \end{enumerate}
\end{definition}

\begin{definition}
  Let $\varphi \in \MF$.
  We say $\varphi$ is \emph{valid in} a k-model $\Model{M}$, written as $\Model{M} \vDash \varphi$, if $x \Vdash_\Model{M} \varphi$ for every $x \in W$.
  We also say $\varphi$ is \emph{valid on} a k-frame $\Model{F}$, written as $\Model{F} \vDash \varphi$, if $\varphi$ is valid in every k-model $\Model{M}$ based on $\Model{F}$.
\end{definition}

\begin{theorem} \label{thm:k-completeness}
  Let $\Theory{K} = \set{ \varphi \in \MF; \forall \Model{F}: \text{k-frame}\ (\Model{F} \vDash \varphi) }$, then the following properties hold:
  \begin{enumerate}[label=(\arabic*)]
    \item $\Logic{K}_\Box = \Theory{K} \cap \MFb$ and $\Logic{K}_\Diamond = \Theory{K} \cap \MFd$.
    \item $\Logic{K}_\Box + \name{Dual}_\Box = \Theory{K}$ and $\Logic{K}_\Diamond + \name{Dual}_\Diamond = \Theory{K}$.
  \end{enumerate}
  \begin{proof}
    (1) Routine completeness proof. (2) Follows from (1) and the usual translations $\tau_\Box(\Diamond \varphi) = \neg \Box \neg \tau_\Box(\varphi)$ and $\tau_\Diamond(\Box \varphi) = \neg \Diamond \neg \tau_\Diamond(\varphi)$.
  \end{proof}
\end{theorem}

\begin{remark}
  The double negation elimination plays an important role in the proof of Theorem \ref{thm:k-completeness} (2).
  In particular, it uses the fact $\Logic{K}_\Box \vdash \neg \neg \Box \neg \neg \varphi \leftrightarrow \Box \varphi$.
  This would obviously fail in an intuitionistic setting, and IMLs and CMLs both treat $\Diamond$ as primitive and give up the full de Morgan duality between $\Box$ and $\Diamond$.
  The actual interaction between $\Box$ and $\Diamond$, however, varies between IMLs and CMLs, which we will see in later sections.
\end{remark}

By Theorem \ref{thm:k-completeness}, we obtain the following finite axiomatization of $\Theory{K}$:
\begin{definition}
  Let $\Logic{K} = \Logic{K}_\Box + \name{Dual}_\Box$.
\end{definition}

\subsection{Intuitionistic logic and its semantics}

Intuitionistic propositional logic $\Int$ also has a semantics based on Kripke frames and models.

\begin{definition}
  We say $\Model{F} = (W, \IRel)$ is an \emph{intuitionistic frame (i-frame)} if $W$ is a non-empty set and $\IRel$ is a preorder on $W$.
  We say $\Model{M} = (\Model{F}, V)$ is an \emph{intuitionistic model (i-model)} if $\Model{F}$ is an i-frame and $V: \PropVar \to \mathcal{P}(W)$,
  where each $V(p)$ is upward closed with respect to $\IRel$ (i.e. $\forall x, y\, (x \IRel y \tand x \in V(p) \imp y \in V(p))$).
  Such $\Model{M}$ is said to be \emph{based on} $\Model{F}$.
  For $\Model{M} = (W, \IRel, V)$, we define a \emph{satisfaction relation} $\Vdash_\Model{M}$ on $W \times \PF$ by:
  \begin{enumerate}
    \item $x \nVdash_\Model{M} \bot$
    \item $x \Vdash_\Model{M} p \defiff x \in V(p)$
    \item $x \Vdash_\Model{M} \psi_1 \land \psi_2 \defiff (x \Vdash_\Model{M} \psi_1 \tand x \Vdash_\Model{M} \psi_2)$
    \item $x \Vdash_\Model{M} \psi_1 \lor \psi_2 \defiff (x \Vdash_\Model{M} \psi_1 \tor x \Vdash_\Model{M} \psi_2)$
    \item $x \Vdash_\Model{M} \psi_1 \to \psi_2 \defiff \forall x' \IRelRev x\, (x' \Vdash_\Model{M} \psi_1 \imp x' \Vdash_\Model{M} \psi_2)$
  \end{enumerate}
  We write $\Model{M} \vDash \varphi$ (``$\varphi$ is \emph{valid in} $\Model{M}$'') to mean $\forall x \in W\,(x \Vdash_\Model{M} \varphi)$,
  and $\Model{F} \vDash \varphi$ (``$\varphi$ is \emph{valid on} $\Model{F}$'') to mean that $\Model{M} \vDash \varphi$ for every $\Model{M}$ based on $\Model{F}$.
\end{definition}

i-models possess a distinct property called \emph{persistency}.

\begin{proposition}[Persistency]
  Let $\Model{M} = (W, \IRel, V)$ be an i-model, and take any $\varphi \in \PF$.
  Then for any $x, x' \in W$, $x \Vdash_\Model{M} \varphi$ and $x \IRel x'$ imply $x' \Vdash_\Model{M} \varphi$.
  \begin{proof}
    Easy induction on the construction of $\varphi$.
  \end{proof}
\end{proposition}

\begin{theorem} \label{thm:int-completeness}
  $\Int = \set{ \varphi \in \PF; \forall \Model{F}: \text{i-frame}\ (\Model{F} \vDash \varphi) }$.
\end{theorem}

Here, i-frames are just preordered k-frames. On the other hand, i-models differ from k-models in that
(1) the satisfaction relation is only defined for $\PF$,
(2) the truth of $\psi_1 \to \psi_2$ is determined by $\IRel$, just like how the truth of $\Box(\psi_1 \to \psi_2)$ in k-model is determined by $\MRel$, and
(3) $V$ has an additional requirement with respect to $\IRel$.
As a matter of fact, we can embed $\Int$ into a modal logic that is complete with respect to preordered k-frames.
Let $\Logic{S4} \defeq \Logic{K} + \Box \varphi \to \varphi + \Box \varphi \to \Box \Box \varphi$.

\begin{proposition}
  $\Logic{S4} = \set{ \varphi \in \MF; \forall \Model{F}: \text{preordered k-frame}\ (\Model{F} \vDash \varphi) }$.
\end{proposition}

\begin{definition} \label{def:godel-translation}
  We inductively define the \emph{G\"odel translation} $(\cdot)^G: \PF \to \MF$:
  \begin{align*}
    \bot^G &\defeq \Box \bot, &
    p^G &\defeq \Box p, \\
    (\psi_1 \land \psi_2)^G &\defeq \psi_1^G \land \psi_2^G, &
    (\psi_1 \lor \psi_2)^G &\defeq \psi_1^G \lor \psi_2^G, \\
    (\psi_1 \to \psi_2)^G &\defeq \Box(\psi_1^G \to \psi_2^G).
  \end{align*}
\end{definition}

\begin{theorem}[{e.g.\ \cite[Thm.\,3.83]{CZ97}}]
  For any $\varphi \in \PF$, $\Int \vdash \varphi$ iff $\Logic{S4} \vdash \varphi^G$.
\end{theorem}

As such, $\Logic{S4}$ is called a \emph{modal companion} of $\Int$.

\begin{definition}
  Let $\Int \subseteq L \subseteq \Cl$ be closed under modus ponens and substitution.
  Such $L$ is called an \emph{intermediate logic}.
  We say a modal logic $M$ is a \emph{modal companion} of $L$
  if $\varphi \in L \siff (\varphi)^G \in M$ for any $\varphi \in \PF$.
\end{definition}

We note that there are infinitely many modal companions of $\Int$, and there are also (infinitely many) modal companions of $\Cl$,
the most famous one being $\Logic{S5} \defeq \Logic{S4} + \Diamond \varphi \to \Box \Diamond \varphi$.

Although $\Int$ is provably weaker than $\Cl$, there are several embeddings of $\Cl$ into $\Int$.
The simplest example is Glivenko's theorem:
$\Cl \vdash \varphi \siff \Int \vdash \neg \neg \varphi$.
We shall use the following refinement.
\begin{definition} \label{def:negative-translation}
  We inductively define the \emph{(G\"odel--Gentzen) double negation translation}
  $(\cdot)^N: \PF \to \PF$:
  \begin{align*}
    \bot^N &\defeq \bot, & p^N &\defeq \neg\neg p, \\
    (\varphi \land \psi)^N &\defeq \varphi^N \land \psi^N, &
    (\varphi \lor \psi)^N &\defeq \neg(\neg\varphi^N \land \neg\psi^N), \\
    (\varphi \to \psi)^N &\defeq \varphi^N \to \psi^N.
  \end{align*}
\end{definition}

\begin{theorem}[{e.g.\ \cite[Thm.\,2.54]{CZ97}}] \label{thm:negative-embedding}
  For any $\varphi \in \PF$, $\Cl \vdash \varphi$ iff $\Int \vdash \varphi^N$.
\end{theorem}

\subsection{The criterion for intuitionistic modal logics}

When one says \emph{intuitionistic} modal logics, one would naturally expect modal logics that, roughly speaking,
share similar properties to $\Int$ than to $\Cl$, and
have similar relationship to $\Int$ as classical modal logics do to $\Cl$.
Simpson \cite{simpson_proof_1994} has suggested several requirements for a modal logic to be called intuitionistic,
and we shall briefly introduce the concepts that are needed to explain them.

\begin{definition}
  A logic $L$ with a language $\Lang \supseteq \PF$ is said to enjoy the \emph{disjunction property (DP)} if,
  for every $\varphi, \psi \in \Lang$, $\varphi \lor \psi \in L$ iff $\varphi \in L$ or $\psi \in L$.
\end{definition}

\begin{proposition}[{\cite[Thms.\,1.31 and 2.64]{CZ97}}]
  $\Int$ enjoys DP, while $\Cl$ and $\Logic{K}$ lacks DP.
\end{proposition}



\begin{definition}
  Let $\Lang_{\mathrm{ST}}$ be a first-order language (see, e.g., \cite{shoenfield_mathematical_2018}) with a binary predicate symbol $\MRel$ and unary predicate symbols $V_p$ for every $p \in \PropVar$.
  For $\varphi \in \MF$, and a free variable $x$, the \emph{standard translation} $\mathrm{ST}_x(\varphi)$ is defined inductively:
  \begin{itemize}
    \item $\mathrm{ST}_x(\bot) = \bot$, $\mathrm{ST}_x(p) = V_p(x)$;
    \item $\mathrm{ST}_x(\psi_1 \circledcirc \psi_2) = \mathrm{ST}_x(\psi_1) \circledcirc \mathrm{ST}_x(\psi_2)$, for $\circledcirc \in \set{ \land, \lor, \to }$;
    \item $\mathrm{ST}_x(\Box \psi) = \forall y\,(x \MRel y \to \mathrm{ST}_y(\psi))$, where $y$ is fresh;
    \item $\mathrm{ST}_x(\Diamond \psi) = \exists y\,(x \MRel y \land \mathrm{ST}_y(\psi))$, where $y$ is fresh.
  \end{itemize}
\end{definition}

\begin{proposition}[{\cite[Prop.\,4.18]{CZ97}}]
  For any classical first-order structure $\mathscr{M}$ for $\Lang_{\text{ST}}$,
  we obtain a k-model $\Model{M} = (W, \MRel, V)$,
  where $W$ is the universe of $\mathscr{M}$, $x \MRel y \defsiff \mathord{\MRel^\mathscr{M}}(x, y)$, and $x \in V(p) \defsiff V_p^\mathscr{M}(x)$.
  Conversely, for any k-model $\Model{M}$, we obtain a classical first-order structure $\mathscr{M}$ for $\Lang_{\text{ST}}$.
  We note that these translations form a bijection between k-models and classical first-order structures for $\Lang_{\text{ST}}$.
  Then for any $\varphi \in \MF$, $\Model{M} \vDash \varphi$ iff $\forall x\,\mathrm{ST}_x(\varphi)$ is valid in $\mathscr{M}$.
\end{proposition}

\begin{corollary}
  $\varphi \in \Logic{K}$ iff $\forall x\,\mathrm{ST}_x(\varphi)$ is valid in every classical first-order structure for $\Lang_{\text{ST}}$.
\end{corollary}

Now we shall introduce Simpson's six requirements \cite[Section 3.2]{simpson_proof_1994} for a modal logic $L$ to be called intuitionistic.

\begin{enumerate}[label=(\arabic*)]
  \item $L$ should be a conservative extension of $\Int$.
  \item $L$ should be closed under modus ponens, and should contain all substitution instances of theorems of $\Int$.
  \item $L + \varphi \lor \neg \varphi$ should collapse into a classical modal logic just as $\Int + \varphi \lor \neg \varphi$ does to $\Cl$.
  \item $L$ should enjoy DP as $\Int$ does.
  \item Considering the fact $\Box$ translates to $\forall$ and $\Diamond$ translates to $\exists$ via the standard translation,
        $\Box$ and $\Diamond$ should be independent in $L$ as $\forall$ and $\exists$ are in intuitionistic first-order logic.
  \item There should be an intuitionistically comprehensible explanation of the meaning of $\Box$ and $\Diamond$, relative to which $L$ is sound and complete.
\end{enumerate}

The preceding two requirements should be understood against earlier work on the relation between intuitionistic modal logics and intuitionistic/intermediate predicate logics.
Bull \cite{bull_modal_extension_1965,bull_mipc_1966} already related the intuitionistic variant of $\Logic{S5}$ to intuitionistic predicate logic in light of Kripke semantics,
and Ono \cite{ono_intuitionistic_1977} explicitly compared $\Box$ and $\Diamond$ with the universal and existential quantifiers.
Ono and Suzuki \cite{ono_suzuki_relations_1988} later developed this viewpoint systematically as a correspondence between intuitionistic modal logics and intermediate predicate logics.
The sixth requirement is also rather ambiguous, but Simpson has proposed a way to determine an intuitionistic modal logic with a given frame condition.
Let $T$ be a first-order theory expressing the frame condition, then the intuitionistic modal logic should be determined by a set $S_T$
where $\varphi \in S_T$ iff $\forall x\,\mathrm{ST}_x(\varphi)$ is intuitionistically valid in every first-order structure for $\Lang_{\text{ST}}$ that validates the theory $T$.
But as Simpson himself pointed out, there are arbitrarily many formulations of theory $T$ for the same frame condition, some of which may not be intuitionistically equivalent.
In our paper, we shall not emphasize the sixth requirement (except for the intuitionistic equivalent of $\Logic{K}$, where $T$ should be trivial and thus should have no freedom of choice)
and instead focus on the other five ones.

\begin{definition} \label{def:iml}
  We say a modal logic $L$ is \emph{intuitionistic} if the following holds:
  \begin{enumerate}[label=(\arabic*)]
    \item $\Int = L \cap \PF$.
    \item $L$ is axiomatized over $\Int$ by a set of schematic axioms and rules.
    \item $L + \varphi \lor \neg \varphi$ is a classical modal logic extending $\Logic{K}$ (or $\Logic{K}_\Box$ or $\Logic{K}_\Diamond$, if the language of $L$ is $\MFb$ or $\MFd$, respectively).
    \item $L$ enjoys DP.
    \item If the language of $L$ is $\MF$, then $\name{Dual}_\Box \notin L$ and $\name{Dual}_\Diamond \notin L$.
  \end{enumerate}
  We note that our (2) is a stronger requirement than Simpson's (2), as we also require $L$ to be closed under uniform substitution and every additional rule in it.
  We also list some nice-to-have properties for $L$:
  \begin{itemize}
    \item $L$ should be finitely axiomatized if $L + \varphi \lor \neg \varphi$ is finitely axiomatized.
    \item $L$ should be complete and sound with respect to some class of Kripke-like frames, which will be introduced in later sections.
    \item If $L$ is meant to be the intuitionistic equivalent of $\Logic{K}$, then for any $\varphi \in \MF$, it should be that $\varphi \in L$ iff $\forall x\,\mathrm{ST}_x(\varphi)$ is valid in every intuitionistic first-order structure for $\Lang_{\text{ST}}$.
    \item $L$ should have a cut-admissible sequent calculus if $L + \varphi \lor \neg \varphi$ has.
    \item It should be able to embed $L + \varphi \lor \neg \varphi$ into $L$ by a Glivenko/G\"odel--Gentzen-style translation.
    \item $L$ should enjoy Craig, Lyndon, and uniform interpolation properties if $L + \varphi \lor \neg \varphi$ does, respectively.
  \end{itemize}
\end{definition}

Here, the interpolation properties mentioned above are defined as follows.

\begin{definition}
  For a language $\mathscr{L}$ and a formula $\varphi \in \mathscr{L}$,
  let us denote $\VarPos(\varphi)$ and $\VarNeg(\varphi)$ the sets of all propositional variables occurring positively and negatively in $\varphi$, respectively.
  We also let $\Var(\varphi) = \VarPos(\varphi) \cup \VarNeg(\varphi)$.
\end{definition}

\begin{definition} \label{def:cip-lip}
  A logic $L$ with a language $\mathscr{L}$ is said to enjoy \emph{Lyndon interpolation property (LIP)}
  if for every $\varphi, \psi \in \mathscr{L}$ such that $\varphi \to \psi \in L$,
  there is $\chi \in \mathscr{L}$ that satisfies the following conditions:
  \begin{enumerate}
    \item $\VarAny(\chi) \subseteq \VarAny(\varphi) \cap \VarAny(\psi)$, for each $\bullet \in \set{+, -}$;
    \item $\set{ \varphi \to \chi, \chi \to \psi } \subseteq L$.
  \end{enumerate}
  A logic $L$ is said to enjoy \emph{Craig interpolation property (CIP)} if it satisfies the above,
  but with the first condition of $\chi$ replaced by $\Var(\chi) \subseteq \Var(\varphi) \cap \Var(\psi)$.
\end{definition}

\begin{definition} \label{def:ulip-uip}
  A logic $L$ with a language $\mathscr{L}$ is said to enjoy \emph{Uniform interpolation property (UIP)} if
  for any $\varphi \in \mathscr{L}$ and any finite $P \subseteq \PropVar$,
  there is $\chi \in \mathscr{L}$ such that:
  \begin{enumerate}
    \item $\Var(\chi) \subseteq \Var(\varphi) \setminus P$;
    \item $\varphi \to \chi \in L$;
    \item $\chi \to \psi \in L$ for any $\psi \in \mathscr{L}$ such that $\varphi \to \psi \in L$ and $\Var(\psi) \cap P = \emptyset$.
  \end{enumerate}
\end{definition}

CIP easily follows from LIP.
It is also easy to see that UIP implies CIP by taking any $\psi$ such that $\varphi \to \psi \in L$ and letting $P = \Var(\varphi) \setminus \Var(\psi)$.
For a comprehensive overview on interpolation as a subject and various ways to establish it,
we refer to Gabbay and Maksimova \cite{GM09} and van der Giessen et al.'s preprint \cite{giessen_interpolation_2026}.
The following fact is well known.

\begin{proposition}
  $\Logic{Cl}$, $\Logic{Int}$, and $\Logic{K}$ all enjoy CIP, LIP, and UIP.
\end{proposition}

\section{Kripke semantics in an intuitionistic setting} \label{sec:models}

The idea of modal logics in an intuitionistic setting is first introduced by Fischer Servi \cite{fischer_servi_modal_1977}.
In the paper, she introduced the intuitionistic variant of $\Logic{S5}$ and proposed that
it should be treated as a bimodal logic of $\Logic{S4}$-modality and $\Logic{S5}$-modality, the former of which handles the intuitionistic implication
through G\"{o}del translation. Based on her bimodal interpretation of intuitionistic $\Logic{S5}$, several Kripke-like birelational semantics for IML and CML have been proposed.
They share mostly the same definition of frames and models, which is just a \emph{fusion} of i-frames/models and k-frames/models.
However, they use different kinds of satisfaction relations, which aim to ensure persistency through different means.
We will focus on two dominant methods;
we may call one \emph{extrinsic}, where the interpretation of modal operators remains unchanged and persistency is asserted externally by frame conditions,
and the other one \emph{intrinsic}, where it is modified to ensure persistency without any frame conditions.

\begin{definition}
  We say $\Model{F} = (W, \IRel, \MRel)$ is a \emph{fusion frame (f-frame)}\footnote{
    We use ``fusion'' to remain neutral between the IML and CML traditions.
  } if $W$ is a non-empty set, $\IRel$ is a preorder on $W$, and $\MRel$ is a binary relation on $W$.
  We say $\Model{M} = (\Model{F}, V)$ is a \emph{fusion model (f-model)} if $\Model{F}$ is an f-frame and $V: \PropVar \to \mathcal{P}(W)$,
  where each $V(p)$ is upward closed with respect to $\IRel$.
  Such $\Model{M}$ is said to be \emph{based on} $\Model{F}$.
\end{definition}

\begin{definition}[see, e.g., \cite{bozic_models_1984}]
  For an f-model $\Model{M} = (W, \IRel, \MRel, V)$, we define an \emph{extrinsic satisfaction relation} $\SatE_\Model{M}$ on $W \times \MF$ by:
  \begin{enumerate}
    \item $x \SatE_\Model{M} \psi_1 \to \psi_2 \defiff \forall x' \IRelRev x\, (x' \SatE_\Model{M} \psi_1 \imp x' \SatE_\Model{M} \psi_2)$
    \item $x \SatE_\Model{M} \Box \varphi \defiff \forall y \MRelRev x \, (y \SatE_\Model{M} \varphi)$
    \item $x \SatE_\Model{M} \Diamond \varphi \defiff \exists y \MRelRev x\, (y \SatE_\Model{M} \varphi)$
    \item The other cases are identical to the classical $\Vdash_\Model{M}$.
  \end{enumerate}
  We write $\Model{M} \ValidE \varphi$ (``$\varphi$ is \emph{(e-)valid in} $\Model{M}$'') to mean $\forall x \in W\,(x \SatE_\Model{M} \varphi)$,
  and $\Model{F} \ValidE \varphi$ (``$\varphi$ is \emph{(e-)valid on} $\Model{F}$'') to mean that $\Model{M} \ValidE \varphi$ for every $\Model{M}$ based on $\Model{F}$.
\end{definition}

\begin{definition}[see, e.g., \cite{wijesekera_constructive_1990, simpson_proof_1994}]
  For an f-model $\Model{M} = (W, \IRel, \MRel, V)$, we define an \emph{intrinsic satisfaction relation} $\SatI_\Model{M}$ on $W \times \MF$ by:
  \begin{enumerate}
    \item $x \SatI_\Model{M} \psi_1 \to \psi_2 \defiff \forall x' \IRelRev x\, (x' \SatI_\Model{M} \psi_1 \imp x' \SatI_\Model{M} \psi_2)$
    \item $x \SatI_\Model{M} \Box \varphi     \defiff \forall x' \IRelRev x\, \forall y' \MRelRev x'\, (y' \SatI_\Model{M} \varphi)$
    \item $x \SatI_\Model{M} \Diamond \varphi \defiff \forall x' \IRelRev x\, \exists y' \MRelRev x'\, (y' \SatI_\Model{M} \varphi)$
    \item The other cases are identical to the classical $\Vdash_\Model{M}$.
  \end{enumerate}
  We write $\Model{M} \ValidI \varphi$ (``$\varphi$ is \emph{(i-)valid in} $\Model{M}$'') to mean $\forall x \in W\,(x \SatI_\Model{M} \varphi)$,
  and $\Model{F} \ValidI \varphi$ (``$\varphi$ is \emph{(i-)valid on} $\Model{F}$'') to mean that $\Model{M} \ValidI \varphi$ for every $\Model{M}$ based on $\Model{F}$.
\end{definition}

\begin{definition}
  Let $\Model{M} = (W, \IRel, \MRel, V)$ be an f-model and let $\Vdash^\bullet_\Model{M}$ ($\bullet \in \set{i, e}$).
  We say $\Vdash^\bullet_\Model{M}$ is \emph{persistent} if for every $\varphi \in \MF$ and every $x, x' \in W$,
  \begin{equation*}
    \left(x \Vdash^\bullet_\Model{M} \varphi \ \tand\  x \IRel x'\right) \implies x' \Vdash^\bullet_\Model{M} \varphi.
  \end{equation*}
  We also say $\Vdash^\bullet_\Model{M}$ is \emph{$\Box$-persistent} if the above condition holds for every $\varphi \in \MFb$ instead,
  and is \emph{$\Diamond$-persistent} if it holds for every $\varphi \in \MFd$ instead.
\end{definition}

\begin{proposition}[{\cite[Lem.\,1.1.6]{wijesekera_constructive_1990}}] \label{prop:intrinsic-persistency}
  $\SatI_\Model{M}$ is persistent for any f-model $\Model{M}$.
  \begin{proof}
    Let $\Model{M} = (W, \IRel, \MRel, V)$ and take any $\varphi \in \MF$.
    We use an induction on the construction of $\varphi$.
    We shall only prove the cases with modal operators.

    Suppose that $x \IRel x'$ and $x \SatI_\Model{M} \Box \psi$,
    and take any $x'', y''$ such that $x' \IRel x'' \MRel y''$.
    Here, $x \IRel x' \IRel x'' \MRel y''$ implies $x \IRel x'' \MRel y''$ by $\IRel$ being preorder,
    then $x \SatI_\Model{M} \Box \psi$ implies $y'' \SatI_\Model{M} \psi$,
    so $x' \SatI_\Model{M} \Box \psi$.

    Now suppose that $x \IRel x'$ and $x \SatI_\Model{M} \Diamond \psi$,
    and take any $x''$ such that $x' \IRel x''$, then $x \IRel x' \IRel x''$ implies $x \IRel x''$ by $\IRel$ being preorder.
    Since $x \SatI_\Model{M} \Diamond \psi$, there is $y'' \MRelRev x''$ such that $y'' \SatI_\Model{M} \psi$,
    so $x' \SatI_\Model{M} \Diamond \psi$.
  \end{proof}
\end{proposition}

Now it should be clear that there is an f-model $\Model{M}$ in which $\SatE_\Model{M}$ is not persistent.
Since the definition of $x \SatE_\Model{M} \Box \varphi$ does not assert anything about the truth of $\Box \varphi$ in $y \IRelRev x$,
we can easily construct a countermodel in which $x' \nVdash^e_\Model{M} \Box \varphi$.
To assert persistency, we need to pose the following frame conditions on the frame of $\Model{M}$.

\begin{definition}[{\cite[Defs.\,2\,(iv) and 13\,(iv)]{bozic_models_1984}}] \label{def:frame-conditions}
  Let $\Model{F} = (W, \IRel, \MRel)$ be an f-frame.
  \begin{itemize}
    \item $\Model{F}$ is said to be a \emph{$\Box$-p}\footnote{
            Also known by downward confluence \cite{balbiani_constructive_2021}, and forth-down confluence \cite{aguilera_gocalculus_2022}.
          } frame
          if $(\IRel \cdot \MRel) \subseteq (\MRel \cdot \IRel)$, that is, $\forall x,x',y'\, (x \IRel x' \MRel y' \imp \exists y\,(x \MRel y \IRel y'))$.
          Let us denote by $\FrameClass{\Box\text{-p}}$ the class of all $\Box$-p frames.
    \item $\Model{F}$ is said to be a \emph{$\Diamond$-p}\footnote{
            Also known by forward confluence \cite{balbiani_constructive_2021}, F1 \cite{simpson_proof_1994}, forth-up confluence \cite{aguilera_gocalculus_2022}, and $\name{C}_{\Diamond}$-strong \cite{groot_semantical_2025}.
          } frame
          if $(\IRel^{-1} \cdot \MRel) \subseteq (\MRel \cdot \IRel^{-1})$, that is, $\forall x,x',y\, (x' \IRelRev x \MRel y \imp \exists y'\,(x' \MRel y' \IRelRev y))$.
          Let us denote by $\FrameClass{\Diamond\text{-p}}$ the class of all $\Diamond$-p frames.
  \end{itemize}
\end{definition}

\begin{figure}[H]
  \centering

  \begin{subfigure}{.4\textwidth}
    \centering
    \begin{tikzpicture}
      \fill (0, 0) circle (2pt);
      \draw (0, 0) node[below]{$x$};

      \fill (0, 2) circle (2pt);
      \draw (0, 2) node[above]{$x'$};

      \fill (2, 0) circle (2pt);
      \draw (2, 0) node[below]{$\exists y$};

      \fill (2, 2) circle (2pt);
      \draw (2, 2) node[above]{$y'$};

      \draw [thick, ->, rel] (0,0) to (0,2);
      \draw (0,1) node[left]{$\IRel$};
      \draw [thick, ->, rel, squig, dashed] (0,0) to (2,0);
      \draw (1,0) node[below]{$\MRel$};
      \draw [thick, ->, rel, dashed] (2,0) to (2,2);
      \draw (2,1) node[right]{$\IRel$};
      \draw [thick, ->, rel, squig] (0,2) to (2,2);
      \draw (1,2) node[above]{$\MRel$};
    \end{tikzpicture}
    \caption*{The $\Box$-p condition}
  \end{subfigure}%
  \begin{subfigure}{.4\textwidth}
    \centering
    \begin{tikzpicture}
      \fill (0, 0) circle (2pt);
      \draw (0, 0) node[below]{$x$};

      \fill (0, 2) circle (2pt);
      \draw (0, 2) node[above]{$x'$};

      \fill (2, 0) circle (2pt);
      \draw (2, 0) node[below]{$y$};

      \fill (2, 2) circle (2pt);
      \draw (2, 2) node[above]{$\exists y'$};

      \draw [thick, ->, rel, squig] (0,0) to (2,0);
      \draw (0,1) node[left]{$\IRel$};
      \draw [thick, ->, rel] (0,0) to (0,2);
      \draw (1,0) node[below]{$\MRel$};
      \draw [thick, ->, rel, dashed] (2,0) to (2,2);
      \draw (2,1) node[right]{$\IRel$};
      \draw [thick, ->, rel, squig, dashed] (0,2) to (2,2);
      \draw (1,2) node[above]{$\MRel$};
    \end{tikzpicture}
    \caption*{The $\Diamond$-p condition}
  \end{subfigure}

  \caption{Visualization of Definition \ref{def:frame-conditions}}
  \label{fig:frame-conditions}
\end{figure}

\begin{proposition}[{\cite[Lems.\,2, 3, 16, and 17]{bozic_models_1984}}] \label{prop:extrinsic-persistency}
  Let $\Model{F} = (W, \IRel, \MRel)$ be an f-frame. Then:
  \begin{enumerate}[label=(\arabic*)]
    \item $\Model{F} \in \FrameClass{\Box\text{-p}}$ iff $\SatE_\Model{M}$ is $\Box$-persistent for every $\Model{M}$ based on $\Model{F}$.
    \item $\Model{F} \in \FrameClass{\Diamond\text{-p}}$ iff $\SatE_\Model{M}$ is $\Diamond$-persistent for every $\Model{M}$ based on $\Model{F}$.
  \end{enumerate}
  \begin{proof}
    (1 $\Rightarrow$) Take any $\Model{M} = (\Model{F}, V)$ and $\varphi \in \MFb$, then use an induction on the construction of $\varphi$.
    Consider the case when $\varphi = \Box \psi$. Take any $x, x'$ such that $x \IRel x'$ and $x \SatE_\Model{M} \Box \psi$.
    Take any $y'$ such that $x' \MRel y'$, then by the $\Box$-p condition, there is $y$ such that $x \MRel y \IRel y'$.
    Here, $x \SatE_\Model{M} \Box \psi$ implies $y \SatE_\Model{M} \psi$,
    and the induction hypothesis implies $y' \SatE_\Model{M} \psi$.
    Therefore, $x' \SatE_\Model{M} \Box \psi$.

    (1 $\Leftarrow$)  Suppose that $\Model{F} \notin \FrameClass{\Box\text{-p}}$, then there are $x, x', y'$ such that $x \IRel x' \MRel y'$ but no $y$ such that $x \MRel y \IRel y'$.
    Take a fresh variable $p^\ast \in \PropVar$, then we construct a countermodel $\Model{M} = (\Model{F}, V)$ by letting $w \in V(p) \defsiff (p \ne p^\ast \tor \exists v\,(x \MRel v \IRel w))$.
    Here, each $V(p)$ is upward closed with respect to $\IRel$.
    Then, $x \SatE_\Model{M} \Box p^\ast$ by $\forall z \MRelRev x\,(z \SatE_\Model{M} p^\ast)$, and $x' \nVdash^e_\Model{M} \Box p^\ast$ by $y' \nVdash^e_\Model{M} p^\ast$.

    (2 $\Rightarrow$) Take any $\Model{M} = (\Model{F}, V)$ and $\varphi \in \MFd$, then use an induction on the construction of $\varphi$.
    Consider the case when $\varphi = \Diamond \psi$. Take any $x, x'$ such that $x \IRel x'$ and $x \SatE_\Model{M} \Diamond \psi$,
    then there is $y$ such that $x \MRel y \SatE_\Model{M} \psi$.
    Now that $x' \IRelRev x \MRel y$, then by the $\Diamond$-p condition, there is $y'$ such that $x' \MRel y' \IRelRev y$.
    Here, $y \IRel y'$ implies $y' \SatE_\Model{M} \psi$ by the induction hypothesis.
    This and $x' \MRel y'$ imply $x' \SatE_\Model{M} \Diamond \psi$.

    (2 $\Leftarrow$)  Suppose that $\Model{F} \notin \FrameClass{\Diamond\text{-p}}$, then there are $x, x', y$ such that $x' \IRelRev x \MRel y$ but no $y'$ such that $x' \MRel y' \IRelRev y$.
    Take a fresh variable $p^\ast \in \PropVar$, then we construct a countermodel $\Model{M} = (\Model{F}, V)$ by letting $w \in V(p) \defsiff (p \ne p^\ast \tor w \IRelRev y)$.
    Here, each $V(p)$ is upward closed with respect to $\IRel$.
    Then, $x \SatE_\Model{M} \Diamond p^\ast$ by $y \SatE_\Model{M} p^\ast$, and $x' \nVdash^e_\Model{M} \Diamond p^\ast$ by $\forall z' \MRelRev x'\,(z' \nVdash^e_\Model{M} p^\ast)$.
  \end{proof}
\end{proposition}

\begin{proposition} \label{prop:combined-persistency}
  Let $\Model{F} = (W, \IRel, \MRel)$ be an f-frame.
  $\Model{F} \in \FrameClass{\Box\text{-p}} \cap \FrameClass{\Diamond\text{-p}}$ iff $\SatE_\Model{M}$ is persistent for every $\Model{M}$ based on $\Model{F}$.
  \begin{proof}
    ($\Rightarrow$)
    Take any $\Model{M} = (\Model{F}, V)$ and $\varphi \in \MF$,
    then we use an induction on the construction of $\varphi$.
    We shall only consider the cases with modal operators.
    The case when $\varphi = \Box \psi$ is proven similarly to Proposition \ref{prop:extrinsic-persistency} (1 $\Rightarrow$),
    and the case when $\varphi = \Diamond \psi$ is proven similarly to Proposition \ref{prop:extrinsic-persistency} (2 $\Rightarrow$).

    ($\Leftarrow$)
    For any $\Model{M} = (\Model{F}, V)$,
    $\SatE_\Model{M}$ is persistent by the assumption,
    and so both $\Box$-persistent and $\Diamond$-persistent.
    Therefore, $\Model{F} \in \FrameClass{\Box\text{-p}}$ and $\Model{F} \in \FrameClass{\Diamond\text{-p}}$ by Proposition \ref{prop:extrinsic-persistency} (1 $\Leftarrow$) and (2 $\Leftarrow$), respectively.
  \end{proof}
\end{proposition}

\begin{corollary}
  Let $\Model{F} = (W, \IRel, \MRel)$ be an f-frame.
  $\SatE_\Model{M}$ is persistent for every $\Model{M}$ based on $\Model{F}$
  iff
  $\SatE_\Model{M}$ is $\Box$-persistent and $\Diamond$-persistent for every $\Model{M}$ based on $\Model{F}$.
\end{corollary}

\begin{problem}
  Does the above corollary hold locally? In other words, take any f-model $\Model{M} = (W, \IRel, \MRel, V)$,
  then can we show that $\SatE_\Model{M}$ is persistent whenever $\SatE_\Model{M}$ is $\Box$-persistent and $\Diamond$-persistent?
\end{problem}

Now the reader may be curious to what extent these two versions of satisfaction relations coincide.
We first show that the two satisfaction relations coincide in $\Box$-p and $\Diamond$-p f-models.
We define the following sets of formulas for convenience.
\begin{align*}
  \Theory{V}^i &\defeq \set{ \varphi \in \MF; \forall \Model{F}: \text{f-frame}\, (\Model{F} \ValidI \varphi) }, \\
  \Theory{V}^e &\defeq \set{ \varphi \in \MF; \forall \Model{F} \in \FrameClass{\Box\text{-p}} \cap \FrameClass{\Diamond\text{-p}}\, (\Model{F} \ValidE \varphi) }, \\
  \Theory{V}^e_\Box &\defeq \set{ \varphi \in \MFb; \forall \Model{F} \in \FrameClass{\Box\text{-p}}\, (\Model{F} \ValidE \varphi) }, \\
  \Theory{V}^e_\Diamond &\defeq \set{ \varphi \in \MFd; \forall \Model{F} \in \FrameClass{\Diamond\text{-p}}\, (\Model{F} \ValidE \varphi) }. \\
\end{align*}

\begin{proposition} \label{prop:int-ext-coincides}
  Let $\Model{M} = (W, \IRel, \MRel, V)$ be a $\Box$-p and $\Diamond$-p f-model, then \\
  $x \SatI_\Model{M} \varphi \iff x \SatE_\Model{M} \varphi$ for any $\varphi \in \MF$ and $x \in W$.
  \begin{proof}
    We use an induction on the construction of $\varphi$.
    We shall consider the cases with modal operators, as the other cases are either trivial or easily proven with the induction hypothesis.
    We note that $\SatE_\Model{M}$ is persistent by the assumption and Proposition \ref{prop:combined-persistency}.

    ($\Box\Rightarrow$)
    Suppose $x \SatI_\Model{M} \Box \psi$, then $\forall x',y'\, ((x \IRel x' \MRel y') \imp y' \SatI_\Model{M} \psi)$.
    Now take any $y$ such that $x \IRel x \MRel y$, then it follows that $y \SatI_\Model{M} \psi$, so $y \SatE_\Model{M} \psi$ by the induction hypothesis.
    Therefore, $x \SatE_\Model{M} \Box \psi$.

    ($\Box\Leftarrow$)
    Suppose $x \SatE_\Model{M} \Box \psi$, then $\forall y\, (x \MRel y \imp y \SatE_\Model{M} \psi)$.
    Now take any $x', y'$ such that $x \IRel x' \MRel y'$,
    then $x \IRel x' \SatE_\Model{M} \Box \psi$ by persistency,
    then $y' \SatE_\Model{M} \psi$,
    so $y' \SatI_\Model{M} \psi$ by the induction hypothesis.
    Therefore, $x \SatI_\Model{M} \Box \psi$.

    ($\Diamond\Rightarrow$)
    Suppose $x \SatI_\Model{M} \Diamond \psi$, then $\forall x'\, (x \IRel x' \imp \exists y'\, (x' \MRel y' \tand y' \SatI_\Model{M} \varphi))$.
    This and $x \IRel x$ imply that there is $y$ such that $x \MRel y$ and $y \SatI_\Model{M} \psi$, then $y \SatE_\Model{M} \psi$ by the induction hypothesis.
    Therefore, $x \SatE_\Model{M} \Diamond \psi$.

    ($\Diamond\Leftarrow$)
    Suppose $x \SatE_\Model{M} \Diamond \psi$, then there is $y$ such that $x \MRel y$ and $y \SatE_\Model{M} \psi$.
    Now take any $x'$ such that $x \IRel x'$,
    then by $\Diamond$-p, $x' \IRelRev x \MRel y$ implies that there is $y'$ such that $x' \MRel y' \IRelRev y$.
    Here, $y \IRel y' \SatE_\Model{M} \psi$ by persistency,
    so $y' \SatI_\Model{M} \psi$ by the induction hypothesis.
    Therefore, $x \SatI_\Model{M} \Diamond \psi$.
  \end{proof}
\end{proposition}

\begin{corollary} \label{cor:i-subset-e}
  $\Theory{V}^i \subseteq \Theory{V}^e$.
  \begin{proof}
    Suppose that $\varphi \notin \Theory{V}^e$, then there is a $\Box$-p and $\Diamond$-p f-model $\Model{M} = (W, \IRel, \MRel, V)$ and $w \in W$ such that $w \nVdash^e_\Model{M} \varphi$.
    By Proposition \ref{prop:int-ext-coincides}, $w \nVdash^e_\Model{M} \varphi$ implies $w \nVdash^i_\Model{M} \varphi$, so $\varphi \notin \Theory{V}^i$.
  \end{proof}
\end{corollary}

Now we show that the other direction does not hold.

\begin{proposition} \label{prop:separation-cdia}
  (1) $\name{C}_\Diamond \in \Theory{V}^e$. (2) $\name{C}_\Diamond \notin \Theory{V}^i$.
  \begin{proof}
    (1) Take any f-model $\Model{M} = (W, \IRel, \MRel, V)$.
    Here, $\Model{M}$ does not need to satisfy the $\Box$-p and $\Diamond$-p conditions.
    Then take any $\varphi, \psi \in \MF$ and $x, x'$ such that $x \IRel x' \SatE_\Model{M} \Diamond(\varphi \lor \psi)$.
    Then there is $y'$ such that $x' \MRel y' \SatE_\Model{M} \varphi \lor \psi$.
    Here, either $y' \SatE_\Model{M} \varphi$ or $y' \SatE_\Model{M} \psi$,
    then either $x' \SatE_\Model{M} \Diamond \varphi$ or $x' \SatE_\Model{M} \Diamond \psi$, respectively,
    so $x' \SatE_\Model{M} \Diamond \varphi \lor \Diamond \psi$.
    Therefore, $x \SatE_\Model{M} \name{C}_\Diamond$.

    (2) Consider the f-model $\Model{M} = (W, \IRel, \MRel, V)$, where $W = \set{ x, x', y, y' }$,
    $\mathord{\IRel} = \set{ (x, x), (x, x'), (x', x'), (y, y), (y', y') }$,
    $\mathord{\MRel} = \set{ (x, y), (x', y'), (y, y), (y', y') }$,
    and $V(p) = \set{ y }$, $V(q) = \set{ y' }$.
    The model $\Model{M}$ is visualized in Figure \ref{fig:kdia-countermodel}.

    \begin{figure}[ht]
      \centering
      \begin{tikzpicture}
      \fill (0, 2) circle (2pt);
      \draw (0, 2) node[left]{$x$};

      \draw (2, 2) circle (2pt);
      \draw (2, 2) node[right]{$y \SatI p$};

      \draw [thick, ->, squig, rel] (0,2) to (2, 2);
      \draw (1, 2) node[above]{$\MRel$};

      \fill (0, 4) circle (2pt);
      \draw (0, 4) node[left]{$x'$};

      \draw [thick, ->, rel] (0, 2) to (0,4);
      \draw (0,3) node[left]{$\IRel$};

      \draw (2, 4) circle (2pt);
      \draw (2, 4) node[right]{$y' \SatI q$};

      \draw [thick, ->, squig, rel] (0,4) to (2, 4);
      \draw (1, 4) node[above]{$\MRel$};

      \end{tikzpicture}
      \caption{The model $\Model{M}$. $\circ$ denotes a $\MRel$-reflexive world.}
      \label{fig:kdia-countermodel}
    \end{figure}

    It is easy to see that $w \SatI_\Model{M} \Diamond (p \lor q)$ for every $w \in W$,
    but $x \nVdash^i_\Model{M} \Diamond p \lor \Diamond q$.
    Therefore, $x \nVdash^i_\Model{M} \name{C}_\Diamond$.
  \end{proof}
\end{proposition}

\begin{corollary}
  $\Theory{V}^i \subsetneq \Theory{V}^e$.
\end{corollary}

In fact, we can show that even the $\Diamond$-free fragments of $\Theory{V}^i$ and $\Theory{V}^e$ do not coincide
because the presence of the $\Diamond$-p condition makes some $\Diamond$-free formulas e-valid on the f-frame.
The below example of such $\Diamond$-free formula is due to Das and Marin \cite[Prop.\,15]{das_intuitionistic_2023}.

\begin{proposition} \label{prop:separation-nnb}
  $\Theory{V}^e_\Box \subsetneq \Theory{V}^e \cap \MFb$. In particular:
  (1) $\neg \neg \Box \bot \to \Box \bot \in \Theory{V}^e$, and
  (2) $\neg \neg \Box \bot \to \Box \bot \notin \Theory{V}^e_\Box$.
  \begin{proof}
    (1) Take any $\Diamond$-p f-model $\Model{M} = (W, \IRel, \MRel, V)$,
    then take any $x, x'$ such that $x \IRel x' \SatE_\Model{M} \neg \neg \Box \bot$.
    Then $x' \IRel x'$ implies $x' \nVdash^e_\Model{M} \neg \Box \bot$, so there is $x''$ such that $x' \IRel x'' \SatE_\Model{M} \Box \bot$,
    which implies there is no $y''$ such that $x'' \MRel y''$.
    Now suppose, by way of contradiction, that there is $y'$ such that $x' \MRel y'$,
    then $x'' \IRelRev x' \MRel y'$, so by the $\Diamond$-p condition, there must be some $y''$ such that $x'' \MRel y'' \IRelRev y'$, which is a contradiction.
    Now that there is no $y'$ such that $x' \MRel y'$, we obtain $x' \SatE_\Model{M} \Box \bot$.
    Therefore, $x \SatE_\Model{M} \neg \neg \Box \bot \to \Box \bot$.

    (2) Consider the f-model $\Model{M} = (W, \IRel, \MRel, V)$, where $W = \set{ x, x', x'', y, y' }$,
    $\mathord{\IRel}$ is the reflexive and transitive closure of $\set{ (x, x'), (x', x''), (y, y') }$,
    $\mathord{\MRel} = \set{ (x, y), (x', y')}$,
    and $V$ is arbitrary.
    The model $\Model{M}$ is visualized in Figure \ref{fig:nnb-countermodel}.

    \begin{figure}[ht]
      \centering
      \begin{tikzpicture}
      \fill (0, 0) circle (2pt);
      \draw (0, 0) node[left]{$x$};

      \fill (2, 0) circle (2pt);
      \draw (2, 0) node[right]{$y$};

      \fill (0, 2) circle (2pt);
      \draw (0, 2) node[left]{$x'$};

      \fill (2, 2) circle (2pt);
      \draw (2, 2) node[right]{$y' \nVdash \bot$};

      \fill (0, 4) circle (2pt);
      \draw (0, 4) node[right]{$x'' \Vdash \Box \bot$};

      \draw [thick, ->, rel] (0,0) to (0,2);
      \draw (0,1) node[left]{$\IRel$};

      \draw [thick, ->, rel, squig] (0,0) to (2,0);
      \draw (1, 0) node[below]{$\MRel$};

      \draw [thick, ->, rel] (2,0) to (2,2);
      \draw (2,1) node[right]{$\IRel$};

      \draw [thick, ->, rel, squig] (0,2) to (2,2);
      \draw (1, 2) node[above]{$\MRel$};

      \draw [thick, ->, rel] (0,2) to (0,4);
      \draw (0,3) node[left]{$\IRel$};

      \end{tikzpicture}
      \caption{The model $\Model{M}$.}
      \label{fig:nnb-countermodel}
    \end{figure}

    It is easy to see that $\Model{M}$ is a $\Box$-p model.
    Also, we can easily check that $x' \SatE_\Model{M} \neg \neg \Box \bot$,
    $x' \nVdash^e_\Model{M} \Box \bot$,
    and so $x \nVdash^e_\Model{M} \neg \neg \Box \bot \to \Box \bot$.
  \end{proof}
\end{proposition}

\begin{proposition} \label{prop:diafree-vi-is-veb}
  $\Theory{V}^i \cap \MFb = \Theory{V}^e_\Box$.
  \begin{proof}
    ($\subseteq$)
    It suffices to show that for any $\Box$-p f-model $\Model{M} = (W, \IRel, \MRel, V)$, $\varphi \in \MFb$, and $w \in W$,
    $w \SatE_\Model{M} \varphi$ iff $w \SatI_\Model{M} \varphi$.
    This can easily be proven by induction on the construction of $\varphi$,
    reusing the proofs of Proposition \ref{prop:int-ext-coincides} ($\Box \Leftrightarrow$) by replacing persistency with $\Box$-persistency.

    ($\supseteq$) Suppose that there are $\varphi \in \MFb$, $\Model{M} = (W, \IRel, \MRel, V)$, and $w \in W$ such that $w \nVdash^i_\Model{M} \varphi$.
    We construct a new f-model $\Model{M'} = (W, \IRel, \MRel', V)$ by letting
    $\mathord{\MRel'} = \mathord{\MRel} \cup \set{ (x, y'); \exists x'\, (x \IRel x' \MRel y') }$,
    then $\Model{M'}$ is clearly a $\Box$-p f-model.
    \begin{claim*}
      $x \SatI_\Model{M} \psi \iff x \SatE_\Model{M'} \psi$ for every $\psi \in \MFb$.
      \begin{subproof}[Proof of claim]
        We use an induction on the construction of $\psi$.
        We shall only prove the case when $\psi = \Box \chi$.

        ($\Rightarrow$)
        Suppose $x \nVdash^e_\Model{M'} \Box \chi$, then
        there is $y$ such that $x \MRel' y \nVdash^e_\Model{M'} \chi$.
        Here, $y \nVdash^i_\Model{M} \chi$ by the induction hypothesis.
        Also, $x \MRel' y$ implies either $x \IRel x \MRel y$ or $x \IRel x' \MRel y$ for some $x'$.
        In any case, $x \nVdash^i_\Model{M} \Box \chi$ holds.

        ($\Leftarrow$)
        Suppose $x \nVdash^i_\Model{M} \Box \chi$, then
        there are $x', y'$ such that $x \IRel x' \MRel y' \nVdash^i_\Model{M} \chi$.
        Here, $y' \nVdash^e_\Model{M'} \chi$ by the induction hypothesis.
        Also, $x \IRel x' \MRel y'$ implies $x \MRel' y$ by definition.
        Therefore, $x \nVdash^e_\Model{M'} \Box \chi$ holds.
      \end{subproof}
    \end{claim*}
    By the claim, $w \nVdash^i_\Model{M} \varphi$ implies $w \nVdash^e_\Model{M'} \varphi$.
  \end{proof}
\end{proposition}

\begin{corollary}
  $\Theory{V}^i \cap \MFb \,\subsetneq\, \Theory{V}^e \cap \MFb$.
\end{corollary}

\begin{remark}
  It is clear that $\Theory{V}^i \cap \MFd \subsetneq \Theory{V}^e_\Diamond$ by the proof of Proposition \ref{prop:separation-cdia}.
  The reader may wonder if $\Theory{V}^e_\Diamond \subsetneq \Theory{V}^e \cap \MFd$ also holds.
  We will answer it positively in Corollary \ref{cor:fik-ik-lik-diamond-fragment-separation}.
\end{remark}

We shall briefly check the i- and e-validity of the other basic modal axioms in the following proposition:

\begin{proposition} \label{prop:soundness-basis} \leavevmode
  \begin{enumerate}[label=(\arabic*)]
    \item $\Theory{V}^i$, $\Theory{V}^e$, and $\Theory{V}^e_\Box$ are closed under $\name{M}_\Box$.
          $\Theory{V}^i$, $\Theory{V}^e$, and $\Theory{V}^e_\Diamond$ are closed under $\name{M}_\Diamond$.
    \item $\name{N}_\Box$ is in $\Theory{V}^i$, $\Theory{V}^e$, and $\Theory{V}^e_\Box$.
          $\name{N}_\Diamond$ is in $\Theory{V}^i$, $\Theory{V}^e$ and $\Theory{V}^e_\Diamond$.
    \item $\name{C}_\Box$ is in $\Theory{V}^i$, $\Theory{V}^e$, and $\Theory{V}^e_\Box$.
          $\name{C}_\Diamond$ is in $\Theory{V}^e_\Diamond$ (and $\Theory{V}^e$).
    \item Both $\name{Dual}_\Box$ and $\name{Dual}_\Diamond$ are not in $\Theory{V}^i$ and $\Theory{V}^e$.
  \end{enumerate}
  \begin{proof}
    Exercise. These are also proven throughout the paper by giving an axiomatization to each logic.
  \end{proof}
\end{proposition}

Now we shall summarize the facts we have obtained so far:
\begin{enumerate}
  \item $\Theory{V}^i$ contains all the basic modal axioms and rules, except for $\name{C}_\Diamond$.
  \item $\Theory{V}^e_\Box$ and $\Theory{V}^e_\Diamond$ contains all the basic $\Diamond$/$\Box$-free modal axioms and rules, respectively.
  \item $\Theory{V}^e$ contains all the basic modal axioms and rules, but contains strictly more $\Diamond$-free formulas than $\Theory{V}^e_\Box$, namely $\neg \neg \Box \bot \to \Box \bot$.
\end{enumerate}

The above observation puts both $\Theory{V}^i$ and $\Theory{V}^e$ in an awkward position;
one is too weak to capture the expected behavior of $\Diamond$, and the other allows $\Box$ to behave too much classically.

Here is where IMLs and CMLs part company with each other.
Researchers such as Wijesekera, Mendler, and de Paiva considered that
the distribution of $\Diamond$ over $\lor$ should be questioned from an intuitionistic point of view (``The fact that a disjunction A $\lor$ B is satisfiable in a context does not warrant the conclusion that one of the disjuncts is satisfiable'' \cite[Section 3]{mendler_constructive_2005}),
developing a family of CMLs which may lack $\name{C}_\Diamond$ \cite{wijesekera_constructive_1990} and even $\name{N}_\Diamond$ \cite{mendler_constructive_2005}.
On the other hand, researchers such as Fischer Servi and Simpson considered that
$\Box$ and $\Diamond$ should be related to some extent for a logic to be ``the true intuitionistic analogue of $\Logic{K}$'' \cite[p.51]{simpson_proof_1994},
and thus required $\name{C}_\Diamond$ and some additional axioms in their IMLs \cite{fischer_axiomatizations_1984,simpson_proof_1994}.

In the following sections, we will introduce CMLs, where $\SatI$ is adopted as the standard satisfaction relation with a slight modification,
and then IMLs, where a certain \emph{hybrid} of $\SatI$ and $\SatE$ is commonly used;
although $\SatE$ would appear natural from a point of view of classical modal logic,
it is not used as the standard satisfaction relation by both parties,
the reason of which we will see later in Section \ref{sec:other}.

Before moving on to CMLs and IMLs,
we shall take a look at the the $\Diamond$/$\Box$-free logics $\Theory{V}^e_\Box$ and $\Theory{V}^e_\Diamond$,
which largely avoid the tension between CMLs and IMLs.
We first give the following axiomatizations.

\begin{definition}
  Let $\Logic{iK}_\Box \defeq \Int + \name{M}_\Box + \name{N}_\Box + \name{C}_\Box$, $\Logic{iK}_\Diamond \defeq \Int + \name{M}_\Diamond + \name{N}_\Diamond + \name{C}_\Diamond$.
\end{definition}

\begin{theorem}[{\cite[Thms.\,1 and 4]{bozic_models_1984}}] \label{thm:ikb-ikd-completeness}
  $\Logic{iK}_\Box = \Theory{V}^e_\Box$ and $\Logic{iK}_\Diamond = \Theory{V}^e_\Diamond$.
\end{theorem}

As such, the $\Diamond$/$\Box$-free restrictions of $\SatE$ are commonly used among the researchers who focus on $\Diamond$/$\Box$-free intuitionistic modal logics.
Moreover, the $\Diamond$/$\Box$-freeness enables us some aggressive truth-preserving operations on f-frames.

\begin{definition}[{\cite[Defs.\,10 and 15]{bozic_models_1984}, \cite{iemhoff_properties_2005}}] \label{def:condensed-brilliant}
  Let $\Model{F} = (W, \IRel, \MRel)$ be an f-frame.
  \begin{itemize}
    \item We say $\Model{F}$ is \emph{$\Box$-condensed} if $(\IRel \cdot \MRel) \subseteq {\MRel}$,
          and is \emph{$\Box$-brilliant} (or \emph{strictly $\Box$-condensed}) if $(\MRel \cdot \IRel) \subseteq {\MRel}$.
    \item We say $\Model{F}$ is \emph{$\Diamond$-condensed} if $(\IRel^{-1} \cdot \MRel) \subseteq {\MRel}$,
          and is \emph{$\Diamond$-brilliant} (or \emph{strictly $\Diamond$-condensed}) if $(\MRel \cdot \IRel^{-1}) \subseteq {\MRel}$.
  \end{itemize}
\end{definition}

\begin{figure}[H]
  \centering
  \begin{subfigure}{.225\textwidth}
    \centering
    \begin{tikzpicture}
      \fill (0, 0) circle (2pt);
      \draw (0, 0) node[left]{$x$};

      \fill (0, 2) circle (2pt);
      \draw (0, 2) node[left]{$x'$};

      \fill (2, 2) circle (2pt);
      \draw (2, 2) node[right]{$y$};

      \draw [thick, ->, rel] (0,0) to (0,2);
      \draw (0,1) node[left]{$\IRel$};
      \draw [thick, ->, rel, squig] (0,2) to (2,2);
      \draw (1,1) node[below right]{$\MRel$};
      \draw [thick, ->, rel, squig, dashed] (0,0) to (2,2);
      \draw (1,2) node[above]{$\MRel$};

      \draw (1,0) node[below]{\vphantom{$\MRel$}};
    \end{tikzpicture}
    \caption*{$\Box$-condensed}
  \end{subfigure}%
  \begin{subfigure}{.25\textwidth}
    \centering
    \begin{tikzpicture}
      \fill (0, 0) circle (2pt);
      \draw (0, 0) node[left]{$x$};

      \fill (2, 0) circle (2pt);
      \draw (2, 0) node[right]{$y$};

      \fill (2, 2) circle (2pt);
      \draw (2, 2) node[right]{$y'$};

      \draw [thick, ->, rel, squig] (0,0) to (2,0);
      \draw (1,0) node[below]{$\MRel$};
      \draw [thick, ->, rel] (2,0) to (2,2);
      \draw (2,1) node[right]{$\IRel$};
      \draw [thick, ->, rel, squig, dashed] (0,0) to (2,2);
      \draw (1,1) node[above left]{$\MRel$};
    \end{tikzpicture}
    \caption*{$\Box$-brilliant}
  \end{subfigure}%
  \begin{subfigure}{.25\textwidth}
    \centering
    \begin{tikzpicture}
      \fill (0, 0) circle (2pt);
      \draw (0, 0) node[left]{$x$};

      \fill (0, 2) circle (2pt);
      \draw (0, 2) node[left]{$x'$};

      \fill (2, 0) circle (2pt);
      \draw (2, 0) node[right]{$y$};

      \draw [thick, ->, rel] (0,0) to (0,2);
      \draw (0,1) node[left]{$\IRel$};
      \draw [thick, ->, rel, squig, dashed] (0,2) to (2,0);
      \draw (1,1) node[above]{$\MRel$};
      \draw [thick, ->, rel, squig] (0,0) to (2,0);
      \draw (1,0) node[below]{$\MRel$};
    \end{tikzpicture}
    \caption*{$\Diamond$-condensed}
  \end{subfigure}%
  \begin{subfigure}{.225\textwidth}
    \centering
    \begin{tikzpicture}
      \fill (0, 2) circle (2pt);
      \draw (0, 2) node[left]{$x'$};

      \fill (2, 0) circle (2pt);
      \draw (2, 0) node[right]{$y$};

      \fill (2, 2) circle (2pt);
      \draw (2, 2) node[right]{$y'$};

      \draw [thick, ->, rel, squig, dashed] (0,2) to (2,0);
      \draw (1,1) node[below left]{$\MRel$};
      \draw [thick, ->, rel] (2,0) to (2,2);
      \draw (2,1) node[right]{$\IRel$};
      \draw [thick, ->, rel, squig] (0,2) to (2,2);
      \draw (1,2) node[above]{$\MRel$};

      \draw (1,0) node[below]{\vphantom{$\MRel$}};
    \end{tikzpicture}
    \caption*{$\Diamond$-brilliant}
  \end{subfigure}%
  \caption{Visualization of Definition \ref{def:condensed-brilliant}}
  \label{fig:condensed-brilliant}
\end{figure}

\begin{proposition} \leavevmode
  \begin{itemize}
    \item A $\Box$-condensed f-frame satisfies the $\Box$-p condition.
    \item A $\Box$-p and $\Box$-brilliant f-frame is $\Box$-condensed.
    \item A $\Diamond$-condensed f-frame satisfies the $\Diamond$-p condition.
    \item A $\Diamond$-p and $\Diamond$-brilliant f-frame is $\Diamond$-condensed.
  \end{itemize}
  \begin{proof}
    Trivial.
  \end{proof}
\end{proposition}

\begin{definition}[{Condensation, \cite{bozic_models_1984}}] \label{def:condensation} \leavevmode
  \begin{itemize}
    \item For a $\Box$-p f-frame $\Model{F} = (W, \IRel, \MRel)$,
          we define a new f-frame $\beta \Model{F} = (W, \IRel, \MRel_\beta)$,
          the \emph{$\Box$-condensation} of $\Model{F}$,
          by letting ${\MRel_\beta} = (\MRel \cdot \IRel)$.
          For a $\Box$-p f-model $\Model{M}$, we also define $\beta\Model{M}$ by $\Box$-condensing its frame.
    \item For a $\Diamond$-p f-frame $\Model{F} = (W, \IRel, \MRel)$,
          we define a new f-frame $\delta \Model{F} = (W, \IRel, \MRel_\delta)$,
          the \emph{$\Diamond$-condensation} of $\Model{F}$,
          by letting ${\MRel_\delta} = (\MRel \cdot \IRel^{-1})$.
          For a $\Diamond$-p f-model $\Model{M}$, we also define $\delta\Model{M}$ by $\Diamond$-condensing its frame.
  \end{itemize}
\end{definition}

\begin{proposition} \leavevmode
  \begin{enumerate}[label=(\arabic*)]
    \item For every $\Box$-p f-frame $\Model{F}$, $\beta \Model{F}$ is $\Box$-condensed and $\Box$-brilliant.
    \item For every $\Diamond$-p f-frame $\Model{F}$, $\delta \Model{F}$ is $\Diamond$-condensed and $\Diamond$-brilliant.
  \end{enumerate}
  \begin{proof}
    (1)
    ($\Box$-condensed)
    Take any $x \IRel x' \MRel_\beta y''$, then there is $y'$ such that $(x \IRel) x' \MRel y' \IRel y''$.
    By $\Box$-p of $\Model{F}$, there is $y$ such that $x \MRel y \IRel y' (\IRel y'')$,
    then $x \MRel y \IRel y''$, so $x \MRel_\beta y''$.
    ($\Box$-brilliant)
    Take any $x \MRel_\beta y' \IRel y''$, then there is $y$ such that $x \MRel y \IRel y' (\IRel y'')$.
    then $x \MRel y \IRel y''$, so $x \MRel_\beta y''$.

    (2)
    ($\Diamond$-condensed)
    Take any $x'' \IRelRev x \MRel_\delta y$, then there is $y'$ such that $x \MRel y' \IRelRev y$.
    By $\Diamond$-p of $\Model{F}$, there is $y''$ such that $x'' \MRel y''$ and $y'' \IRelRev y' (\IRelRev y)$,
    then $x'' \MRel y'' \IRelRev y$, so $x'' \MRel_\delta y$.
    ($\Diamond$-brilliant)
    Take any $x' \MRel_\delta y' \IRelRev y$, then there is $y''$ such that $x' \MRel y'' \IRelRev y' (\IRelRev y)$,
    then $x' \MRel y'' \IRelRev y$, so $x' \MRel_\delta y$.
  \end{proof}
\end{proposition}

\begin{proposition}[{\cite[Lems.\,4 and 18]{bozic_models_1984}}] \label{prop:condensation} \leavevmode
  \begin{enumerate}[label=(\arabic*)]
    \item Let $\Model{M} = (W, \IRel, \MRel, V)$ be a $\Box$-p f-model, then
          $x \SatE_\Model{M} \varphi \iff x \SatE_{\beta\Model{M}} \varphi$
          for any $\varphi \in \MFb$ and $x \in W$.
    \item Let $\Model{M} = (W, \IRel, \MRel, V)$ be a $\Diamond$-p f-model, then
          $x \SatE_\Model{M} \varphi \iff x \SatE_{\delta\Model{M}} \varphi$
          for any $\varphi \in \MFd$ and $x \in W$.
  \end{enumerate}
  \begin{proof}
    We first make the following claims.
    \begin{claim*}
      For any $w \in W$ and $\rho \in \MFb$, $w \SatE_\Model{M} \rho \iff \forall w' \IRelRev w\,(w' \SatE_\Model{M} \rho)$.
      \begin{subproof}[Proof of claim]
        ($\Leftarrow$) Trivial.
        ($\Rightarrow$) Follows from $\Box$-persistency.
      \end{subproof}
    \end{claim*}
    \begin{claim*}
      For any $w' \in W$ and $\rho \in \MFd$, $w' \SatE_\Model{M} \rho \iff \exists w \IRel w'\,(w \SatE_\Model{M} \rho)$.
      \begin{subproof}[Proof of claim]
        ($\Leftarrow$) Follows from $\Diamond$-persistency.
        ($\Rightarrow$) Trivial.
      \end{subproof}
    \end{claim*}
    Now we prove (1) and (2) separately.

    (1) Induction on the construction of $\varphi$. It suffices to consider the $\Box$ case.
    \begin{align*}
      x \SatE_\Model{M} \Box \psi &\iff \forall y \MRelRev x\, (y \SatE_\Model{M} \psi) & \\
                                  &\iff \forall y \MRelRev x\, \forall y' \IRelRev y\, (y' \SatE_\Model{M} \psi) & (\because \text{claim}) \\
                                  &\iff \forall y \MRelRev x\, \forall y' \IRelRev y\, (y' \SatE_{\beta \Model{M}} \psi) & (\because \text{induction hypothesis})\\
                                  &\iff \forall y' \MRelRev_\beta x\, (y' \SatE_{\beta \Model{M}} \psi) & (\because x \MRel y \IRel y') \\
                                  &\iff x \SatE_{\beta \Model{M}} \Box \psi. &
    \end{align*}

    (2) Induction on the construction of $\varphi$. It suffices to consider the $\Diamond$ case:
    \begin{align*}
      x \SatE_\Model{M} \Diamond \psi &\iff \exists y' \MRelRev x\, (y' \SatE_\Model{M} \psi) & \\
                                      &\iff \exists y' \MRelRev x\, \exists y \IRel y'\, (y \SatE_\Model{M} \psi) & (\because \text{claim}) \\
                                      &\iff \exists y' \MRelRev x\, \exists y \IRel y'\, (y \SatE_{\delta\Model{M}} \psi) & (\because \text{induction hypothesis}) \\
                                      &\iff \exists y \MRelRev_\delta x\, (y \SatE_{\delta\Model{M}} \psi) & (\because x \MRel y' \IRelRev y) \\
                                      &\iff x \SatE_{\delta \Model{M}} \Diamond \psi. & \qedhere
    \end{align*}%
  \end{proof}
\end{proposition}

\begin{corollary}
  (1) For a $\Box$-p f-frame $\Model{F}$ and $\varphi \in \MFb$, $\Model{F} \ValidE \varphi$ iff $\beta \Model{F} \ValidE \varphi$.\\
  (2) For a $\Diamond$-p f-frame $\Model{F}$ and $\varphi \in \MFd$, $\Model{F} \ValidE \varphi$ iff $\delta \Model{F} \ValidE \varphi$.
\end{corollary}

It should be noted that, while the $\Box$/$\Diamond$-condensation operations, respectively,
do preserve the validity of $\Box$/$\Diamond$-only formulas,
they do \emph{not} preserve the validity of formulas containing $\Diamond$/$\Box$.
In fact, the $\Box$/$\Diamond$-condensedness conditions would validate weak forms of the excluded middle involving the other modality:

\begin{proposition} \label{prop:condensedness-is-dangerous}
  (1) $\Box \varphi \lor \neg \Box \varphi$ is valid on every $\Diamond$-condensed f-frame.\\
  (2) $\Diamond \varphi \lor \neg \Diamond \varphi$ is valid on every $\Box$-condensed f-frame.
  \begin{proof}
    Exercise.
  \end{proof}
\end{proposition}

This is why condensation only makes sense in the $\Box$/$\Diamond$-free context;
otherwise, the resulting logic with both $\Box$ and $\Diamond$ as its modalities
would instantly lose DP, making it not quite intuitionistic.

Moving forward, using condensation, we obtain the following completeness results of $\Logic{iK}_\Box$ and $\Logic{iK}_\Diamond$
with respect to condensed and brilliant f-frames.

\begin{theorem}[{\cite[Thms.\,2 and 5]{bozic_models_1984}}] \label{cor:condensed-completeness}
  Let
  $\FrameClass{\Box\text{-con}}$, $\FrameClass{\Box\text{-bri}}$,
  $\FrameClass{\Diamond\text{-con}}$, and $\FrameClass{\Diamond\text{-bri}}$
  be the classes of all $\Box$-condensed f-frames, $\Box$-brilliant f-frames, $\Diamond$-condensed f-frames, and $\Diamond$-brilliant f-frames, respectively.
  Then,
  \begin{align*}
    \Logic{iK}_\Box &= \set{ \varphi \in \MFb; \forall \Model{F} \in \FrameClass{\Box\text{-con}} \cap \FrameClass{\Box\text{-bri}}\,\, (\Model{F} \ValidE \varphi) }, \\
    \Logic{iK}_\Diamond &= \set{ \varphi \in \MFd; \forall \Model{F} \in \FrameClass{\Diamond\text{-con}} \cap \FrameClass{\Diamond\text{-bri}}\,\, (\Model{F} \ValidE \varphi) }.
  \end{align*}
  Consequently,
  as $\FrameClass{\Box{\text{-p}}} \supseteq \FrameClass{\Box\text{-con}} \supseteq \FrameClass{\Box\text{-con}} \cap \FrameClass{\Box\text{-bri}}$,
  we have $\Logic{iK}_\Box = \set{ \varphi \in \MFb; \forall \Model{F} \in \FrameClass{\Box\text{-con}}\,\, (\Model{F} \ValidE \varphi) }$,
  and
  as $\FrameClass{\Diamond{\text{-p}}} \supseteq \FrameClass{\Diamond\text{-con}} \supseteq \FrameClass{\Diamond\text{-con}} \cap \FrameClass{\Diamond\text{-bri}}$,
  we have $\Logic{iK}_\Diamond = \set{ \varphi \in \MFd; \forall \Model{F} \in \FrameClass{\Diamond\text{-con}}\,\, (\Model{F} \ValidE \varphi) }$.
  \begin{proof}
    For the convenience, let $\heartsuit \in \set{\Box, \Diamond}$.
    ($\subseteq$)
    Trivial from Theorem \ref{thm:ikb-ikd-completeness} and $\FrameClass{\heartsuit\text{-p}} \supseteq \FrameClass{\heartsuit\text{-con}} \cap \FrameClass{\heartsuit\text{-bri}}$.
    ($\supseteq$)
    If $\Logic{iK}_\heartsuit \nvdash \varphi$,
    then by Theorem \ref{thm:ikb-ikd-completeness}, we have a $\heartsuit$-p f-model that falsifies $\varphi$.
    We condense it to obtain a $\heartsuit$-condensed and $\heartsuit$-brilliant f-model,
    which also falsifies $\varphi$ by Proposition \ref{prop:condensation}.
  \end{proof}
\end{theorem}

$\Logic{iK}_\Box$ and $\Logic{iK}_\Diamond$ are indeed intuitionistic according to Definition \ref{def:iml}.

\begin{proposition}
  $\Logic{iK}_\Box + \varphi \lor \neg \varphi = \Logic{K}_\Box$
  and $\Logic{iK}_\Diamond + \varphi \lor \neg \varphi = \Logic{K}_\Diamond$.
  \begin{proof}
    Trivial from the axiomatizations of the logics.
  \end{proof}
\end{proposition}

\begin{proposition}[{\cite[Lems.\,1 and 15]{bozic_models_1984}}]
  Both $\Logic{iK}_\Box$ and $\Logic{iK}_\Diamond$ enjoy DP.
\end{proposition}

A cut-admissible sequent calculus for $\Logic{iK}_\Box$ is obtained from the intuitionistic propositional calculus $\Logic{LJ}$ by
adding a rule $\frac{\Gamma \imp \varphi}{\Box \Gamma \imp \Box \varphi}$ (see Appendix \ref{appendix:lip}).
G3 and G4-style sequent calculi for $\Logic{iK}_\Box$ are also developed by Iemhoff \cite{iemhoff_uniform_2019},
establishing CIP and UIP of it. Later in Appendix \ref{appendix:lip}, we will show that $\Logic{iK}_\Box$ also enjoys LIP.
As far as the author is aware, not much is known about $\Logic{iK}_\Diamond$ in terms of sequent calculi and interpolation properties.

\begin{proposition}[{\cite[Thm.\,34]{iemhoff_uniform_2019}}]
  $\Logic{iK}_\Box$ enjoys UIP, and, consequently, CIP.
\end{proposition}

\begin{proposition}
  $\Logic{iK}_\Box$ enjoys LIP.
  \begin{proof}
    See Corollary \ref{cor:ikb-enjoys-lip}.
  \end{proof}
\end{proposition}

\begin{problem}
  Does $\Logic{iK}_\Diamond$ enjoy CIP, LIP, and UIP?
\end{problem}

\section{Constructive modal logics} \label{sec:cml}

\emph{Constructive modal logics (CMLs)}, which do not admit the distribution of $\Diamond$ over $\lor$,
were originally introduced by Wijesekera \cite{wijesekera_constructive_1990}
and later developed further by Mendler and de Paiva \cite{mendler_constructive_2005}.

\begin{table}[H]
  \centering
  \caption{Basic modal axioms for CMLs.}
  \label{tab:basic-cml-axioms}
  \begin{tabular}{rl|rl}
    $\name{M}_\Box$     & \AxiomC{$\varphi \to \psi$} \UnaryInfC{$\Box \varphi \to \Box \psi$} \DisplayProof &
    $\name{M}_\Diamond$ & \AxiomC{$\varphi \to \psi$} \UnaryInfC{$\Diamond \varphi \to \Diamond \psi$} \DisplayProof \\
    $\name{N}_\Box$      & $\Box \top$ &
    $\name{N}_\Diamond$ & $\neg \Diamond \bot$ \\
    $\name{C}_\Box$     & $(\Box \varphi \land \Box \psi) \to \Box (\varphi \land \psi)$ & & \\
    $\name{K}_\Diamond$ & $\Box(\varphi \to \psi) \to (\Diamond \varphi \to \Diamond \psi)$ & & \\
    $\name{FS1}$ & $\Diamond (\varphi \to \psi) \to (\Box \varphi \to \Diamond \psi)$ & & \\
  \end{tabular}
\end{table}

\begin{definition}
  Let $\Logic{CK} \defeq \Logic{iK}_\Box + \name{K}_\Diamond$ and $\Logic{WK} \defeq \Logic{CK} + \name{N}_\Diamond$
  (see Table \ref{tab:basic-cml-axioms} for the definitions).
\end{definition}

The logic $\Logic{CK}$ has an alternative axiomatization $\Logic{CK}' \defeq \Logic{iK}_\Box + \name{M}_\Diamond + \name{FS1}$.

\begin{proposition} \label{prop:ck-alternative-axiomatization}
  $\Logic{CK} \dashv \vdash \Logic{CK}'$, that is,
  (1) $\Logic{CK} \vdash \name{FS1}$,
  (2) $\name{M}_\Diamond$ is admissible in $\Logic{CK}$, and
  (3) $\Logic{CK}' \vdash \name{K}_\Diamond$.
  \begin{proof}
    The following derivations prove (1), (2), and (3), respectively.

    \begin{center}
      \AxiomC{$ $}
      \UnaryInfC{$\varphi \to ((\varphi \to \psi) \to \psi)$}
      \RightLabel{\scriptsize{$\name{M}_\Box$}}
      \UnaryInfC{$\Box \varphi \to \Box((\varphi \to \psi) \to \psi)$}

      \AxiomC{$ $}
      \RightLabel{\scriptsize{$\name{K}_\Diamond$}}
      \UnaryInfC{$\Box((\varphi \to \psi) \to \psi) \to (\Diamond(\varphi \to \psi) \to \Diamond \psi)$}

      \RightLabel{\scriptsize{Syll.}}
      \BinaryInfC{$\Box \varphi \to (\Diamond(\varphi \to \psi) \to \Diamond \psi)$}
      \UnaryInfC{$\Diamond (\varphi \to \psi) \to (\Box \varphi \to \Diamond \psi)$}
      \DisplayProof
    \end{center}

    \begin{center}
      \AxiomC{$ $}
      \RightLabel{\scriptsize{$\name{N}_\Box$}}
      \UnaryInfC{$\Box \top$}

      \AxiomC{$\varphi \to \psi$}
      \UnaryInfC{$\top \to (\varphi \to \psi)$}
      \RightLabel{\scriptsize{$\name{M}_\Box$}}
      \UnaryInfC{$\Box \top \to \Box (\varphi \to \psi)$}

      \RightLabel{\scriptsize{MP}}
      \BinaryInfC{$\Box (\varphi \to \psi)$}

      \AxiomC{$ $}
      \RightLabel{\scriptsize{$\name{K}_\Diamond$}}
      \UnaryInfC{$\Box (\varphi \to \psi) \to (\Diamond \varphi \to \Diamond \psi)$}

      \RightLabel{\scriptsize{MP}}
      \BinaryInfC{$\Diamond \varphi \to \Diamond \psi$}
      \DisplayProof
    \end{center}

    \begin{center}
      \AxiomC{$ $}
      \UnaryInfC{$\varphi \to ((\varphi \to \psi) \to \psi)$}
      \RightLabel{\scriptsize{$\name{M}_\Diamond$}}
      \UnaryInfC{$\Diamond \varphi \to \Diamond((\varphi \to \psi) \to \psi)$}

      \AxiomC{$ $}
      \RightLabel{\scriptsize{$\name{FS1}$}}
      \UnaryInfC{$\Diamond((\varphi \to \psi) \to \psi) \to (\Box(\varphi \to \psi) \to \Diamond \psi)$}

      \RightLabel{\scriptsize{Syll.}}
      \BinaryInfC{$\Diamond \varphi \to (\Box(\varphi \to \psi) \to \Diamond \psi)$}
      \UnaryInfC{$\Box (\varphi \to \psi) \to (\Diamond \varphi \to \Diamond \psi)$}
      \DisplayProof
      \qedhere
    \end{center}
  \end{proof}
\end{proposition}

Fischer Servi first introduced the axiom $\name{FS1}$ as a modest bridge between $\Box$ and $\Diamond$,
aiming to obtain what she regarded as the ``true'' intuitionistic analogue of $\Logic{K}$.
Although her original motivation lay in the development of what we would now call an IML,
the same axiom continues to perform its intended role on the constructive side:
with the presence of $\name{N}_\Diamond$, $\name{FS1}$ behaves just as well in CMLs.

\begin{proposition} \label{prop:wk-partial-duality}
  (1) $\Logic{WK} \vdash \Box \neg \varphi \to \neg \Diamond \varphi$
  and (2) $\Logic{WK} \vdash \Diamond \neg \varphi \to \neg \Box \varphi$.
  \begin{proof}
    The following derivations prove (1) and (2), respectively.

    \begin{center}
      \AxiomC{$ $}
      \RightLabel{\scriptsize{$\name{K}_\Diamond$}}
      \UnaryInfC{$\Box \neg \varphi \to (\Diamond \varphi \to \Diamond \bot)$}

      \AxiomC{$ $}
      \RightLabel{\scriptsize{$\name{N}_\Diamond$}}
      \UnaryInfC{$\Diamond \bot \to \bot$}

      \RightLabel{\scriptsize{Syll.}}
      \BinaryInfC{$\Box \neg \varphi \to \neg \Diamond \varphi$}
      \DisplayProof
    \end{center}
    \smallskip

    \begin{center}
      \AxiomC{$ $}
      \RightLabel{\scriptsize{$\name{FS1}$}}
      \UnaryInfC{$\Diamond \neg \varphi \to (\Box \varphi \to \Diamond \bot)$}

      \AxiomC{$ $}
      \RightLabel{\scriptsize{$\name{N}_\Diamond$}}
      \UnaryInfC{$\Diamond \bot \to \bot$}

      \RightLabel{\scriptsize{Syll.}}
      \BinaryInfC{$\Diamond \neg \varphi \to \neg \Box \varphi$}
      \DisplayProof
    \end{center}
  \end{proof}
\end{proposition}

With Propositions \ref{prop:ck-alternative-axiomatization} and \ref{prop:wk-partial-duality} in mind, the author suggests that it would make more sense to rather consider $\name{K}_\Diamond$ as the derived axiom from $\name{M}_\Diamond$ and $\name{FS1}$,
one for the monotonicity of $\Diamond$ and another for the weak interaction between $\Box$ and $\Diamond$.
Now we shall present the semantics for CMLs.

\begin{definition}
  We say $\Model{F} = (W, W_\bot, \IRel, \MRel)$ is a \emph{CK-frame} if $(W, \IRel, \MRel)$ is an f-frame and $W_\bot \subseteq W$,
  where $W_\bot$ is upward closed with respect to both $\IRel$ and $\MRel$ (i.e. $\forall x,y\, (x \in W_\bot \tand (x \IRel y \tor x \MRel y) \imp y \in W_\bot)$)
  and is $\MRel$-serial (i.e. $\forall x\, (x \in W_\bot \imp \exists y\, (y \in W_\bot \tand x \MRel y))$).
  Each $x \in W_\bot$ is called a \emph{fallible} world.
\end{definition}

In the original paper \cite{mendler_constructive_2005}, Mendler and de Paiva did not require $\MRel$-seriality on $W_\bot$,
which is problematic. We will discuss it later in Remark \ref{rem:original-paper-problem}.

\begin{definition}
  We say $\Model{M} = (\Model{F}, V)$ is a \emph{CK-model} if $\Model{F}$ is a CK-frame and $V: \PropVar \to \mathcal{P}(W)$,
  where each $V(p)$ is upward closed with respect to $\IRel$, and $x \in V(p)$ for every $x \in W_\bot$ and $p \in \PropVar$.
  Such $\Model{M}$ is said to be \emph{based on} $\Model{F}$.
\end{definition}

\begin{definition}
  For a CK-model $\Model{M} = (W, W_\bot, \IRel, \MRel, V)$, we define a \emph{constructive satisfaction relation} $\SatCK_\Model{M}$ on $W \times \MF$ by:
  \begin{enumerate}
    \item $x \SatCK_\Model{M} \bot \defiff x \in W_\bot$
    \item $x \SatCK_\Model{M} \psi_1 \to \psi_2 \defiff \forall x' \IRelRev x\, (x' \SatCK_\Model{M} \psi_1 \imp x' \SatCK_\Model{M} \psi_2)$
    \item $x \SatCK_\Model{M} \Box \varphi     \defiff \forall x' \IRelRev x\, \forall y' \MRelRev x'\, (y' \SatCK_\Model{M} \varphi)$
    \item $x \SatCK_\Model{M} \Diamond \varphi \defiff \forall x' \IRelRev x\, \exists y' \MRelRev x'\, (y' \SatCK_\Model{M} \varphi)$
    \item The other cases are identical to the classical $\Vdash_\Model{M}$.
  \end{enumerate}
  We write $\Model{M} \ValidCK \varphi$ (``$\varphi$ is \emph{(c-)valid in} $\Model{M}$'') to mean $\forall x \in W\,(x \SatCK_\Model{M} \varphi)$,
  and $\Model{F} \ValidCK \varphi$ (``$\varphi$ is \emph{(c-)valid on} $\Model{F}$'') to mean that $\Model{M} \ValidCK \varphi$ for every $\Model{M}$ based on $\Model{F}$.
\end{definition}

Here, we can see that $\SatCK$ is a variant of the intrinsic satisfaction relation $\SatI$.
Consequently, $\name{C}_\Diamond$ need not be c-valid without any further frame conditions.
In addition to that, CK frames also have special worlds, fallible worlds, at which $\bot$ is true.
With the presence of fallible worlds, $\name{N}_\Diamond$ also fails to be valid in a CK-frame
if there are $w \in W \setminus W_\bot$ and $w' \in W_\bot$ such that $w \MRel w'$.

\begin{proposition}[All is fair in a fallible world] \label{prop:all-is-fair}
  Let $\Model{M} = (W, W_\bot, \IRel, \MRel, V)$ be a CK-model. Then $x \SatCK_\Model{M} \varphi$ for every $x \in W_\bot$ and $\varphi \in \MF$.
  \begin{proof}
    We use an induction on the construction of $\varphi$.

    It is trivial when $\varphi = \bot$ or $\varphi = p$ for some $p \in \PropVar$.

    Suppose $\varphi = \psi_1 \circledcirc \psi_2$, for $\circledcirc \in \set{\land, \lor}$, then it is easily proven with the induction hypothesis.

    Suppose $\varphi = \psi_1 \to \psi_2$. Take any $x'$ such that $x \IRel x' \SatCK_\Model{M} \psi_1$, then $x' \in W_\bot$ by $W_\bot$ being upward closed with respect to $\IRel$,
    then $x' \SatCK_\Model{M} \psi_2$ by the induction hypothesis. Therefore, $x \SatCK_\Model{M} \psi_1 \to \psi_2$ holds.

    Suppose $\varphi = \Box \psi$. Take any $x', y'$ such that $x \IRel x' \MRel y'$, then $x', y' \in W_\bot$ by $W_\bot$ being upward closed with respect to $\IRel$ and $\MRel$,
    then $y' \SatCK_\Model{M} \psi$ by the induction hypothesis. Therefore, $x \SatCK_\Model{M} \Box \psi$.

    Suppose $\varphi = \Diamond \psi$. Take any $x'$ such that $x \IRel x'$, then $x' \in W_\bot$ by $W_\bot$ being upward closed with respect to $\IRel$.
    By $W_\bot$ being $\MRel$-serial, there is $y'$ such that $x' \MRel y'$,
    then $y' \SatCK_\Model{M} \psi$ by the induction hypothesis. Therefore, $x \SatCK_\Model{M} \Diamond \psi$.
  \end{proof}
\end{proposition}

\begin{remark} \label{rem:original-paper-problem}
  We note that the $\MRel$-seriality on $W_\bot$ is crucial to prove the $\Diamond$ case of Proposition \ref{prop:all-is-fair}.
  It is also crucial for all the intuitionistically valid principles to be c-valid on CK-frames;
  the formula $\bot \to \Diamond p$, which is an instance of the axiom $\bot \to \varphi$ in $\Int$,
  would not be c-valid on a frame $(W, \set{ x }, \IRel, \MRel)$ where $x \not \MRel w$ for any $w \in W$.
  This means that the original semantics by Mendler and de Paiva \cite{mendler_constructive_2005}, which lacks the $\MRel$-seriality on $W_\bot$, is not sound.
\end{remark}

\begin{proposition}[Persistency] \label{prop:constructive-persistency}
  Let $\Model{M} = (W, W_\bot, \IRel, \MRel, V)$ be a CK-model, and take any $\varphi \in \MF$.
  Then for any $x, x' \in W$, $x \SatCK_\Model{M} \varphi$ and $x \IRel x'$ implies $x' \SatCK_\Model{M} \varphi$.
  \begin{proof}
    Suppose that $x' \in W_\bot$, then it trivially holds since $x' \SatCK_\Model{M} \varphi$ by Proposition \ref{prop:all-is-fair}.
    Now suppose that $x' \notin W_\bot$, then $x \notin W_\bot$ by $W_\bot$ being upward closed with respect to $\IRel$.
    The rest is proven similarly to Proposition \ref{prop:intrinsic-persistency}.
  \end{proof}
\end{proposition}

\begin{definition}
  Let us say a CK-frame $(W, W_\bot, \IRel, \MRel)$ is \emph{infallible} if $W_\bot = \emptyset$.
\end{definition}

\begin{theorem}[{\cite[Thm.\,1]{mendler_constructive_2005},  \cite[Thm.\,1.4.5]{wijesekera_constructive_1990}, \cite[Thm.\,IV.9]{groot_semantical_2025}}] \label{thm:ck-wk-completeness} \leavevmode
  \begin{itemize}
    \item $\Logic{CK} = \set{ \varphi \in \MF; \forall \Model{F}: \text{CK-frame}\ (\Model{F} \ValidCK \varphi) }$.
    \item $\Logic{WK} = \set{ \varphi \in \MF; \forall \Model{F}: \text{infallible CK-frame}\ (\Model{F} \ValidCK \varphi) }$.
  \end{itemize}
\end{theorem}

We mentioned in Remark \ref{rem:original-paper-problem} that the original completeness proof for $\Logic{CK}$ by Mendler and Paiva \cite{mendler_constructive_2005} contains a flaw in the soundness direction ($\subseteq$),
which can easily be fixed by requiring $\MRel$-seriality on $W_\bot$.
The completeness direction ($\supseteq$) is unaffected, however, because their canonical model for $\Logic{CK}$ has only one fallible world that is $\MRel$-reflexive, and so it is indeed a CK-model in our sense.

\begin{corollary} \label{cor:wk-completeness}
  $\Logic{WK} = \Theory{V}^i$.
  \begin{proof}
    It is easy to see that an infallible CK-frame is just an f-frame,
    and the constructive satisfaction relation $\SatCK$ for infallible CK-models is just the intrinsic satisfaction relation $\SatI$ for f-models.
  \end{proof}
\end{corollary}

\begin{corollary} \label{cor:wk-ck-restriction}
  (1) $\Logic{WK} \cap \MFb = \Logic{iK}_\Box$ and (2) $\Logic{CK} \cap \MFb = \Logic{iK}_\Box$.
  \begin{proof}
    (1) follows from Corollary \ref{cor:wk-completeness}, Proposition \ref{prop:diafree-vi-is-veb}, and Theorem \ref{thm:ikb-ikd-completeness}.
    (2 $\supseteq$) is obvious. (2 $\subseteq$) follows from (1) and $\Logic{CK} \subseteq \Logic{WK}$.
  \end{proof}
\end{corollary}

\begin{proposition} \label{prop:wk-ck-dp}
  Both $\Logic{WK}$ and $\Logic{CK}$ enjoy DP.
  \begin{proof}
    We shall only prove the case when $L = \Logic{CK}$, as the other case is proven similarly.
    Suppose that $\Logic{CK} \nvdash \varphi_1$ and $\Logic{CK} \nvdash \varphi_2$,
    then there are $\Model{M}_1 = (W_1, W_{\bot1}, \IRel_1, \MRel_1, V_1)$, $\Model{M}_2 = (W_2, W_{\bot2}, \IRel_2, \MRel_2, V_2)$,
    $w_1 \in W_1$, and $w_2 \in W_2$ such that $w_1 \not\SatCK_{\Model{M}_1} \varphi_1$ and $w_2 \not\SatCK_{\Model{M}_2} \varphi_2$.
    By replacing the two models with isomorphic copies if necessary, we may assume that $W_1 \cap W_2 = \emptyset$.
    We construct a new model $\Model{M} = (W, W_\bot, \IRel, \MRel, V)$ by the following:
    \begin{itemize}
      \item $W \defeq W_1 \cup W_2 \cup \set{w_0}$, where $w_0$ is fresh;
      \item $W_\bot = W_{\bot1} \cup W_{\bot2}$;
      \item $\IRel$ is the reflexive and transitive closure of $\set{ (w_0, w_1), (w_0, w_2) } \cup \mathord{\IRel_1} \cup \mathord{\IRel_2}$;
      \item $\mathord{\MRel} \defeq \mathord{\MRel_1} \cup \mathord{\MRel_2}$;
      \item $V(p) \defeq V_1(p) \cup V_2(p) \cup V_0(p)$, where $V_0(p) \defeq \set{w_0}$ if $w_1 \in V_1(p)$ and $w_2 \in V_2(p)$, and $V_0(p) \defeq \emptyset$ otherwise.
    \end{itemize}
    Here, $\Model{M}$ is clearly a CK-model since each $V(p)$ is upward closed.
    Also, it is easy to see that for any $i \in \set{1,2}$, $x \in W_i$, and $\psi \in \MF$, $x \SatCK_\Model{M} \psi$ iff $x \SatCK_{\Model{M}_i} \psi$.
    Then $w_0 \not\SatCK_\Model{M} \varphi_1$ and $w_0 \not\SatCK_\Model{M} \varphi_2$ by persistency,
    so $w_0 \not\SatCK_\Model{M} \varphi_1 \lor \varphi_2$. Therefore, $\Logic{CK} \nvdash \varphi_1 \lor \varphi_2$.
    The other direction is trivial.
  \end{proof}
\end{proposition}

Interestingly, neither CML collapses into $\K$ just by adding $\varphi \lor \neg \varphi$.

\begin{proposition}[{\cite{simpson_proof_1994}}]
  $\Logic{WK} + \varphi \lor \neg \varphi \subsetneq \Logic{K}$. Consequently, $\Logic{CK} + \varphi \lor \neg \varphi \subsetneq \Logic{K}$.
  \begin{proof}
    It suffices to prove it for $\Logic{WK}$.
    ($\subseteq$) Trivial from the axiomatizations.
    ($\neq$) It suffices to construct an infallible CK-model in which $\varphi \lor \neg \varphi$ is valid for every $\varphi \in \MF$ but $\neg \Box \neg p \to \Diamond p$ is not.
    We let $\Model{M} = (W, \IRel, \MRel, V)$, where:
    \begin{itemize}
      \item $W \defeq \set{ x_i, y_i; i \in \mathbb{N} }$;
      \item $\IRel$ is the reflexive and transitive closure of $\set*{ (x_i, x_{i+1}), (y_i, y_{i+1}); i \in \mathbb{N} }$;
      \item $\mathord{\MRel} \defeq \set*{ (x_{2i}, y_{2i}), (y_{2i}, x_{2i}), (x_{2i+1}, x_{2i+1}), (y_{2i+1}, y_{2i+1}); i \in \mathbb{N} }$;
      \item $V(p) \defeq \set{ y_i; i \in \mathbb{N} }$.
    \end{itemize}
    The model is due to Simpson \cite[pp.48--49]{simpson_proof_1994} and is visualized in Figure \ref{fig:wk-em-dual-countermodel}.
    \begin{figure}[ht]
      \centering
      \begin{tikzpicture}
      \fill (0,2) circle (2pt);
      \draw (0,2) node[left]{$x_0$};

      \fill (2,2) circle (2pt);
      \draw (2,2) node[right]{$y_0 \Vdash p$};

      \draw [thick, <->, squig, rel] (0,2) to (2,2);
      \draw (1,2) node[above]{\tiny $\MRel$};

      \draw (0,3) circle (2pt);
      \draw (0,3) node[left]{$x_1$};
      \draw [thick, ->, rel] (0,2) to (0,3);
      \draw (0,2.5) node[left]{\tiny $\IRel$};

      \draw (2,3) circle (2pt);
      \draw (2,3) node[right]{$y_1 \Vdash p$};
      \draw [thick, ->, rel] (2,2) to (2,3);
      \draw (2,2.5) node[right]{\tiny $\IRel$};

      \fill (0,4) circle (2pt);
      \draw (0,4) node[left]{$x_2$};
      \draw [thick, ->, rel] (0,3) to (0,4);
      \draw (0,3.5) node[left]{\tiny $\IRel$};

      \fill (2,4) circle (2pt);
      \draw (2,4) node[right]{$y_2 \Vdash p$};
      \draw [thick, ->, rel] (2,3) to (2,4);
      \draw (2,3.5) node[right]{\tiny $\IRel$};

      \draw [thick, <->, squig, rel] (0,4) to (2,4);
      \draw (1,4) node[above]{\tiny $\MRel$};

      \draw (0,5) circle (2pt);
      \draw (0,5) node[left]{$x_3$};
      \draw [thick, ->, rel] (0,4) to (0,5);
      \draw (0,4.5) node[left]{\tiny $\IRel$};

      \draw [thick, dashed, rel] (0,5) to (0,6);

      \draw (2,5) circle (2pt);
      \draw (2,5) node[right]{$y_3 \Vdash p$};
      \draw [thick, ->, rel] (2,4) to (2,5);
      \draw (2,4.5) node[right]{\tiny $\IRel$};

      \draw [thick, dashed, rel] (2,5) to (2,6);

      \end{tikzpicture}
      \caption{The model $\Model{M}$. $\circ$ denotes a $\MRel$-reflexive world.}
      \label{fig:wk-em-dual-countermodel}
    \end{figure}

    Here, $\Model{M}$ is clearly an infallible CK-model.
    Also, it is easy to see that $x_0 \SatCK_\Model{M} \neg \Box \neg p$ but $x_0 \not\SatCK_\Model{M} \Diamond p$ since there is no $y$ such that $x_0 \IRel x_1 \MRel y \SatCK_\Model{M} p$.
    Now it suffices to show that $w \SatCK_\Model{M} \varphi \lor \neg \varphi$ for any $w \in W$ and $\varphi \in \MF$,
    which is proven by an easy induction on the construction of $\varphi$.
  \end{proof}
\end{proposition}

This means that they do not meet Simpson's criterion and our modified ones for a logic to be called ``intuitionistic''.
Simpson \cite{simpson_proof_1994} himself criticized $\Logic{WK}$ for the lack of this property, stating that $\Box$ and $\Diamond$ in $\Logic{WK}$ are ``hardly related at all.''
However, CMLs have independent motivations in computer science.
Wijesekera and Nerode \cite{wijesekera_tableaux_2005} use $\Logic{WK}$ as the propositional basis
of a constructive concurrent dynamic logic intended for the specification and modular verification of concurrent programs.
Mendler and de Paiva \cite{mendler_constructive_2005} motivate $\Logic{CK}$ by its potential application
to reasoning about contexts in artificial intelligence.
Mendler and Scheele \cite{mendler_computational_2014} develop a typed lambda calculus $\lambda \mathsf{CK}_n$
for contextual information processing.
CMLs also have elegant sequent calculi, which are obtained from the standard Gentzen-style sequent calculus $\mathrm{G}_\Logic{K}$ for $\Logic{K}$ by just restricting the number of formulas of succedents.
The admissibility of cut in these sequent calculi is thus proven easily as one may do for $\mathrm{G}_\Logic{K}$.

\begin{definition}[\cite{wijesekera_constructive_1990}, \cite{dalmonte_minimal_2025}]
  A sequent calculus $\mathrm{G}_\Logic{WK}$ is obtained from $\mathrm{G}_\Logic{K}$ by allowing only empty or single succedents.
  A sequent calculus $\mathrm{G}_\Logic{CK}$ is obtained similarly to $\mathrm{G}_\Logic{WK}$,
  but also disallowing the usage of the $\Diamond$-rule with empty succedents (i.e. $\frac{\Gamma, \varphi \imp {}}{\Box \Gamma, \Diamond \varphi \imp {}}$).
\end{definition}

\begin{remark}
  It is mentioned in \cite{das_intuitionistic_2023} that a sequent calculus for $\Logic{CK}$ can also be obtained
  from $\mathrm{G}_\Logic{K}$ by allowing only single (nonempty) succedents.
\end{remark}

\begin{proposition}[{\cite[Lem.\,1.5.1]{wijesekera_constructive_1990}, \cite[Thm.\,6.3]{dalmonte_minimal_2025}}]
  For $L \in \set{ \Logic{WK}, \Logic{CK} }$,
  $\mathrm{G}_L \vdash \Gamma \imp \varphi$ iff $L \vdash \bigwedge \Gamma \to \varphi$.
\end{proposition}

\begin{proposition}[{\cite[Thm.\,2.1.4]{wijesekera_constructive_1990}, \cite[Thm.\,6.2]{dalmonte_minimal_2025}}]
  Both $\mathrm{G}_\Logic{WK}$ and $\mathrm{G}_\Logic{CK}$ admit cut elimination.
\end{proposition}

Moreover, both $\Logic{CK}$ and $\Logic{WK}$ enjoy CIP, LIP, and UIP.
The first two properties are easy consequences of cut elimination,
and van der Giessen and Shillito recently proved the remaining one in their preprint \cite{giessen_uniform_2026}.

\begin{corollary}
  Both $\Logic{CK}$ and $\Logic{WK}$ enjoy LIP.
  \begin{proof}
    Although Wijesekera \cite[Thm.\,2.1.6]{wijesekera_constructive_1990} does not explicitly state LIP of $\Logic{WK}$,
    this follows easily by a slight modification of his proof that $\Logic{WK}$ enjoys CIP.
    The same argument applies to $\Logic{CK}$.
    See Theorem \ref{thm:ck-and-wk-enjoy-lip}.
  \end{proof}
\end{corollary}

\begin{proposition}[{\cite[Cors.\,1 and 2]{giessen_uniform_2026}}]
  Both $\Logic{CK}$ and $\Logic{WK}$ enjoy UIP.
\end{proposition}

Before moving on to IMLs, let us see that $\Logic{WK}$ can also be characterized by a class of CK-frames in which $W_\bot$ is nonempty.
De Groot et al.\ \cite{groot_semantical_2025} adopted a variation of CK-frames where there always is only one fallible world.

\begin{definition}
  We say a CK-frame $\Model{F} = (W, W_\bot, \IRel, \MRel)$ is \emph{singular}
  if $W_\bot = \set{ \Singularity }$ and $\Singularity \MRel \Singularity$.
  We conveniently call $\Singularity$ the \emph{singularity} of $\Model{F}$.
\end{definition}

\begin{theorem}[{\cite[Thm.\,VI.1]{groot_semantical_2025}}]
  \begin{equation*}
    \Logic{CK} = \set{ \varphi \in \MF; \forall \Model{F}: \text{singular CK-frame}\ (\Model{F} \ValidCK \varphi) }.
  \end{equation*}
\end{theorem}

They identified the following frame condition for $\name{N}_\Diamond$ over singular frames:

\begin{definition}[{\cite[Prop.\,V.6]{groot_semantical_2025}}]
  We say a singular CK-frame $\Model{F} = (W, \set{ \Singularity }, \IRel, \MRel)$ is \emph{$\name{N}_\Diamond$-corr}
  if it satisfies
  $\forall x\, \left( \left(\forall x'\, \left( x \IRel x' \imp x' \MRel \Singularity \right)\right) \imp x = \Singularity \right)$.
\end{definition}

The following propositions show that infallible frames and $\name{N}_\Diamond$-corr singular frames are indeed equivalent.

\begin{definition} \label{def:singularization} \leavevmode
  \begin{itemize}
    \item For an infallible CK-frame $\Model{F} = (W, \emptyset, \IRel, \MRel)$,
          pick a fresh world $\Singularity \notin W$, then
          we define a new frame $\sigma\Model{F} = (W \cup \set{\Singularity}, \set{\Singularity}, \IRel', \MRel')$,
          the \emph{singularization} of $\Model{F}$,
          by letting ${\IRel'} = {\IRel} \cup \set{ (\Singularity, \Singularity) }$ and ${\MRel'} = {\MRel} \cup \set{ (\Singularity, \Singularity) }$.
          Here, $\sigma\Model{F}$ is clearly a singular CK-frame, and trivially satisfies the $\name{N}_\Diamond$-corr condition.
    \item For an infallible CK-model $\Model{M} = (\Model{F}, V)$,
          the singularization of $\Model{M}$ is defined as
          $\sigma\Model{M} = (\sigma\Model{F}, \sigma V)$, where $\sigma V(p) = V(p) \cup \set{ \Singularity }$ for each $p \in \PropVar$.
    \item For an $\name{N}_\Diamond$-corr singular CK-frame $\Model{F} = (W, \set{\Singularity}, \IRel, \MRel)$ such that $W \setminus \set{ \Singularity } \ne \emptyset$,
          we define a new frame $\sigma^{-1}\Model{F} = (W', \emptyset, \IRel', \MRel')$,
          the \emph{desingularization} of $\Model{F}$,
          by letting
          $W' = W \setminus \set{ \Singularity }$,
          ${\IRel'} = {\IRel} \cap (W' \times W')$, and
          ${\MRel'} = ({\MRel} \cap (W' \times W')) \cup \set{ (x, y') \in W' \times W'; x \MRel \Singularity \tand \exists x' \in W'\,(x \IRel x' \MRel y') }$.
          Here, $\sigma^{-1}\Model{F}$ is clearly an infallible CK-frame.
    \item For an $\name{N}_\Diamond$-corr singular CK-model $\Model{M} = (\Model{F}, V)$ with nonempty non-singularity worlds,
          the desingularization of $\Model{M}$ is defined as
          $\sigma^{-1}\Model{M} = (\sigma^{-1}\Model{F}, \sigma^{-1} V)$, where $\sigma^{-1} V(p) = V(p) \cap W'$ for each $p \in \PropVar$.
  \end{itemize}
  We note that $\sigma^{-1} \sigma \Model{F} = \Model{F}$ holds, but $\sigma \sigma^{-1} \Model{F} = \Model{F}$ does not hold in general.
\end{definition}

\begin{proposition} \label{prop:singularization-bijective} \leavevmode
  \begin{itemize}
    \item For an infallible CK-frame $\Model{F} = (W, \emptyset, \IRel, \MRel)$,
          $\sigma$ is a bijection between infallible CK-models based on $\Model{F}$ and singular CK-models based on $\sigma \Model{F}$.
    \item For an $\name{N}_\Diamond$-corr singular CK-frame $\Model{F} = (W, \set{\Singularity}, \IRel, \MRel)$ with $W \setminus \set{ \Singularity } \ne \emptyset$,
          $\sigma^{-1}$ is a bijection between singular CK-models based on $\Model{F}$ and infallible CK-models based on $\sigma^{-1} \Model{F}$.
  \end{itemize}
  \begin{proof}
    Both are trivial by the definition of $\sigma V$ and $\sigma^{-1} V$, respectively.
  \end{proof}
\end{proposition}

\begin{proposition} \label{prop:singularity-appending}
  Let $\Model{M} = (W, \emptyset, \IRel, \MRel, V)$ be an infallible CK-model, then
  $\sigma\Model{M}$ validates the same set of formulas as $\Model{M}$ does.
  \begin{proof}
    Since every formula is true at $\Singularity$ by Proposition \ref{prop:all-is-fair},
    it suffices to show that $w \SatCK_\Model{M} \varphi \siff w \SatCK_{\sigma\Model{M}} \varphi$ for any $w \in W$ and $\varphi \in \MF$,
    which is routine.
  \end{proof}
\end{proposition}

\begin{proposition} \label{prop:singularity-removal}
  Let $\Model{M} = (W, \set{ \Singularity }, \IRel, \MRel, V)$
  be an $\name{N}_\Diamond$-corr singular CK-model with $W \setminus \set{ \Singularity } \ne \emptyset$, then
  $\sigma^{-1}\Model{M} = (W', \emptyset, \IRel', \MRel', \sigma^{-1} V)$ validates the same set of formulas as $\Model{M}$ does.
  \begin{proof}
    Since every formula is true at $\Singularity$ by Proposition \ref{prop:all-is-fair},
    it suffices to show that $w \SatCK_{\sigma^{-1}\Model{M}} \varphi \siff w \SatCK_\Model{M} \varphi$ for any $w \in W'$ and $\varphi \in \MF$.
    We use an induction on the construction of $\varphi$. We shall only prove the modal cases.

    ($\Box\Rightarrow$) Suppose $w \SatCK_{\sigma^{-1}\Model{M}} \Box \psi$ and take any $w', x' \in W$ such that $w \IRel w' \MRel x'$.
    If $x' = \Singularity$, then $\Singularity \SatCK_\Model{M} \psi$.
    Otherwise, $x' \ne \Singularity$, so $w' \ne \Singularity$ by $\set{ \Singularity }$ being upward closed with respect to $\MRel$.
    Here, $w \IRel w' \MRel x'$ and $w, w', x' \in W'$ imply $w \IRel' w' \MRel' x'$ by definition,
    then $x' \SatCK_{\sigma^{-1}\Model{M}} \psi$ by the assumption, so $x' \SatCK_\Model{M} \psi$ by the induction hypothesis.
    Therefore, $w \SatCK_\Model{M} \Box \psi$.

    ($\Box\Leftarrow$) Suppose $w \not \SatCK_{\sigma^{-1}\Model{M}} \Box \psi$, then there are $w', x'' \in W'$ such that $w \IRel' w' \MRel' x'' \not \SatCK_{\sigma^{-1}\Model{M}} \psi$.
    Here, $x'' \not \SatCK_\Model{M} \psi$ by the induction hypothesis.
    Also, $w' \MRel' x''$ implies either $(w \IRel)\, w' \MRel x'' \,(\not \SatCK_\Model{M} \psi)$,
    or there is $w'' \in W'$ such that $(w \IRel)\, w' \IRel w'' \MRel x'' \,(\not \SatCK_\Model{M} \psi)$.
    Therefore, $w \not \SatCK_\Model{M} \Box \psi$.

    ($\Diamond\Rightarrow$) Suppose $w \SatCK_{\sigma^{-1}\Model{M}} \Diamond \psi$ and take any $w' \in W$ such that $w \IRel w'$.
    If $w' = \Singularity$, then $w' \MRel \Singularity \SatCK_\Model{M} \psi$.
    Otherwise, $w' \ne \Singularity$ implies $w' \in W'$, then $w \IRel' w'$ by definition,
    so by assumption, there is $x' \in W'$ such that $w' \MRel' x' \SatCK_{\sigma^{-1}\Model{M}} \psi$, then $x' \SatCK_\Model{M} \psi$ by the induction hypothesis.
    Here, $w' \MRel' x'$ implies either $w' \MRel x' \,(\SatCK_\Model{M} \psi)$
    or $w' \MRel \Singularity \,(\SatCK_\Model{M} \psi)$.
    Therefore, $w \SatCK_\Model{M} \Diamond \psi$.

    ($\Diamond\Leftarrow$) Suppose $w \not \SatCK_{\sigma^{-1}\Model{M}} \Diamond \psi$, then there is $w' \in W'$ such that $w \IRel' w'$ and $\forall x' \in W'\, (w' \MRel' x' \imp x' \not \SatCK_{\sigma^{-1}\Model{M}} \psi)$.
    Here, $w \IRel w'$ by definition.
    We distinguish the following cases:
    \begin{itemize}
      \item Suppose $w' \not \MRel \Singularity$.
            Take any $x' \in W$ such that $w' \MRel x'$, then $x' \ne \Singularity$, then $w' \MRel' x'$ by definition,
            then $x' \not \SatCK_{\sigma^{-1}\Model{M}} \psi$ by the assumption,
            so $x' \not \SatCK_\Model{M} \psi$ by the induction hypothesis.
            Therefore, $w \not \SatCK_\Model{M} \Diamond \psi$.
      \item Now suppose $w' \MRel \Singularity$.
            Here, $w' \in W'$ implies $w' \ne \Singularity$,
            so by the contrapositive of $\name{N}_\Diamond$-corr, there is $w'' \in W$ such that $w' \IRel w''$ and $w'' \not \MRel \Singularity$.
            Take any $x'' \in W$ such that $w'' \MRel x''$, then $x'' \ne \Singularity$, then $w'' \ne \Singularity$ by $\set{ \Singularity }$ being upward closed with respect to $\MRel$.
            Now that $w', w'', x'' \in W'$, $w' \MRel \Singularity$, and $w' \IRel w'' \MRel x''$,
            then $w' \MRel' x''$ by definition,
            then $x'' \not \SatCK_{\sigma^{-1}\Model{M}} \psi$ by the assumption,
            so $x'' \not \SatCK_\Model{M} \psi$ by the induction hypothesis.
            Therefore, $w' \not \SatCK_\Model{M} \Diamond \psi$,
            which implies $w \not \SatCK_\Model{M} \Diamond \psi$ by persistency. \qedhere
    \end{itemize}
  \end{proof}
\end{proposition}

\begin{corollary} Let $\varphi \in \MF$.
  \begin{itemize}
    \item For any infallible CK-frame $\Model{F}$,
          $\Model{F} \ValidCK \varphi$ iff $\sigma\Model{F} \ValidCK \varphi$.
    \item For any $\name{N}_\Diamond$-corr singular CK-frame  $\Model{F} = (W, \set{\Singularity}, \IRel, \MRel)$ such that $W \setminus \set{ \Singularity } \ne \emptyset$,
          $\Model{F} \ValidCK \varphi$ iff $\sigma^{-1}\Model{F} \ValidCK \varphi$.
  \end{itemize}
\end{corollary}

\begin{corollary}[{\cite[Thm.\,VI.1]{groot_semantical_2025}}]\leavevmode
  $$\Logic{WK} = \set{ \varphi \in \MF; \forall \Model{F}: \name{N}_\Diamond\text{-corr singular CK-frame}\ (\Model{F} \ValidCK \varphi) }.$$
\end{corollary}

We note that they actually gave a unified semantics for both CMLs and IMLs, as they even characterized $\Logic{IK}$, the smallest IML, over their semantics.
We will refer to their results from time to time to see the relationship between CMLs and IMLs from the semantical point of view.
We however would like to point out that, with their semantics, it is hard to apply well-known proof strategies we would see in classical modal logics
(e.g.\ taking a join of models as we did in Proposition \ref{prop:wk-ck-dp}, as we need to merge two singularities into one).
This is the reason why we adopt the original semantics rather than theirs.

\section{Intuitionistic modal logics} \label{sec:iml}

\emph{Intuitionistic modal logics (IMLs)}, which admit the distribution of $\Diamond$ over $\lor$,
are originally proposed by Fischer Servi \cite{fischer_servi_modal_1977, fischer_axiomatizations_1984}
and studied further by other researchers such as Ewald \cite{ewald_intuitionistic_1986}, Plotkin \& Stirling \cite{plotkin_framework_1986}, and Simpson \cite{simpson_proof_1994}.

\begin{table}[hbt]
  \centering
  \caption{Basic modal axioms for IMLs.}
  \label{tab:basic-iml-axioms}
  \begin{tabular}{rl|rl}
    $\name{M}_\Box$     & \AxiomC{$\varphi \to \psi$} \UnaryInfC{$\Box \varphi \to \Box \psi$} \DisplayProof &
    $\name{M}_\Diamond$ & \AxiomC{$\varphi \to \psi$} \UnaryInfC{$\Diamond \varphi \to \Diamond \psi$} \DisplayProof \\
    $\name{N}_\Box$      & $\Box \top$ &
    $\name{N}_\Diamond$ & $\neg \Diamond \bot$ \\
    $\name{C}_\Box$     & $(\Box \varphi \land \Box \psi) \to \Box (\varphi \land \psi)$ &
    $\name{C}_\Diamond$ & $\Diamond(\varphi \lor \psi) \to (\Diamond \varphi \lor \Diamond \psi)$ \\
     $\name{K}_\Diamond$ & $\Box(\varphi \to \psi) \to (\Diamond \varphi \to \Diamond \psi)$ & & \\
    $\name{FS1}$ & $\Diamond (\varphi \to \psi) \to (\Box \varphi \to \Diamond \psi)$ &
    $\name{FS2}$ & $(\Diamond \varphi \to \Box \psi) \to \Box (\varphi \to \psi)$ \\
  \end{tabular}
\end{table}

\begin{definition}
  Let $\Logic{IK} \defeq \Logic{WK} + \name{C}_\Diamond + \name{FS2} = \Logic{iK}_\Box + \name{K}_\Diamond + \name{N}_\Diamond + \name{C}_\Diamond + \name{FS2}$.
\end{definition}

$\Logic{IK}$ has an alternative axiomatization, just as $\Logic{CK}$ and $\Logic{WK}$ do.

\begin{proposition}
  $\Logic{IK} \dashv \vdash \Logic{iK}_\Box + \name{M}_\Diamond + \name{N}_\Diamond + \name{C}_\Diamond + \name{FS1} + \name{FS2}$.
  \begin{proof}
    It easily follows from Proposition \ref{prop:ck-alternative-axiomatization}.
  \end{proof}
\end{proposition}

$\Logic{IK}$ enjoys the weak interaction between $\Box$ and $\Diamond$ we have seen in $\Logic{WK}$.

\begin{proposition} \label{prop:ik-partial-duality}
  (1) $\Logic{IK} \vdash \Box \neg \varphi \to \neg \Diamond \varphi$
  and (2) $\Logic{IK} \vdash \Diamond \neg \varphi \to \neg \Box \varphi$.
  \begin{proof}
    It is a direct consequence of Proposition \ref{prop:wk-partial-duality} and $\Logic{WK} \subseteq \Logic{IK}$.
  \end{proof}
\end{proposition}

Moreover, the new axiom $\name{FS2}$ allows even stronger interaction.

\begin{proposition} \label{prop:ik-partial-duality-2}
  $\Logic{IK} \vdash \neg \Diamond \varphi \to \Box \neg \varphi$.
  \begin{proof}
    Follows from the following derivation.
    \begin{center}
      \AxiomC{$ $}
      \UnaryInfC{$(\Diamond \varphi \to \bot) \to (\bot \to \Box \bot) \to (\Diamond \varphi \to \Box \bot)$}
      \UnaryInfC{$(\Diamond \varphi \to \bot) \to (\Diamond \varphi \to \Box \bot)$}

      \AxiomC{$ $}
      \RightLabel{\scriptsize{$\name{FS2}$}}
      \UnaryInfC{$(\Diamond \varphi \to \Box \bot) \to \Box (\varphi \to \bot)$}

      \BinaryInfC{$(\Diamond \varphi \to \bot) \to \Box (\varphi \to \bot)$}
      \UnaryInfC{$\neg \Diamond \varphi \to \Box \neg \varphi$}
      \DisplayProof
    \end{center}
    We note that no modal principles are used except for $\name{FS2}$.
  \end{proof}
\end{proposition}

Therefore, $\Box \neg \varphi$ and $\neg \Diamond \varphi$ are equivalent in $\Logic{IK}$.
We still do not have the full duality between $\Box$ and $\Diamond$, however,
because $\neg \Box \varphi \to \Diamond \neg \varphi$, the last interaction, is not provable in $\Logic{IK}$.
We shall present the semantics for $\Logic{IK}$ to see that.

\begin{definition} \label{def:fs-frame-condition}
  An f-frame $\Model{F} = (W, \IRel, \MRel)$ is said to be an \emph{FS2}\footnote{Also known by backward confluence \cite{balbiani_constructive_2021}, F2 \cite{simpson_proof_1994}, back-up confluence \cite{aguilera_gocalculus_2022}, and $\name{I}_{\Diamond \Box}$-weak \cite{groot_semantical_2025}.} frame
  if $(\MRel \cdot \IRel) \subseteq (\IRel \cdot \MRel)$,
  that is, $\forall x,y,y'\, (x \MRel y \IRel y' \imp \exists x'\,(x \IRel x' \MRel y'))$.
  Let us denote by $\FrameClass{FS2}$ the class of all FS2 frames.
\end{definition}

\begin{definition} \label{def:ik-frame}
  An f-frame $\Model{F} = (W, \IRel, \MRel)$ is said to be an \emph{IK-frame} if $\Model{F} \in \FrameClass{\Diamond\text{-p}} \cap \FrameClass{FS2}$.
  We also say $\Model{M} = (\Model{F}, V)$ is an \emph{IK-model} if $\Model{F}$ is an IK-frame.
\end{definition}

\begin{figure}[h]
  \centering

  \begin{subfigure}{.4\textwidth}
    \centering
    \begin{tikzpicture}
      \fill (0, 0) circle (2pt);
      \draw (0, 0) node[below]{$x$};

      \fill (0, 2) circle (2pt);
      \draw (0, 2) node[above]{$x'$};

      \fill (2, 0) circle (2pt);
      \draw (2, 0) node[below]{$y$};

      \fill (2, 2) circle (2pt);
      \draw (2, 2) node[above]{$\exists y'$};

      \draw [thick, ->, rel, squig] (0,0) to (2,0);
      \draw (0,1) node[left]{$\IRel$};
      \draw [thick, ->, rel] (0,0) to (0,2);
      \draw (1,0) node[below]{$\MRel$};
      \draw [thick, ->, rel, dashed] (2,0) to (2,2);
      \draw (2,1) node[right]{$\IRel$};
      \draw [thick, ->, rel, squig, dashed] (0,2) to (2,2);
      \draw (1,2) node[above]{$\MRel$};
    \end{tikzpicture}
    \caption*{The $\Diamond$-p condition (recap)}
  \end{subfigure}%
  \begin{subfigure}{.4\textwidth}
    \centering
    \begin{tikzpicture}
      \fill (0, 0) circle (2pt);
      \draw (0, 0) node[below]{$x$};

      \fill (0, 2) circle (2pt);
      \draw (0, 2) node[above]{$\exists x'$};

      \fill (2, 0) circle (2pt);
      \draw (2, 0) node[below]{$y$};

      \fill (2, 2) circle (2pt);
      \draw (2, 2) node[above]{$y'$};

      \draw [thick, ->, rel, dashed] (0,0) to (0,2);
      \draw (0,1) node[left]{$\IRel$};
      \draw [thick, ->, rel, squig] (0,0) to (2,0);
      \draw (1,0) node[below]{$\MRel$};
      \draw [thick, ->, rel] (2,0) to (2,2);
      \draw (2,1) node[right]{$\IRel$};
      \draw [thick, ->, rel, squig, dashed] (0,2) to (2,2);
      \draw (1,2) node[above]{$\MRel$};
    \end{tikzpicture}
    \caption*{The FS2 condition}
  \end{subfigure}

  \caption{Visualization of the required frame conditions for IK-frames}
  \label{fig:ik-frame-conditions}
\end{figure}

\begin{definition}
  For an f-model $\Model{M} = (W, \IRel, \MRel, V)$, we define a \emph{hybrid satisfaction relation} $\SatIK_\Model{M}$ on $W \times \MF$ by:
  \begin{enumerate}
    \item $x \SatIK_\Model{M} \psi_1 \to \psi_2 \defiff \forall x' \IRelRev x\, (x' \SatIK_\Model{M} \psi_1 \imp x' \SatIK_\Model{M} \psi_2)$
    \item $x \SatIK_\Model{M} \Box \varphi     \defiff \forall x' \IRelRev x\, \forall y' \MRelRev x'\, (y' \SatIK_\Model{M} \varphi)$
    \item $x \SatIK_\Model{M} \Diamond \varphi \defiff \exists y \MRelRev x\, (y \SatIK_\Model{M} \varphi)$
    \item The other cases are identical to the classical $\Vdash_\Model{M}$.
  \end{enumerate}
  We write $\Model{M} \ValidIK \varphi$ (``$\varphi$ is \emph{(h-)valid in} $\Model{M}$'') to mean $\forall x \in W\,(x \SatIK_\Model{M} \varphi)$,
  and $\Model{F} \ValidIK \varphi$ (``$\varphi$ is \emph{(h-)valid on} $\Model{F}$'') to mean that $\Model{M} \ValidIK \varphi$ for every f-model $\Model{M}$ based on $\Model{F}$.
\end{definition}

Here, we can see that $\SatIK$ evaluates $\Box \varphi$ as $\SatI$ does and $\Diamond \varphi$ as $\SatE$ does, thus the name \emph{hybrid}.
This means that the presence of the $\Diamond$-p condition is essential for persistency as we saw in Proposition \ref{prop:extrinsic-persistency} (2).
This also makes $\name{C}_\Diamond$ h-valid in every IK-model, just as it is e-valid in every f-model by Proposition \ref{prop:separation-cdia} (1), without the need of any frame condition.
On the other hand, $\name{FS2}$ is h-valid in the presence of the FS2 condition.

\begin{proposition}[Persistency] \label{prop:hybrid-persistency}
  Let $\Model{M} = (W, \IRel, \MRel, V)$ be a $\Diamond$-p model, and take any $\varphi \in \MF$.
  For any $x, x' \in W$, $x \SatIK_\Model{M} \varphi$ and $x \IRel x'$ implies $x' \SatIK_\Model{M} \varphi$.
  \begin{proof}
    Induction on the construction of $\varphi$. If $\varphi = \Box \psi$, then it can be proven similarly to Proposition \ref{prop:intrinsic-persistency}.
    If $\varphi = \Diamond \psi$, then it can be proven similarly to Proposition \ref{prop:extrinsic-persistency} (2 $\Rightarrow$).
    The other cases are routine.
  \end{proof}
\end{proposition}

\begin{proposition}[{\cite[Thm.\,4]{fischer_axiomatizations_1984}}] \label{prop:ik-model-validates-cdia-and-fs2}
  Both $\name{C}_\Diamond$ and $\name{FS2}$ are valid in every IK-model.
  \begin{proof}
    The former is proven similarly to Proposition \ref{prop:separation-cdia} (1), without using any frame condition on $\Model{M}$.
    For the latter, take any IK-model $\Model{M} = (W, \IRel, \MRel, V)$, any $\varphi, \psi \in \MF$, and any $x, x'$ such that $x \IRel x' \SatIK_\Model{M} \Diamond \varphi \to \Box \psi$,
    and it suffices to show that $x' \SatIK_\Model{M} \Box (\varphi \to \psi)$.
    To do so, take any $x'', y'', y'''$ such that $x' \IRel x'' \MRel y'' \IRel y''' \SatIK_\Model{M} \varphi$,
    and we shall show that $y''' \SatIK_\Model{M} \psi$.
    By the FS2 condition, there is $x'''$ such that $x'' \IRel x''' \MRel y'''$,
    then $y''' \SatIK_\Model{M} \varphi$ implies $x''' \SatIK_\Model{M} \Diamond \varphi$.
    Here, $x'' \SatIK_\Model{M} \Diamond \varphi \to \Box \psi$ by persistency,
    so $x''' \SatIK_\Model{M} \Box \psi$. Therefore, $y''' \SatIK_\Model{M} \psi$.
  \end{proof}
\end{proposition}

The completeness of $\Logic{IK}$ with respect to all IK-frames is routinely proven by constructing the canonical model.

\begin{theorem}[{\cite[Cor.\,1]{fischer_axiomatizations_1984}, \cite[pp.52--53]{simpson_proof_1994}}]
  \begin{equation*}
    \Logic{IK} = \set{ \varphi \in \MF; \forall \Model{F}: \text{IK-frame}\ (\Model{F} \ValidIK \varphi) }.
  \end{equation*}
\end{theorem}

$\Theory{V}^i = \Logic{WK} \subsetneq \Logic{IK}$ trivially holds from their axiomatizations.
We can also prove that their $\Diamond$-free fragments do not coincide.

\begin{proposition}[{\cite[Lem.\,10]{das_intuitionistic_2023}}] \label{prop:ik-nnbox-to-boxnn}
  $\Logic{IK} \vdash \neg \neg \Box \varphi \to \Box \neg \neg \varphi$.
  \begin{proof}
    Follows from the following derivation. \qedhere
    \begin{center}
      \scalebox{0.85}{
      \AxiomC{$ $}
      \RightLabel{\scriptsize{Prop.\,\ref{prop:ik-partial-duality} (2)}}
      \UnaryInfC{$\Diamond \neg \varphi \to \neg \Box \varphi$}
      \UnaryInfC{$\neg \neg \Box \varphi \to \neg \Diamond \neg \varphi$}
      \UnaryInfC{$\neg \neg \Box \varphi \to (\Diamond \neg \varphi \to \bot)$}

      \AxiomC{$ $}
      \UnaryInfC{$\bot \to \Box \bot$}

      \RightLabel{\scriptsize{Syll.}}
      \BinaryInfC{$\neg \neg \Box \varphi \to (\Diamond \neg \varphi \to \Box \bot)$}

      \AxiomC{$ $}
      \RightLabel{\scriptsize{$\name{FS2}$}}
      \UnaryInfC{$(\Diamond \neg \varphi \to \Box \bot) \to \Box (\neg \varphi \to \bot)$}
      \UnaryInfC{$(\Diamond \neg \varphi \to \Box \bot) \to \Box \neg \neg \varphi$}

      \RightLabel{\scriptsize{Syll.}}
      \BinaryInfC{$\neg \neg \Box \varphi \to \Box \neg \neg \varphi$}
      \DisplayProof
      }
    \end{center}
  \end{proof}
\end{proposition}

\begin{corollary}[{\cite{das_intuitionistic_2023}}] \label{cor:ik-proves-nnbb-to-bb}
  $\Logic{IK} \vdash \neg \neg \Box \bot \to \Box \bot$. Consequently, $\Logic{iK}_\Box \subsetneq \Logic{IK} \cap \MFb$.
  \begin{proof}
    Follows from the following derivation. \qedhere
      \begin{center}
        \AxiomC{$ $}
        \RightLabel{\scriptsize{Prop.\,\ref{prop:ik-nnbox-to-boxnn}}}
        \UnaryInfC{$\neg \neg \Box \bot \to \Box \neg \neg \bot$}

        \AxiomC{$ $}
        \UnaryInfC{$\neg \neg \bot \to \bot$}
        \RightLabel{\scriptsize{$\name{M}_\Box$}}
        \UnaryInfC{$\Box \neg \neg \bot \to \Box \bot$}

        \RightLabel{\scriptsize{Syll.}}
        \BinaryInfC{$\neg \neg \Box \bot \to \bot$}
        \DisplayProof
      \end{center}
  \end{proof}
\end{corollary}

Now we shall check if $\Logic{IK}$ is indeed intuitionistic according to Definition \ref{def:iml}.

\begin{proposition}[{\cite[Chapters 5 and 6]{simpson_proof_1994}}] \label{prop:ik-enjoys-dp}
  $\Logic{IK}$ enjoys DP.
  \begin{proof}
    Suppose that $\Logic{IK} \nvdash \varphi_1$ and $\Logic{IK} \nvdash \varphi_2$,
    then there are $\Model{M}_1 = (W_1, \IRel_1, \MRel_1, V_1)$, $\Model{M}_2 = (W_2, \IRel_2, \MRel_2, V_2)$,
    $w_1 \in W_1$, and $w_2 \in W_2$ such that $w_1 \not\SatIK_{\Model{M}_1} \varphi_1$ and $w_2 \not\SatIK_{\Model{M}_2} \varphi_2$.
    By replacing the two models with isomorphic copies if necessary, we may assume that $W_1 \cap W_2 = \emptyset$.
    We construct a new model $\Model{M} = (W, \IRel, \MRel, V)$ by the following:
    \begin{itemize}
      \item $W \defeq W_1 \cup W_2 \cup \set{w_0}$, where $w_0$ is fresh;
      \item $\IRel$ is the reflexive and transitive closure of $\set{ (w_0, w_1), (w_0, w_2) } \cup \mathord{\IRel_1} \cup \mathord{\IRel_2}$;
      \item $\mathord{\MRel} \defeq \mathord{\MRel_1} \cup \mathord{\MRel_2}$;
      \item $V(p) \defeq V_1(p) \cup V_2(p) \cup V_0(p)$, where $V_0(p) \defeq \set{w_0}$ if $w_1 \in V_1(p)$ and $w_2 \in V_2(p)$, and $V_0(p) \defeq \emptyset$ otherwise.
    \end{itemize}
    Here, it is easy to see that for any $i \in \set{1,2}$, $x \in W_i$, and $\psi \in \MF$, $x \SatIK_\Model{M} \psi$ iff $x \SatIK_{\Model{M}_i} \psi$.
    Then $w_0 \not\SatIK_\Model{M} \varphi_1$ and $w_0 \not\SatIK_\Model{M} \varphi_2$ by persistency,
    so $w_0 \not\SatIK_\Model{M} \varphi_1 \lor \varphi_2$.
    Now it remains to show that $\Model{M}$ is an IK-model,
    which is trivial since both $\Model{M}_1$ and $\Model{M}_2$ are IK-models and there is no $x$ such that $w_0 \MRel x$.
    Therefore, $\Logic{IK} \nvdash \varphi_1 \lor \varphi_2$ by completeness.
    The other direction is trivial.
  \end{proof}
\end{proposition}

\begin{proposition}[{\cite{simpson_proof_1994}}] \label{prop:ik-plus-em-is-k}
  $\Logic{IK} + \varphi \lor \neg \varphi = \Logic{K}$.
  \begin{proof}
    ($\subseteq$) Trivial.
    ($\supseteq$)
    We first show that $\Logic{IK} + \varphi \lor \neg \varphi \vdash \name{Dual}_\Box$.
    The following derivations prove it in both directions.
    \begin{center}
      \AxiomC{$ $}
      \RightLabel{\scriptsize{Prop.\,\ref{prop:ik-partial-duality-2}}}
      \UnaryInfC{$\neg \Diamond \neg \varphi \to \Box \neg \neg \varphi$}

      \AxiomC{$ $}
      \UnaryInfC{$\varphi \lor \neg \varphi$}
      \UnaryInfC{$\neg \neg \varphi \to \varphi$}
      \RightLabel{\scriptsize{$\name{M}_\Box$}}
      \UnaryInfC{$\Box \neg \neg \varphi \to \Box \varphi$}

      \RightLabel{\scriptsize{Syll.}}
      \BinaryInfC{$\neg \Diamond \neg \varphi \to \Box \varphi$}
      \DisplayProof
    \end{center}
    \begin{center}
      \AxiomC{$ $}
      \RightLabel{\scriptsize{Prop.\,\ref{prop:ik-partial-duality} (2)}}
      \UnaryInfC{$\Diamond \neg \varphi \to \neg \Box \varphi$}
      \UnaryInfC{$\neg \neg \Box \varphi \to \neg \Diamond \neg \varphi$}

      \AxiomC{$ $}
      \UnaryInfC{$\Box \varphi \to \neg \neg \Box \varphi$}

      \RightLabel{\scriptsize{Syll.}}
      \BinaryInfC{$\Box \varphi \to \neg \Diamond \neg \varphi$}
      \DisplayProof
    \end{center}
    We note that the second derivation does not use $\varphi \lor \neg \varphi$ and thus is possible in $\Logic{IK}$ and even in $\Logic{WK}$.
    Then, $\Logic{IK} + \varphi \lor \neg \varphi \supseteq \Logic{K}_\Box + \name{Dual}_\Box = \Logic{K}$.
  \end{proof}
\end{proposition}

\begin{proposition}[{\cite[pp.54--55]{simpson_proof_1994}}] \label{prop:ik-non-duality}
  (1) $\Logic{IK} \nvdash \name{Dual}_\Box$. (2) $\Logic{IK} \nvdash \name{Dual}_\Diamond$.
  \begin{proof}
    Let $\Model{M} = (\set{ x_0, x_1, y_0, y_1 }, \IRel, \MRel, V)$, where $\IRel$ is the reflexive closure of $\set{ (x_0, x_1), (y_0, y_1) }$,
    $\mathord{\MRel} = \set{ (x_0, y_0), (x_1, y_1) }$, and $V(p) = \set{ y_1 }$,
    then $\Model{M}$ is an IK-model.
    The model is visualized in Figure \ref{fig:ik-countermodels}.
    We use it to prove both.

    (1) $x_0 \not \SatIK_\Model{M} \Diamond \neg p$ and $x_1 \not \SatIK_\Model{M} \Diamond \neg p$ hold, which imply $x_0 \SatIK_\Model{M} \neg \Diamond \neg p$.
    Also, $x_0 \not \SatIK_\Model{M} \Box p$ holds.
    Therefore, $\Logic{IK} \nvdash \neg \Diamond \neg p \to \Box p$ by completeness.

    (2) $x_0 \not \SatIK_\Model{M} \Box \neg p$ and $x_1 \not \SatIK_\Model{M} \Box \neg p$ hold, which imply $x_0 \SatIK_\Model{M} \neg \Box \neg p$.
    Also, $x_0 \not \SatIK_\Model{M} \Diamond p$ holds.
    Therefore, $\Logic{IK} \nvdash \neg \Box \neg p \to \Diamond p$ by completeness.
  \end{proof}
\end{proposition}

\begin{figure}[h]
  \centering
  \begin{subfigure}{.5\textwidth}
    \centering
    \begin{tikzpicture}
      \fill (0, 0) circle (2pt);
      \draw (0, 0) node[left]{$x_0$};

      \fill (0, 2) circle (2pt);
      \draw (0, 2) node[left]{$x_1$};

      \fill (2, 0) circle (2pt);
      \draw (2, 0) node[right]{$y_0 \nVdash p$};

      \fill (2, 2) circle (2pt);
      \draw (2, 2) node[right]{$y_1 \Vdash p$};

      \draw [thick, ->, rel, squig] (0,0) to (2,0);
      \draw (1,0) node[below]{$\MRel$};
      \draw [thick, ->, rel] (0,0) to (0,2);
      \draw (0,1) node[left]{$\IRel$};
      \draw [thick, ->, rel] (2,0) to (2,2);
      \draw (2,1) node[right]{$\IRel$};
      \draw [thick, ->, rel, squig] (0,2) to (2,2);
      \draw (1,2) node[above]{$\MRel$};
    \end{tikzpicture}
    \caption*{Prop.\,\ref{prop:ik-non-duality}}
  \end{subfigure}%
  \begin{subfigure}{.5\textwidth}
    \centering
    \begin{tikzpicture}
      \fill (0, 0) circle (2pt);
      \draw (0, 0) node[left]{$x_0$};

      \fill (2, 0) circle (2pt);
      \draw (2, 0) node[right]{$y_0 \nVdash p$};

      \fill (2, 2) circle (2pt);
      \draw (2, 2) node[right]{$y_1 \Vdash p$};

      \draw [thick, ->, rel, squig] (0,0) to (2,0);
      \draw (1,0) node[below]{$\MRel$};
      \draw [thick, ->, rel] (2,0) to (2,2);
      \draw (2,1) node[right]{$\IRel$};
      \draw [thick, ->, rel, squig] (0,0) to (2,2);
      \draw (1,1) node[above left]{$\MRel$};
    \end{tikzpicture}
    \caption*{Prop.\,\ref{prop:ik-missing-interaction}}
  \end{subfigure}%

  \caption{Visualization of the countermodels for Propositions \ref{prop:ik-non-duality}, \ref{prop:ik-missing-interaction}}
  \label{fig:ik-countermodels}
\end{figure}

Therefore, $\Logic{IK}$ is indeed intuitionistic in our sense.
Moreover, $\Logic{IK}$ also satisfies Simpson's sixth requirement in a sense, that is, $\Logic{IK}$ corresponds to intuitionistic first-order logic via the standard translation.

\begin{proposition}[{\cite[Chapter 5]{simpson_proof_1994}}]
  $\Logic{IK} \vdash \varphi$ iff $\forall x\,\mathrm{ST}_x(\varphi)$ is valid in every intuitionistic first-order structure for $\Lang_{\text{ST}}$.
\end{proposition}

Here, the fact that $\Logic{IK} \nvdash \neg \Box \varphi \to \Diamond \neg \varphi$ roughly corresponds to the fact that $\neg \forall x \varphi(x) \to \exists x \neg \varphi(x)$ is not intuitionistically valid.
\begin{proposition} \label{prop:ik-missing-interaction}
  $\Logic{IK} \nvdash \neg \Box \varphi \to \Diamond \neg \varphi$.
  \begin{proof}
    Let $\Model{M} = (\set{ x_0, y_0, y_1 }, \IRel, \MRel, V)$, where $\mathord{\IRel} = \set{ (x_0, x_0), (y_0, y_0), (y_1, y_1) }$,
    $\mathord{\MRel} = \set{ (x_0, y_0), (x_0, y_1) }$, and $V(p) = \set{ y_1 }$,
    then $\Model{M}$ is an IK-model.
    The model is visualized in Figure \ref{fig:ik-countermodels}.
    Here, it is clear that $x_0 \SatIK_\Model{M} \neg \Box p$ but $x_0 \not \SatIK_\Model{M} \Diamond \neg p$.
    Therefore, $\Logic{IK} \nvdash \neg \Box p \to \Diamond \neg p$ by completeness.
  \end{proof}
\end{proposition}

\begin{remark}
  The countermodel for $\neg \Box \varphi \to \Diamond \neg \varphi$, presented above, also happens to be a countermodel for $\neg \Diamond \neg p \to \Box p$.
  We shall note that, however, these two formulas are incomparable over $\Logic{IK}$ i.e. there is an IK-model that validates $\neg \Diamond \neg p \to \Box p$ but falsifies $\neg \Box \varphi \to \Diamond \neg \varphi$, and vice versa.
  The latter is the countermodel presented in Proposition \ref{prop:ik-non-duality}.
  The former is left for the reader as an exercise.
\end{remark}

Moreover, $\Logic{IK}$ enjoys an analogue of Theorem \ref{thm:negative-embedding}:
classical $\Logic{K}$ can be embedded into $\Logic{IK}$ by a modified double negation translation.

\begin{definition}[{\cite[Def.\,9]{das_intuitionistic_2023}}] \label{def:modal-gg-translation}
  We inductively define the \emph{modal (G\"odel--Gentzen) double negation translation}
  $(\cdot)^N: \MF \to \MF$:
  \begin{align*}
    \bot^N &\defeq \bot, &
    p^N    &\defeq \neg \neg p, \\
    (\varphi \land \psi)^N &\defeq \varphi^N \land \psi^N, &
    (\varphi \lor \psi)^N  &\defeq \neg(\neg \varphi^N \land \neg \psi^N), \\
    (\varphi \to \psi)^N &\defeq \varphi^N \to \psi^N, \\
    (\Box\varphi)^N &\defeq \Box\varphi^N, &
    (\Diamond\varphi)^N &\defeq \neg \Box \neg \varphi^N.
  \end{align*}
\end{definition}

\begin{theorem}[{\cite[Thm.\,12]{das_intuitionistic_2023}}] \label{thm:modal-negative-embedding}
  For any $\varphi \in \MF$, $\Logic{K} \vdash \varphi$ iff $\Logic{IK} \vdash \varphi^N$.
  \begin{proof}[Proof (sketch)]
    An induction on $\psi$ shows that
    $\Logic{IK} \vdash \neg \neg \psi^N \to \psi^N$;
    the $\Box$ case uses Proposition \ref{prop:ik-nnbox-to-boxnn},
    followed by $\name{M}_\Box$ applied to the induction hypothesis.
    The forward implication then follows, as in Theorem \ref{thm:negative-embedding},
    by induction on a proof of $\varphi$ in $\Logic{K}$.
    Conversely, $\Logic{IK} \subseteq \Logic{K}$ and
    $\Logic{K} \vdash \varphi^N \leftrightarrow \varphi$.
  \end{proof}
\end{theorem}

\begin{remark}[{\cite{das_intuitionistic_2023}}]
  Let $\chi \defeq \neg \neg \Box \bot \to \Box \bot$.
  Then $\Logic{K} \vdash \chi$ and $\chi^N = \chi$, whereas
  $\chi \notin \Theory{V}^e_\Box = \Theory{V}^i \cap \MFb = \Logic{WK} \cap \MFb$
  by Propositions \ref{prop:separation-nnb}, \ref{prop:diafree-vi-is-veb},
  and Corollary \ref{cor:wk-completeness}.
  Therefore, the analogue of Theorem \ref{thm:modal-negative-embedding}
  does not hold for $\Logic{WK}$ or $\Logic{CK}$.
\end{remark}

So far, $\Logic{IK}$ seems like a perfect candidate for \emph{the} intuitionistic variant of $\Logic{K}$.
But it comes with a high price; Almeida et al.\ \cite{almeida_fischer-servi_2026} proved in their preprint that $\Logic{IK}$ loses all interpolation properties that are enjoyed by $\Logic{K}$, $\Logic{CK}$, and $\Logic{WK}$.

\begin{proposition}[{\cite[Thm.\,4.4]{almeida_fischer-servi_2026}}]
  $\Logic{IK}$ does not enjoy CIP, and, consequently, neither LIP nor UIP.
\end{proposition}

As such, there cannot be a ``good'' Gentzen-style sequent calculus for $\Logic{IK}$ that is cut-admissible,
since the cut elimination in such a calculus implies Craig interpolation property (see, e.g., \cite{tabatabai_universal_2025}).
Though, a cut-free labelled sequent calculus is given by Simpson \cite{simpson_proof_1994},
which is obtained from the labelled sequent calculus for $\Logic{K}$ by allowing only single succedents.

Finally, we shall present a folklore result\footnote{\scriptsize\url{https://prooftheory.blog/2022/08/19/brouwer-meets-kripke-constructivising-modal-logic/}} which says the $\Diamond$-free fragment of $\Logic{IK}$ is not finitely axiomatizable by $\Diamond$-free formulas.
Das and Marin \cite[Prop.\,17]{das_intuitionistic_2023} stated that the $\Diamond$-free flagment of $\Logic{WK} + \name{FS2}$ ($\Logic{CK} + \name{k}_4 + \name{k}_5$ in their paper) is finitely $\Diamond$-free axiomatizable,
and the $\Diamond$-free and $\lor$-free fragment of $\Logic{IK}$ is also finitely $\Diamond$-free $\lor$-free axiomatizable,
so the proof of this folklore, if it indeed holds, ``must make crucial use of $\lor$.''

\begin{problem}
  Show that $\Logic{IK} \cap \MFb$ is not finitely $\Diamond$-free axiomatizable.
\end{problem}

\section{Borderline logics between CMLs and IMLs} \label{sec:other}

So far, we have seen four different kinds of satisfaction relations.

\begin{definition}[recap]
  The satisfaction relations $\SatI$, $\SatCK$, $\SatIK$, $\SatE$ are defined as follows:
  \begin{itemize}
    \item For the $\bot$ case:
          \begin{itemize}
            \item $x \not \Vdash^\bullet \bot$, for $\bullet \in \set{ i, h, e }$.
            \item $x \SatCK \bot \defiff x \in W_\bot$, where $W_\bot$ is the set of fallible worlds.
          \end{itemize}
    \item For the $\Box$ case:
          \begin{itemize}
            \item $x \Vdash^\iota \Box \varphi \defiff \forall x' \IRelRev x \, \forall y' \MRelRev x' \, (y' \Vdash^\iota \varphi)$, for $\iota \in \set{ c, i, h }$.
            \item $x \SatE \Box \varphi \defiff \forall y \MRelRev x \, (y \SatE \varphi)$.
          \end{itemize}
    \item For the $\Diamond$ case:
          \begin{itemize}
            \item $x \Vdash^\iota \Diamond \varphi \defiff \forall x' \IRelRev x \, \exists y' \MRelRev x' \, (y' \Vdash^\iota \varphi)$, for $\iota \in \set{ c, i }$.
            \item $x \Vdash^\varepsilon \Diamond \varphi \defiff \exists y \MRelRev x \, (y \Vdash^\varepsilon \varphi)$, for $\varepsilon \in \set{ h, e }$.
          \end{itemize}
  \end{itemize}
\end{definition}

Here, $\SatCK$ and $\SatI$ are very similar; the only difference between them is in the case of $\bot$ i.e.\ whether fallible worlds are allowed or not.
Therefore, we can say that $\SatCK$ and $\SatI$ coincide in infallible CK-models, with respect to which $\Logic{WK}$ is sound and complete.
In this sense, we say $\Logic{WK}$ is the minimal logic in which $\SatCK$ and $\SatI$ coincide.
Similarly, as we have seen in Proposition \ref{prop:int-ext-coincides}, $\SatI$ and $\SatE$ coincide in $\Box$-p and $\Diamond$-p f-models.
Here, $\Theory{V}^e$ is the minimal logic in which $\SatI$ and $\SatE$ coincide.
Moreover, as $\SatIK$ is a ``hybrid'' of $\SatI$ and $\SatE$, we can even think of a minimal logic in which $\SatI$ and $\SatIK$ coincide,
which would sit right between $\Logic{WK}$ and $\Theory{V}^e$ (and between $\Logic{WK}$ and $\Logic{IK}$).

These logics would neither be traditional CMLs (logics that lack $\name{C}_\Diamond$) nor be traditional IMLs (logics over $\Logic{IK}$),
and are being studied by researchers such as Das and Marin \cite{das_intuitionistic_2023}, de Groot et al.\ \cite{groot_semantical_2025}, and Balbiani \cite{balbiani_natural_2024}.
The purpose of this section is to discuss these kinds of \emph{borderline} logics between CMLs and IMLs.
The following is easily proven as we did in Proposition \ref{prop:int-ext-coincides}.

\begin{figure}[htb]
  \centering

  \begin{subfigure}{.33\textwidth}
    \centering
    \begin{tikzpicture}
      \fill (0, 0) circle (2pt);
      \draw (0, 0) node[below]{$x$};

      \fill (0, 2) circle (2pt);
      \draw (0, 2) node[above]{$x'$};

      \fill (2, 0) circle (2pt);
      \draw (2, 0) node[below]{$\vphantom{\exists}y$};

      \fill (2, 2) circle (2pt);
      \draw (2, 2) node[above]{$\exists y'$};

      \draw [thick, ->, rel, squig] (0,0) to (2,0);
      \draw (0,1) node[left]{$\IRel$};
      \draw [thick, ->, rel] (0,0) to (0,2);
      \draw (1,0) node[below]{$\MRel$};
      \draw [thick, ->, rel, dashed] (2,0) to (2,2);
      \draw (2,1) node[right]{$\IRel$};
      \draw [thick, ->, rel, squig, dashed] (0,2) to (2,2);
      \draw (1,2) node[above]{$\MRel$};
    \end{tikzpicture}
    \caption*{The $\Diamond$-p condition}
  \end{subfigure}%
  \begin{subfigure}{.33\textwidth}
    \centering
    \begin{tikzpicture}
      \fill (0, 0) circle (2pt);
      \draw (0, 0) node[below]{$x$};

      \fill (0, 2) circle (2pt);
      \draw (0, 2) node[above]{$x'$};

      \fill (2, 0) circle (2pt);
      \draw (2, 0) node[below]{$\exists y$};

      \fill (2, 2) circle (2pt);
      \draw (2, 2) node[above]{$\vphantom{\exists}y'$};

      \draw [thick, ->, rel] (0,0) to (0,2);
      \draw (0,1) node[left]{$\IRel$};
      \draw [thick, ->, rel, squig, dashed] (0,0) to (2,0);
      \draw (1,0) node[below]{$\MRel$};
      \draw [thick, ->, rel, dashed] (2,0) to (2,2);
      \draw (2,1) node[right]{$\IRel$};
      \draw [thick, ->, rel, squig] (0,2) to (2,2);
      \draw (1,2) node[above]{$\MRel$};
    \end{tikzpicture}
    \caption*{The $\Box$-p condition}
  \end{subfigure}%
  \begin{subfigure}{.33\textwidth}
    \centering
    \begin{tikzpicture}
      \fill (0, 0) circle (2pt);
      \draw (0, 0) node[below]{$x$};

      \fill (0, 2) circle (2pt);
      \draw (0, 2) node[above]{$\exists x'$};

      \fill (2, 0) circle (2pt);
      \draw (2, 0) node[below]{$\vphantom{\exists}y$};

      \fill (2, 2) circle (2pt);
      \draw (2, 2) node[above]{$y'$};

      \draw [thick, ->, rel, dashed] (0,0) to (0,2);
      \draw (0,1) node[left]{$\IRel$};
      \draw [thick, ->, rel, squig] (0,0) to (2,0);
      \draw (1,0) node[below]{$\MRel$};
      \draw [thick, ->, rel] (2,0) to (2,2);
      \draw (2,1) node[right]{$\IRel$};
      \draw [thick, ->, rel, squig, dashed] (0,2) to (2,2);
      \draw (1,2) node[above]{$\MRel$};
    \end{tikzpicture}
    \caption*{The FS2 condition}
  \end{subfigure}%

  \caption{Visualization of the various frame conditions (recap)}
  \label{fig:various-frame-conditions}
\end{figure}

\begin{proposition}\label{prop:int-ik-coincides}
  Let $\Model{M} = (W, \IRel, \MRel, V)$ be a $\Diamond$-p model, then
  $x \SatI_\Model{M} \varphi \iff x \SatIK_\Model{M} \varphi$ for any $\varphi \in \MF$ and $x \in W$.
  \begin{proof}
    We use an induction on the construction of $\varphi$.
    We shall only prove the modal cases.
    We note that $\SatIK_\Model{M}$ is persistent by the proof of Proposition \ref{prop:hybrid-persistency}.

    ($\Box\Leftrightarrow$)
    Easily proven by the induction hypothesis.

    ($\Diamond\Rightarrow$)
    Suppose $x \SatI_\Model{M} \Diamond \psi$, then $\forall x'\, (x \IRel x' \imp \exists y'\, (x' \MRel y' \tand y' \SatI_\Model{M} \varphi))$.
    This and $x \IRel x$ imply that there is $y$ such that $x \MRel y$ and $y \SatI_\Model{M} \psi$, then $y \SatIK_\Model{M} \psi$ by the induction hypothesis.
    Therefore, $x \SatIK_\Model{M} \Diamond \psi$.

    ($\Diamond\Leftarrow$)
    Suppose $x \SatIK_\Model{M} \Diamond \psi$, then there is $y$ such that $x \MRel y$ and $y \SatIK_\Model{M} \psi$.
    Now take any $x'$ such that $x \IRel x'$,
    then by the $\Diamond$-p condition, $x' \IRelRev x \MRel y$ implies that there is $y'$ such that $x' \MRel y' \IRelRev y$.
    Here, $y \IRel y' \SatIK_\Model{M} \psi$ by persistency,
    so $y' \SatI_\Model{M} \psi$ by the induction hypothesis.
    Therefore, $x \SatI_\Model{M} \Diamond \psi$.
  \end{proof}
\end{proposition}

Since $\Diamond$-p is the frame condition that is required for the persistency of $\SatIK$,
We can say that $\set{ \varphi \in \MF; \forall \Model{F} \in \FrameClass{\Diamond\text{-p}}\,(\Model{F} \ValidI \varphi) }$ is the minimal logic under which $\SatI$ and $\SatIK$ coincide.
One would simply expect that it is axiomatized by the below logic:
\begin{definition}
  Let $\Logic{IK}^- \defeq \Logic{WK} + \name{C}_\Diamond$.
\end{definition}

$\Logic{IK}^-$ is introduced by Das and Marin \cite{das_intuitionistic_2023} ($\Logic{CK} + \name{k}_3 + \name{k}_5$ in their paper)
to show that the presence of the $\name{C}_\Diamond$ and $\name{N}_\Diamond$ axioms do not force any $\Diamond$-free axioms over $\Logic{CK}$.
De Groot et al.\ \cite{groot_semantical_2025} gave a completeness proof of $\Logic{IK}^-$ ($\Logic{CK} + \name{C}_\Diamond + \name{N}_\Diamond$ in their paper)
with respect to a certain class of singular CK-frames, and used it to prove the same property semantically.

\begin{theorem}[{\cite[Thm.\ VI.1]{groot_semantical_2025}}]
  We say a singular CK-frame $\Model{F} = (W, \set{ \Singularity }, \IRel, \MRel)$ is $\name{C}_\Diamond$-suff if it satisfies
  $\forall x_0\, \exists x_1\, (x_0 \IRel x_1 \tand \forall x_2\, \forall y_1\, (x_0 \IRel x_2 \tand x_1 \MRel y_1) \imp \exists y_2\, (x_2 \MRel y_2 \tand y_1 \IRel y_2))$.
  Then,
  $$\Logic{IK}^- = \set{ \varphi \in \MF; \forall \Model{F}: \name{N}_\Diamond\text{-corr}\ \name{C}_\Diamond\text{-suff singular CK-frame}\ (\Model{F} \ValidCK \varphi) }.$$
\end{theorem}

\begin{problem}
  Identify the class of f-frames (i.e.\ infallible CK-frames) with respect to which $\Logic{IK}^-$ is sound and complete.
\end{problem}

\begin{theorem}[{\cite{das_intuitionistic_2023}}]
  $\Logic{IK}^-$ has a cut-free nested sequent calculus.
\end{theorem}

\begin{corollary}[{\cite{das_intuitionistic_2023}, \cite[Prop.\ VII.2]{groot_semantical_2025}}] \label{cor:ikminus-box-fragment}
  $\Logic{IK}^- \cap \MFb = \Logic{iK}_\Box$.
\end{corollary}

We will later see that $\Logic{IK}^- \cap \MFd = \Logic{iK}_\Diamond$ also hold.

\begin{problem}
  It would be easy to prove that $\Logic{IK}^-$ has DP by a semantical argument.
  It is also clear that $\Logic{IK}^- \nvdash \name{Dual}_\Box$ and $\Logic{IK}^- \nvdash \name{Dual}_\Diamond$ since $\Logic{IK}^- \subseteq \Logic{IK}$.
  Now, does $\Logic{IK}^- + \varphi \lor \neg \varphi = \Logic{K}$ hold?
\end{problem}

$\Logic{IK}^-$ would seem as a perfect logic that fits right between $\Logic{WK}$ and $\Logic{IK}$.
Interestingly, however, it is not the logic we are looking for.
Balbiani et al.\ \cite{balbiani_natural_2024} gave the following axiomatization of the logic we are after.

\begin{definition}
  Let $\Logic{FIK} \defeq \Logic{IK}^- + \name{wCD}$,
  where $\name{wCD} \defeq \Box(\varphi \lor \psi) \to ((\Diamond \varphi \to \Box \psi) \to \Box \psi)$.
\end{definition}

\begin{theorem}[{\cite[Thms.\,7 and 15]{balbiani_natural_2024}}]
  \begin{equation*}
    \Logic{FIK} = \set{ \varphi \in \MF; \forall \Model{F} \in \FrameClass{\Diamond\text{-p}}\,(\Model{F} \ValidIK \varphi) }.
  \end{equation*}
\end{theorem}

The $\name{wCD}$ axiom is a weakening of $\name{FS2}$, which means we have $\Logic{FIK} \subsetneq \Logic{IK}$.

\begin{proposition} \label{prop:fik-is-weaker-than-ik}
  $\Logic{FIK} \subsetneq \Logic{IK}$, that is,
  (1) $\Logic{IK} \vdash \name{wCD}$ and (2) $\Logic{FIK} \nvdash \name{FS2}$.
  \begin{proof}
    (1) Follows from the derivation below:
    \begin{center}
      \AxiomC{$ $}
      \RightLabel{\scriptsize{$\name{FS2}$}}
      \UnaryInfC{$(\Diamond \varphi \to \Box \psi) \to \Box (\varphi \to \psi)$}

      \AxiomC{$ $}
      \RightLabel{\scriptsize{}}
      \UnaryInfC{$(\varphi \to \psi) \to ((\varphi \lor \psi) \to \psi)$}
      \RightLabel{\scriptsize{$\name{M}_\Box$}}
      \UnaryInfC{$\Box(\varphi \to \psi) \to (\Box(\varphi \lor \psi) \to \Box\psi)$}

      \RightLabel{\scriptsize{Syll.}}
      \BinaryInfC{$(\Diamond \varphi \to \Box \psi) \to (\Box (\varphi \lor \psi) \to \Box \psi)$}
      \UnaryInfC{$\Box (\varphi \lor \psi) \to ((\Diamond \varphi \to \Box \psi) \to \Box \psi)$}
      \DisplayProof
    \end{center}
    (2) By soundness, it suffices to construct a $\Diamond$-p model that falsifies $\name{FS2}$.
    Let $\Model{M}_{\name{FS2}} = (\set{x, y, y'}, \IRel, \MRel, V)$, where
    $\mathord{\IRel} = \set{ (x,x), (y,y), (y,y'), (y',y') }$,
    $\mathord{\MRel} = \set{ (x,y) }$,
    $V(p) = \set{ y' }$, and $V(q) = \emptyset$,
    then $\Model{M}_{\name{FS2}}$ is clearly a $\Diamond$-p model.
    The model $\Model{M}_{\name{FS2}}$ is visualized in Figure \ref{fig:fs2-countermodel}.
    \begin{figure}[htb]
      \centering
      \begin{tikzpicture}
      \fill (0, 2) circle (2pt);
      \draw (0, 2) node[left]{$x$};

      \fill (2, 2) circle (2pt);
      \draw (2, 2) node[right]{$y \not \SatIK p, q$};

      \draw [thick, ->, squig, rel] (0,2) to (2, 2);
      \draw (1, 2) node[below]{$\MRel$};

      \draw [thick, ->, rel] (2, 2) to (2, 4);
      \draw (2, 3) node[right]{$\IRel$};

      \fill (2, 4) circle (2pt);
      \draw (2, 4) node[right]{$y' \SatIK p,\, \not \SatIK q$};

      \end{tikzpicture}
      \caption{The model $\Model{M}_{\name{FS2}}$.}
      \label{fig:fs2-countermodel}
    \end{figure}
    
    Here, $x \MRel y \not \SatIK p$ implies $x \not \SatIK \Diamond p$, so $x \SatIK \Diamond p \to \Box q$.
    On the other hand, $y \IRel y'$, $y' \SatIK p$, and $y' \not \SatIK q$ imply $y \not \SatIK p \to q$,
    so $x \not \SatIK \Box(p \to q)$. Therefore, $x \not \SatIK \name{FS2}$.
  \end{proof}
\end{proposition}

$\Logic{FIK}$ satisfies most of the properties that are required for a modal logic to be called ``intuitionistic''.

\begin{proposition}[{\cite[p.\,13:6]{balbiani_natural_2024}}] \label{prop:fik-is-mostly-intuitionistic} \leavevmode
  (1) $\Logic{FIK}$ enjoys DP.
  (2) $\Logic{FIK} + \varphi \lor \neg \varphi = \Logic{K}$.\\
  (3) $\Logic{FIK} \nvdash \name{Dual}_\Box$ and $\Logic{FIK} \nvdash \name{Dual}_\Diamond$.
  \begin{proof}
    (1) It is easily proven by a semantical argument, similarly to Proposition \ref{prop:ik-enjoys-dp}.
    (2) As $\Logic{WK} \vdash \Box \varphi \to \neg \Diamond \neg \varphi$,
    it suffices to show that $\Logic{FIK} + \varphi \lor \neg \varphi \vdash \neg \Diamond \neg \varphi \to \Box \varphi$ (see Proposition \ref{prop:ik-plus-em-is-k}).
    The following derivation proves it:
    \begin{center}
      \scalebox{0.75}{
        \AxiomC{$ $}
        \UnaryInfC{$\bot \to \Box \varphi$}
        \UnaryInfC{$(\Diamond \neg \varphi \to \bot) \to (\Diamond \neg \varphi \to \Box \varphi)$}

        \AxiomC{$ $}
        \UnaryInfC{$\neg \varphi \lor \varphi$}
        \RightLabel{\scriptsize{$\name{M}_\Box$, $\name{N}_\Box$}}
        \UnaryInfC{$\Box(\neg \varphi \lor \varphi)$}

        \AxiomC{$ $}
        \RightLabel{\scriptsize{$\name{wCD}$}}
        \UnaryInfC{$\Box(\neg \varphi \lor \varphi) \to ((\Diamond \neg \varphi \to \Box \varphi) \to \Box \varphi)$}

        \RightLabel{\scriptsize{MP}}
        \BinaryInfC{$(\Diamond \neg \varphi \to \Box \varphi) \to \Box \varphi$}

        \RightLabel{\scriptsize{Syll.}}
        \BinaryInfC{$(\Diamond \neg \varphi \to \bot) \to \Box \varphi$}
        \UnaryInfC{$\neg \Diamond \neg \varphi \to \Box \varphi$}
        \DisplayProof
      }
    \end{center}
    (3) Follows from Proposition \ref{prop:ik-non-duality} and $\Logic{FIK} \subsetneq \Logic{IK}$.
  \end{proof}
\end{proposition}

The only exception is Simpson's sixth requirement; by $\Logic{FIK} \subsetneq \Logic{IK}$, there is $\varphi \in \MF$
such that $\forall x\,\mathrm{ST}_x(\varphi)$ is valid in every intuitionistic first-order structure for $\Lang_{\text{ST}}$
but $\Logic{FIK} \nvdash \varphi$. An easy example is $\neg \Diamond \varphi \to \Box \neg \varphi$, which corresponds to an intuitionistically valid sentence $\neg \exists x \varphi(x) \to \forall x \neg \varphi(x)$.

\begin{proposition} \label{prop:fik-lacks-negdia-to-boxneg}
  $\Logic{FIK} \nvdash \neg \Diamond \varphi \to \Box \neg \varphi$.
  \begin{proof}
    One can construct a falsifying model that satisfies the $\Diamond$-p and $\Box$-p conditions but not FS2.
    The model is visualized in Figure \ref{fig:negdia-countermodel}.
    One can show that $x \SatIK \neg \Diamond p$ but $x \not \SatIK \Box \neg p$,
    the detail of which is left to the reader as an exercise.
    \begin{figure}[htb]
      \centering
      \begin{tikzpicture}
      \fill (0, 2) circle (2pt);
      \draw (0, 2) node[left]{$x$};

      \fill (2, 2) circle (2pt);
      \draw (2, 2) node[right]{$y \not \SatIK p$};

      \draw [thick, ->, squig, rel] (0,2) to (2, 2);
      \draw (1, 2) node[below]{$\MRel$};

      \fill (0, 4) circle (2pt);
      \draw (0, 4) node[left]{$x'$};

      \draw [thick, ->, rel] (0, 2) to (0, 4);
      \draw (0, 3) node[left]{$\IRel$};

      \draw [thick, ->, squig, rel] (0, 4) to (2, 4);
      \draw (1, 4) node[above]{$\MRel$};

      \draw [thick, ->, rel] (2, 2) to (2, 4);
      \draw (2, 3) node[left]{$\IRel$};

      \fill (2, 4) circle (2pt);
      \draw (2, 4) node[right]{$y_1' \not \SatIK p$};

      \draw [thick, ->, rel] (2, 2) to (4, 4);
      \draw (3, 3) node[right]{$\IRel$};

      \fill (4, 4) circle (2pt);
      \draw (4, 4) node[right]{$y_2' \SatIK p$};

      \end{tikzpicture}
      \caption{The countermodel for Proposition \ref{prop:fik-lacks-negdia-to-boxneg}.}
      \label{fig:negdia-countermodel}
    \end{figure}
  \end{proof}
\end{proposition}

Moreover, the $\Diamond$-free fragment of $\Logic{FIK}$ does not coincide with that of $\Logic{WK}$ and even that of $\Logic{IK}$.

\begin{proposition}[{\cite[Prop.\,VIII.7]{groot_semantical_2025}}] \label{prop:fik-box-fragment}
  $\Logic{iK}_\Box \subsetneq \Logic{FIK} \cap \MFb \subsetneq \Logic{IK} \cap \MFb$.
  \begin{proof}
    For the first strict inclusion, it suffices to show $\Logic{FIK} \vdash \neg \neg \Box \bot \to \Box \bot$.
    The following derivations prove it:
    \begin{center}
      \AxiomC{$ $}
      \RightLabel{\scriptsize{Prop.\,\ref{prop:wk-partial-duality} (2), $\Logic{WK} \subseteq \Logic{FIK}$}}
      \UnaryInfC{$\Diamond \neg \bot \to \neg \Box \bot$}
      \UnaryInfC{$\neg \neg \Box \bot \to \neg \Diamond \neg \bot$}
      \UnaryInfC{$\neg \neg \Box \bot \to (\Diamond \neg \bot \to \bot)$}

      \AxiomC{$ $}
      \UnaryInfC{$\bot \to \Box \bot$}

      \RightLabel{\scriptsize{Syll.}}
      \BinaryInfC{$\neg \neg \Box \bot \to (\Diamond \neg \bot \to \Box \bot)$}
      \UnaryInfC{$\neg \neg \Box \bot \to (\Diamond \top \to \Box \bot)$}
      \DisplayProof
    \end{center}
    \begin{center}
      \AxiomC{$ $}
      \RightLabel{\scriptsize{$\name{N}_\Box$}}
      \UnaryInfC{$\Box \top$}

      \AxiomC{$ $}
      \UnaryInfC{$\top \to (\top \lor \bot)$}
      \RightLabel{\scriptsize{$\name{M}_\Box$}}
      \UnaryInfC{$\Box \top \to \Box (\top \lor \bot)$}
      \RightLabel{\scriptsize{MP}}
      \BinaryInfC{$\Box (\top \lor \bot)$}

      \AxiomC{$ $}
      \RightLabel{\scriptsize{$\name{wCD}$}}
      \UnaryInfC{$\Box(\top \lor \bot) \to ((\Diamond \top \to \Box \bot) \to \Box \bot)$}

      \RightLabel{\scriptsize{MP}}
      \BinaryInfC{$(\Diamond \top \to \Box \bot) \to \Box \bot$}
      \DisplayProof
    \end{center}
    \begin{center}
      \AxiomC{$ \vdots $}
      \UnaryInfC{$\neg \neg \Box \bot \to (\Diamond \top \to \Box \bot)$}

      \AxiomC{$ \vdots $}
      \UnaryInfC{$(\Diamond \top \to \Box \bot) \to \Box \bot$}

      \RightLabel{\scriptsize{Syll.}}
      \BinaryInfC{$\neg \neg \Box \bot \to \bot$}
      \DisplayProof
    \end{center}
    For the second strict inclusion, by Proposition \ref{prop:ik-nnbox-to-boxnn},
    it suffices to show $\Logic{FIK} \nvdash \neg \neg \Box \varphi \to \Box \neg \neg \varphi$.
    Here, the countermodel in Proposition \ref{prop:fik-lacks-negdia-to-boxneg}
    falsifies $\neg \neg \Box \neg p \to \Box \neg \neg \neg p$,
    so this formula is not provable in $\Logic{FIK}$ by soundness.
  \end{proof}
\end{proposition}

\begin{corollary}
  $\Logic{IK}^- \subsetneq \Logic{FIK}$.
  \begin{proof}
    Follows from Corollary \ref{cor:ikminus-box-fragment} and Proposition \ref{prop:fik-box-fragment}.
  \end{proof}
\end{corollary}

Finally, the $\Box$-free fragment of $\Logic{FIK}$ is exactly $\Logic{iK}_\Diamond$.

\begin{theorem}
  $\Logic{FIK} \cap \MFd = \Logic{iK}_\Diamond$.
  \begin{proof}
    ($\supseteq$) Trivial from the axiomatization. ($\subseteq$) Suppose $\Logic{iK}_\Diamond \nvdash \varphi$,
    then by Theorem \ref{thm:ikb-ikd-completeness},
    there are $\Model{M} = (W, \IRel, \MRel, V)$ and $w \in W$ such that $\Model{M}$ satisfies the $\Diamond$-p condition and $x \not \SatE \varphi$.
    Since $\varphi \in \MFd$, it is easy to see that $x \not \SatIK \varphi$. Therefore, $\Logic{FIK} \nvdash \varphi$ by soundness.
  \end{proof}
\end{theorem}

\begin{corollary}
  $\Logic{IK}^- \cap \MFd = \Logic{iK}_\Diamond$.
  \begin{proof}
    ($\supseteq$) Trivial from the axiomatization.
    ($\subseteq$) Follows from $\Logic{IK}^- \subseteq \Logic{FIK}$.
  \end{proof}
\end{corollary}

Now we shall consider $\Theory{V}^e$, the minimal logic in which $\SatI$ and $\SatE$ coincide.
Balbiani et al. \cite{balbiani_local_2024} gave the following axiomatization.
Their axiom $\name{RV}$ is the axiom $\name{CD}$ displayed below.

\begin{definition}
  Let $\Logic{LIK} \defeq \Logic{IK}^- + \name{CD}$,
  where $\name{CD} \defeq \Box(\varphi \lor \psi) \to (\Diamond \varphi \lor \Box \psi)$.
\end{definition}

\begin{theorem}[{\cite[Thms.\,1 and 2]{balbiani_local_2024}}]
  $$\Logic{LIK} = \Theory{V}^e = \set{ \varphi \in \MF; \forall \Model{F} \in \FrameClass{\Box\text{-p}} \cap \FrameClass{\Diamond\text{-p}}\, (\Model{F} \ValidIK \varphi) }.$$
  \begin{proof}
    The first equality is the cited completeness result.
    The second equality follows from Propositions \ref{prop:int-ik-coincides} and \ref{prop:int-ext-coincides}.
  \end{proof}
\end{theorem}

The \emph{constant domain} axiom $\name{CD}$,
which appeared in \cite{balbiani_constructive_2021},
is also known as $\name{RV}$ \cite{balbiani_local_2024}
and as the \emph{crispness} axiom $\name{Cr}$ \cite{rodriguez_axiomatization_2021}.
It is another strengthening of $\name{wCD}$ that is incompatible with $\name{FS2}$,
which means we have $\Logic{FIK} \subsetneq \Logic{LIK}$,
but neither $\Logic{LIK} \subseteq \Logic{IK}$ nor $\Logic{LIK} \supseteq \Logic{IK}$.

\begin{proposition} \label{prop:fik-is-weaker-than-lik}
  $\Logic{FIK} \subsetneq \Logic{LIK}$ i.e.\ 
  (1) $\Logic{LIK} \vdash \name{wCD}$ and (2) $\Logic{FIK} \nvdash \name{CD}$.
  \begin{proof}
    (1) Follows from the derivation below:
    \begin{center}
      \AxiomC{$ $}
      \RightLabel{\scriptsize{$\name{CD}$}}
      \UnaryInfC{$\Box(\varphi \lor \psi) \to (\Diamond \varphi \lor \Box \psi)$}

      \AxiomC{$ $}
      \UnaryInfC{$(p \lor q) \to ((p \to q) \to q)$}
      \UnaryInfC{$(\Diamond \varphi \lor \Box \psi) \to ((\Diamond \varphi \to \Box \psi) \to \Box \psi)$}

      \RightLabel{\scriptsize{Syll.}}
      \BinaryInfC{$\Box (\varphi \lor \psi) \to ((\Diamond \varphi \to \Box \psi) \to \Box \psi)$}
      \DisplayProof
    \end{center}
    (2) By soundness, it suffices to construct a $\Diamond$-p model that falsifies $\name{CD}$.
    Let $\Model{M}_{\name{CD}} = (\set{x, x', y'}, \IRel, \MRel, V)$, where
    $\mathord{\IRel} = \set{ (x,x), (x,x'), (x',x'), (y',y') }$,
    $\mathord{\MRel} = \set{ (x',y') }$,
    $V(p) = \set{ y' }$, and $V(q) = \emptyset$,
    then $\Model{M}_{\name{CD}}$ is clearly a $\Diamond$-p model.
    The model $\Model{M}_{\name{CD}}$ is visualized in Figure \ref{fig:cd-countermodel}.
    \begin{figure}[htb]
      \centering
      \begin{tikzpicture}
      \fill (0, 2) circle (2pt);
      \draw (0, 2) node[left]{$x$};

      \fill (0, 4) circle (2pt);
      \draw (0, 4) node[left]{$x'$};

      \draw [thick, ->, rel] (0, 2) to (0, 4);
      \draw (0, 3) node[left]{$\IRel$};

      \fill (2, 4) circle (2pt);
      \draw (2, 4) node[right]{$y' \SatIK p,\,\not \SatIK q$};

      \draw [thick, ->, squig, rel] (0, 4) to (2, 4);
      \draw (1, 4) node[above]{$\MRel$};

      \end{tikzpicture}
      \caption{The model $\Model{M}_{\name{CD}}$.}
      \label{fig:cd-countermodel}
    \end{figure}

    Here, we have $x \SatIK \Box (p \lor q)$ since $x \IRel x' \MRel y' \SatIK p$ and there is no $y$ such that $x \MRel y$.
    Also, we have $x \not \SatIK \Box q$ by $x \IRel x' \MRel y' \not \SatIK q$.
    However, we have $x \not \SatIK \Diamond p$ since there is no $y$ such that $x \MRel y$.
    Therefore, $x \not \SatIK \name{CD}$.
  \end{proof}
\end{proposition}

\begin{corollary}
  $\Logic{IK} \nparallel \Logic{LIK}$, that is, neither $\Logic{IK} \subseteq \Logic{LIK}$ nor $\Logic{LIK} \subseteq \Logic{IK}$ hold.
  In particular, (1) $\Logic{LIK} \nvdash \name{FS2}$ and (2) $\Logic{IK} \nvdash \name{CD}$.
  \begin{proof}
    (1) The falsifying model $\Model{M}_{\name{FS2}}$ of $\name{FS2}$, which is given in Proposition \ref{prop:fik-is-weaker-than-ik} (2),
    trivially satisfies the $\Box$-p condition, so it follows from the soundness of $\Logic{LIK}$.

    (2) The falsifying model $\Model{M}_{\name{CD}}$ of $\name{CD}$, which is given in Proposition \ref{prop:fik-is-weaker-than-lik} (2),
    trivially satisfies the FS2 condition, so it follows from the soundness of $\Logic{IK}$.
  \end{proof}
\end{corollary}

Like $\Logic{FIK}$, the logic $\Logic{LIK}$ also satisfies most of the required properties for ``intuitionistic'' modal logics,
and also does not satisfy Simpson's sixth requirement,
as it does not prove $\neg \Diamond \varphi \to \Box \neg \varphi$,
which, again, corresponds to $\neg \exists x \varphi(x) \to \forall x \neg \varphi(x)$.
An important difference to $\Logic{FIK}$ is that it does not also satisfy Simpson's fourth requirement, that is, $\Logic{LIK}$ now lacks DP.

\begin{proposition} \label{prop:lik-is-mostly-intuitionistic}
  (1) $\Logic{LIK} + \varphi \lor \neg \varphi = \Logic{K}$.
  (2) $\Logic{LIK} \nvdash \name{Dual}_\Box$ and $\Logic{LIK} \nvdash \name{Dual}_\Diamond$.
  (3) $\Logic{LIK} \nvdash \neg \Diamond \varphi \to \Box \neg \varphi$.
  \begin{proof}
    (1) is similarly to Proposition \ref{prop:fik-is-mostly-intuitionistic} (2).
    (2) The countermodel for Proposition \ref{prop:ik-non-duality} satisfies the $\Box$-p and $\Diamond$-p conditions.
    (3) The countermodel for Proposition \ref{prop:fik-lacks-negdia-to-boxneg} satisfies the $\Box$-p and $\Diamond$-p conditions.
  \end{proof}
\end{proposition}

\begin{proposition}[{\cite{balbiani_local_2024}}] \label{prop:lik-lacks-dp}
  $\Logic{LIK}$ lacks DP.
  \begin{proof}
    $\Diamond \top \lor \Box \bot$ is a theorem of $\Logic{LIK}$ by the following derivation:
    \begin{center}
      \AxiomC{$ $}
      \UnaryInfC{$\top \lor \bot$}
      \RightLabel{\scriptsize{$\name{M}_\Box$, $\name{N}_\Box$}}
      \UnaryInfC{$\Box (\top \lor \bot)$}

      \AxiomC{$ $}
      \RightLabel{\scriptsize{$\name{CD}$}}
      \UnaryInfC{$\Box (\top \lor \bot) \to (\Diamond \top \lor \Box \bot)$}

      \RightLabel{\scriptsize{MP}}
      \BinaryInfC{$\Diamond \top \lor \Box \bot$}
      \DisplayProof
    \end{center}
    It is easy to see that neither $\Diamond \top$ nor $\Box \bot$ is.
  \end{proof}
\end{proposition}

Also, the $\Diamond$-free fragment of $\Logic{LIK}$ is no longer contained in that of $\Logic{IK}$.

\begin{proposition} \label{prop:lik-box-fragment}
  $\Logic{IK} \cap \MFb \nparallel \Logic{LIK} \cap \MFb$.
  \begin{proof}
    The formula $\neg \neg \Box \varphi \to \Box \neg \neg \varphi$ is provable in $\Logic{IK}$ by Proposition \ref{prop:ik-nnbox-to-boxnn} but not in $\Logic{LIK}$:
    the countermodel in Proposition \ref{prop:fik-lacks-negdia-to-boxneg}
    satisfies both the $\Box$-p and $\Diamond$-p conditions and falsifies
    $\neg \neg \Box \neg p \to \Box \neg \neg \neg p$.
    Also, the formula $\Box (\varphi \lor \psi) \to (\neg \Box \neg \varphi \lor \Box \psi)$ is provable in $\Logic{LIK}$,
    which easily follows from $\name{CD}$ and $\Box \neg \varphi \to \neg \Diamond \varphi$,
    but not in $\Logic{IK}$, the countermodel for which is visualized in Figure \ref{fig:boxed-cd-countermodel}.
    \begin{figure}[htb]
      \centering
      \begin{tikzpicture}
        \fill (0, 0) circle (2pt);
        \draw (0, 0) node[left]{$x$};

        \fill (0, 2) circle (2pt);
        \draw (0, 2) node[left]{$x_1'$};
        \draw [thick, ->, rel] (0, 0) to (0, 2);
        \draw (0, 1) node[left]{$\IRel$};

        \fill (2, 2) circle (2pt);
        \draw (2, 2) node[left]{$x_2'$};
        \draw [thick, ->, rel] (0, 0) to (2, 2);
        \draw (1, 1) node[right]{$\IRel$};

        \fill (4, 2) circle (2pt);
        \draw (4, 2) node[right]{$y_2' \SatIK p,\,\not \SatIK q$};
        \draw [thick, ->, rel, squig] (2, 2) to (4, 2);
        \draw (3, 2) node[above]{$\MRel$};
      \end{tikzpicture}
      \caption{\centering The countermodel for Proposition \ref{prop:lik-box-fragment}.}
      \label{fig:boxed-cd-countermodel}
    \end{figure}%
  \end{proof}
\end{proposition}

Now we have a spectrum of logics:
\begin{equation*}
  \Logic{CK} \subsetneq \Logic{WK} \subsetneq \Logic{IK}^- \subsetneq \Logic{FIK} \subsetneq \Logic{IK},
\end{equation*}
with $\Logic{LIK}$ being an outlier in that $\Logic{FIK} \subsetneq \Logic{LIK}$ and $\Logic{IK} \nparallel \Logic{LIK}$.
For the $\Diamond$-free fragments of the above logics, we have:
\begin{equation*}
  \Logic{iK}_\Box = \Logic{CK} \cap \MFb = \Logic{WK} \cap \MFb = \Logic{IK}^- \cap \MFb \subsetneq \Logic{FIK} \cap \MFb \subsetneq \Logic{IK} \cap \MFb,
\end{equation*}
again with $\Logic{FIK} \cap \MFb \subsetneq \Logic{LIK} \cap \MFb$ and $\Logic{IK} \cap \MFb \nparallel \Logic{LIK} \cap \MFb$.
For the $\Box$-free fragments, so far, we have:
\begin{equation*}
  \Logic{CK} \cap \MFd \subsetneq \Logic{WK} \cap \MFd \subsetneq \Logic{iK}_\Diamond = \Logic{IK}^- \cap \MFd = \Logic{FIK} \cap \MFd.
\end{equation*}
In the following propositions, We give the missing separations between $\Logic{FIK} \cap \MFd$, $\Logic{LIK} \cap \MFd$, and $\Logic{IK} \cap \MFd$.

\begin{proposition}
  (1) $\Logic{LIK} \vdash \neg \neg \Diamond \top \to \Diamond \top$.
  (2) $L \nvdash \neg \neg \Diamond \top \to \Diamond \top$ for $L \in \set{ \Logic{FIK}, \Logic{IK} }$.
  \begin{proof}
    (1) Take any $\Box$-p and $\Diamond$-p model $\Model{M} = (W, \IRel, \MRel, V)$,
    and any $x \in W$ such that $x \SatIK_\Model{M} \neg \neg \Diamond \top$.
    By the $\IRel$-reflexivity of $x$ and the fact that no world forces $\bot$,
    the assumption implies $x \not \SatIK_\Model{M} \neg \Diamond \top$.
    Hence there are $x', y' \in W$ such that $x \IRel x'$, $x' \MRel y'$, and $y' \SatIK_\Model{M} \top$.
    By the $\Box$-p condition, there is $y \in W$ such that $x \MRel y$ and $y \IRel y'$.
    Since $y \SatIK_\Model{M} \top$, we have $x \SatIK_\Model{M} \Diamond \top$.
    Therefore, $\Logic{LIK} \vdash \neg \neg \Diamond \top \to \Diamond \top$ by completeness.

    (2) Consider the model $\Model{M}_{\name{CD}}$ from Proposition \ref{prop:fik-is-weaker-than-lik} (2).
    The world $x$ has no $\MRel$-successor, so $x \not \SatIK \Diamond \top$.
    To see that $x \SatIK \neg \neg \Diamond \top$, take any $z \in W$ such that $x \IRel z$.
    If $z = x'$, then $z \SatIK \Diamond \top$ because $x' \MRel y'$.
    If $z = x$, then its $\IRel$-successor $x'$ forces $\Diamond \top$.
    Thus, in either case, $z \not \SatIK \neg \Diamond \top$.
    By $\Model{M}_{\name{CD}}$ satisfying the $\Diamond$-p and FS2 conditions,
    we have $L \nvdash \neg \neg \Diamond \top \to \Diamond \top$ for $L \in \set{ \Logic{FIK}, \Logic{IK} }$ by soundness.
  \end{proof}
\end{proposition}

\begin{proposition}
  (1) $\Logic{IK} \vdash \Diamond \neg \neg \varphi \to \neg \neg \Diamond \varphi$.
  (2) $L \nvdash \Diamond \neg \neg \varphi \to \neg \neg \Diamond \varphi$ for $L \in \set{ \Logic{FIK}, \Logic{LIK} }$.
  \begin{proof}
    (1) Take any $\Diamond$-p and FS2 model $\Model{M} = (W, \IRel, \MRel, V)$,
    and any $x \in W$ such that $x \SatIK_\Model{M} \Diamond \neg \neg \varphi$.
    Then there is $y \in W$ such that $x \MRel y$ and $y \SatIK_\Model{M} \neg \neg \varphi$.
    To show that $x \SatIK_\Model{M} \neg \neg \Diamond \varphi$, take any $x' \in W$ such that $x \IRel x'$.
    By the $\Diamond$-p condition applied to $y \MRelRev x \IRel x'$, there is $u \in W$ such that $x' \MRel u$ and $y \IRel u$.
    By persistency, $u \SatIK_\Model{M} \neg \neg \varphi$, and hence $u \not \SatIK_\Model{M} \neg \varphi$.
    Therefore, there is $u' \in W$ such that $u \IRel u'$ and $u' \SatIK_\Model{M} \varphi$.
    By the FS2 condition applied to $x' \MRel u \IRel u'$, there is $x'' \in W$ such that $x' \IRel x''$ and $x'' \MRel u'$.
    Thus $x'' \SatIK_\Model{M} \Diamond \varphi$, and so $x' \not \SatIK_\Model{M} \neg \Diamond \varphi$.
    Since $x'$ was arbitrary, $x \SatIK_\Model{M} \neg \neg \Diamond \varphi$.
    Therefore, $\Logic{IK} \vdash \Diamond \neg \neg \varphi \to \neg \neg \Diamond \varphi$ by completeness.

    (2) Consider the model $\Model{M}_{\name{FS2}}$ from Proposition \ref{prop:fik-is-weaker-than-ik} (2).
    Here, $x \MRel y \IRel y' \SatIK p$ implies $y \SatIK \neg \neg p$, so $x \SatIK \Diamond \neg \neg p$.
    On the other hand, $y \not \SatIK p$ implies $x \not \SatIK \Diamond p$, so $x \not \SatIK \neg \neg \Diamond p$.
    By $\Model{M}_{\name{FS2}}$ satisfying the $\Diamond$-p and $\Box$-p conditions,
    we have $L \nvdash \Diamond \neg \neg \varphi \to \neg \neg \Diamond \varphi$ for $L \in \set{ \Logic{FIK}, \Logic{LIK} }$ by soundness.
  \end{proof}
\end{proposition}

\begin{corollary} \label{cor:fik-ik-lik-diamond-fragment-separation}
  $\Logic{FIK} \cap \MFd \subsetneq \Logic{LIK} \cap \MFd$,
  $\Logic{FIK} \cap \MFd \subsetneq \Logic{IK} \cap \MFd$,
  and $\Logic{LIK} \cap \MFd \nparallel \Logic{IK} \cap \MFd$.
  In particular, $\Theory{V}_\Diamond^e \subsetneq \Theory{V}^e \cap \MFd$.
\end{corollary}

The resulting spectra of the $\Diamond$/$\Box$-free fragments are visualized in
Figure~\ref{fig:basic-unary-fragments}.

\begin{figure}[H]
  \centering
  \begin{subfigure}[t]{.49\textwidth}
    \centering
    \begin{tikzpicture}[font=\scriptsize,>=stealth]
      \path[use as bounding box] (-2.75,-0.75) rectangle (2.75,4.55);
      \node[draw,rounded corners,align=center,text width=5.2cm,inner sep=4pt] (b0) at (0,0)
        {$\Logic{iK}_\Box$
         \\$= \Logic{IK}^- \cap \MFb$
         \\$= \Logic{WK} \cap \MFb$
         \\$= \Logic{CK} \cap \MFb$
         };
      \node[draw,rounded corners,align=center,text width=3.2cm,inner sep=4pt] (b1) at (0,2.1)
        {$\Logic{FIK} \cap \MFb$};
      \node[draw,rounded corners,align=center,text width=2.3cm,inner sep=4pt] (b2) at (-1.55,4.2)
        {$\Logic{LIK} \cap \MFb$};
      \node[draw,rounded corners,align=center,text width=2.3cm,inner sep=4pt] (b3) at (1.55,4.2)
        {$\Logic{IK} \cap \MFb$};
      \draw[thick,->] (b0) -- (b1);
      \draw[thick,->] (b1) -- (b2);
      \draw[thick,->] (b1) -- (b3);
      \node at (0,4.2) {$\nparallel$};
    \end{tikzpicture}
    \caption{$\Diamond$-free (box-only) fragments.}
  \end{subfigure}%
  \begin{subfigure}[t]{.49\textwidth}
    \centering
    \begin{tikzpicture}[font=\scriptsize,>=stealth]
      \path[use as bounding box] (-2.75,-0.75) rectangle (2.75,4.55);
      \node[draw,rounded corners,align=center,text width=3.2cm,inner sep=4pt] (d0) at (0,0)
        {$\Logic{CK} \cap \MFd$};
      \node[draw,rounded corners,align=center,text width=3.2cm,inner sep=4pt] (d1) at (0,1.1)
        {$\Logic{WK} \cap \MFd$};
      \node[draw,rounded corners,align=center,text width=5.2cm,inner sep=4pt] (d2) at (0,2.55)
        {$\Logic{iK}_\Diamond$
         \\$= \Logic{FIK} \cap \MFd$
         \\$= \Logic{IK}^- \cap \MFd$
        };
      \node[draw,rounded corners,align=center,text width=2.3cm,inner sep=4pt] (d3) at (-1.55,4.2)
        {$\Logic{LIK} \cap \MFd$};
      \node[draw,rounded corners,align=center,text width=2.3cm,inner sep=4pt] (d4) at (1.55,4.2)
        {$\Logic{IK} \cap \MFd$};
      \draw[thick,->] (d0) -- (d1);
      \draw[thick,->] (d1) -- (d2);
      \draw[thick,->] (d2) -- (d3);
      \draw[thick,->] (d2) -- (d4);
      \node at (0,4.2) {$\nparallel$};
    \end{tikzpicture}
    \caption{$\Box$-free (diamond-only) fragments.}
  \end{subfigure}
  \caption{Visualization of the separation results.
    Each arrow denotes strict inclusion; equal logics share a node, and the two top nodes in each panel are incomparable.}
  \label{fig:basic-unary-fragments}
\end{figure}

De Groot et al.\ \cite{groot_semantical_2025} treated a wide range of logics that may or may not fit
in this spectrum of logics.
They analyzed logics in the form $\Logic{CK} + \name{Ax}$, where $\name{Ax} \subseteq \set{ \name{N}_\Diamond, \name{C}_\Diamond, \name{FS2} }$,
proving the strong completeness of these logics and identifying the condition when the $\Diamond$-free fragment of $\Logic{CK} + \name{Ax}$ is exactly $\Logic{iK}_\Box$.

\begin{theorem}[{\cite[Thm.\,VII.3.]{groot_semantical_2025}}]
  $(\Logic{CK} + \name{Ax}) \cap \MFb = \Logic{iK}_\Box$ if and only if $\name{N}_\Diamond \notin \name{Ax}$ or $\name{FS2} \notin \name{Ax}$.
\end{theorem}

\begin{problem}[{\cite[Q.\,VII.4.]{groot_semantical_2025}}]
  Does $(\Logic{CK} + \name{N}_\Diamond + \name{FS2}) \cap \MFb = \Logic{IK} \cap \MFb$ hold?
\end{problem}

\begin{problem}
  Perform a similar analysis for the $\Box$-free fragments of $\Logic{CK} + \name{Ax}$.
\end{problem}

They also analyzed the cases when $\name{FS2}$ is replaced with $\name{wCD}$,
and when $\name{N}_\Diamond$ is replaced with a weaker variant $\name{N}_{\Diamond\Box} \defeq \Diamond \bot \to \Box \bot$.

\begin{problem}
  Let $\name{Ax} \subseteq \set{ \name{N}_\Diamond, \name{N}_{\Diamond\Box}, \name{C}_\Diamond, \name{FS2}, \name{wCD}, \name{CD} }$,
  then when does $\Logic{CK} + \name{Ax} + \varphi \lor \neg \varphi$ collapse to $\Logic{K}$?.
  We note that $\name{Ax} = \set{ \name{N}_\Diamond, \name{wCD} }$ is one such example that indeed collapses (see Proposition \ref{prop:fik-is-mostly-intuitionistic} (2)).
\end{problem}

Balbiani and Gencer \cite{balbiani_intuitionistic_2026} also studied
borderline logics by considering the combinations of various confluence conditions,
forward ($\Diamond$-p), backward ($\mathrm{FS2}$), downward ($\Box$-p), and upward: $({\MRel} \cdot {\IRelRev}) \subseteq ({\IRelRev} \cdot {\MRel})$,
and weaker ``quasi'' variants of the four confluence conditions.
They proved that only some of them are modally definable,
and axiomatized the logics generated from the combinations of forward, backward, and downward confluences,
further extending the spectrum of logics presented in this section.

\begin{problem}
  Let $\Logic{IK}^c \defeq \Logic{IK} + \name{CD}$, where the superscript ``c'' stands for ``crisp''.
  This logic is introduced by Balbiani and Gencer \cite{balbiani_intuitionistic_2026} under the name $\Logic{L}_{fbdc}$,
  and is a proper sublogic of the crisp G\"{o}del modal logic $\Logic{GK}^c$ by Rodriguez and Vidal \cite{rodriguez_axiomatization_2021}.
  $\Logic{IK}^c$ is known to be sound and complete with respect to $\FrameClass{\Box\text{-p}} \cap \FrameClass{\Diamond\text{-p}} \cap \FrameClass{FS2}$ \cite[Prop.\,4]{balbiani_intuitionistic_2026}.
  Since $\Logic{IK} \subsetneq \Logic{IK}^c$, there is $\varphi \in \MF$ such that $\Logic{IK}^c \vdash \varphi$ but $\forall x\,\mathrm{ST}_x(\varphi)$ is not intuitionistically valid in every first-order structure for $\Lang_\mathrm{ST}$.
  Determine the first-order axiom that would recover the correspondence by considering every first-order structure that validates it.
\end{problem}

\begin{problem}
  Does $\Logic{IK}^-$, $\Logic{FIK}$, $\Logic{LIK}$, and $\Logic{IK}^c$ enjoy CIP, LIP, and UIP?
  The first two logics are particularly interesting, as $\Logic{WK}$ enjoys all of them and $\Logic{IK}$ enjoys none of them.
\end{problem}

\section{Extensions and related logics} \label{sec:extensions}

In this section, we shall take a look at several extensions of CMLs, IMLs, and the $\Diamond$/$\Box$-free logics $\Logic{iK}_\Box$ and $\Logic{iK}_\Diamond$, as well as some related logics discussed in the literature.
Recall that $\Box$ and $\Diamond$ are not interdefinable in an intuitionistic setting,
so many of the intuitionistic modal axioms below are defined in the form $\name{A} = (\name{A}_\Box \land \name{A}_\Diamond)$.

\begin{table}[H]
  \centering
  \setlength{\extrarowheight}{3pt}
  \caption{List of more modal axioms.}
  \label{tab:more-modal-axioms}
  \begin{tabular}{rl|rl|rl }
    $\name{D}_\Box$ & $\neg \Box \bot$ & $\name{D}_\Diamond$ & $\Diamond \top$ & $\name{D}$ & $\Box \varphi \to \Diamond \varphi$ \\
    $\name{T}_\Box$ & $\Box \varphi \to \varphi$ & $\name{T}_\Diamond$ & $\varphi \to \Diamond \varphi$ & $\name{T}$ & $\name{T}_\Box \land \name{T}_\Diamond$ \\
    $\name{B}_\Box$ & $\Diamond \Box \varphi \to \varphi$ & $\name{B}_\Diamond$ & $\varphi \to \Box \Diamond \varphi$ & $\name{B}$ & $\name{B}_\Diamond \land \name{B}_\Box$ \\
    $\name{4}_\Box$ & $\Box \varphi \to \Box \Box \varphi$ & $\name{4}_\Diamond$ & $\Diamond \Diamond \varphi \to \Diamond \varphi$ & $\name{4}$ & $\name{4}_\Box \land \name{4}_\Diamond$ \\
    $\name{5}_\Box$ & $\Diamond \Box \varphi \to \Box \varphi$ & $\name{5}_\Diamond$ & $\Diamond \varphi \to \Box \Diamond \varphi$ & $\name{5}$ & $\name{5}_\Diamond \land \name{5}_\Box$ \\
    $\name{L}_\Box$ & $\Box (\Box \varphi \to \varphi) \to \Box \varphi$ \\
    $\name{S}_\Box$ & $\varphi \to \Box \varphi$ \\
  \end{tabular}\\
\end{table}

\begin{definition} \label{def:modal-cube-extensions} Let us define the 17 modal axioms as in Table \ref{tab:more-modal-axioms}.
  \begin{itemize}
    \item For $L \in \set{ \Logic{C}, \Logic{W}, \Logic{I} }$ and $X \subseteq \set{ \name{D}, \name{T}, \name{B}, \name{4}, \name{5} }$,
          we denote by $L\Logic{K}X$ the logic obtained from $L$ by adding all the axioms in $X$.
          If $X = \set{\name{T}, \name{4}}$ or $X = \set{\name{T}, \name{5}}$, then we write $L\Logic{S4}$ and $L\Logic{S5}$ to mean $L\Logic{KT4}$ and $L\Logic{KT5}$, respectively.
          For example, $\Logic{CK45} = \Logic{CK} + \name{4} + \name{5}$ and $\Logic{IS5} = \Logic{IK} + \name{T} + \name{5}$\footnote{$\Logic{IS5}$ is sometimes referred to as $\Logic{MIPQ}$, which is a logic first appeared in Prior's book \cite{prior_1957} and later shown to be equivalent to $\Logic{IS5}$ by Fischer Servi \cite{fischer_servi_modal_1977}.}.
    \item For $X \subseteq \set{ \name{D}, \name{T}, \name{4}, \name{5} }$ and $\heartsuit \in \set{ \Box, \Diamond }$,
          we denote by $\Logic{iK}X_\heartsuit$ the logic obtained from $\Logic{iK}_\heartsuit$ by adding the axiom $A_\heartsuit$ for all $A \in X$.
          For example, $\Logic{iKD}_\Box = \Logic{iK}_\Box + \name{D}_\Box$ and $\Logic{iKT}_\Diamond = \Logic{iK}_\Diamond + \name{T}_\Diamond$.
          We also let $\Logic{iS4}_\Box = \Logic{iKT4}_\Box$ and $\Logic{iS4}_\Diamond = \Logic{iKT4}_\Diamond$.
  \end{itemize}
\end{definition}

The first two subsection are devoted to the logics defined above;
the remaining two axioms $\name{L}_\Box$ (stands for \emph{L\"ob}) and $\name{S}_\Box$ (stands for \emph{strong}) will be treated later.
Here, the deontic axiom $\name{D}$ is distinct in that it does not have the form $\name{D}_\Box \land \name{D}_\Diamond$.
We can easily make the following observations:

\begin{proposition} \label{prop:d-ddia-derivations}
  $\Logic{CKD} \dashv \vdash \Logic{CK} + \name{D}_\Diamond$.
  \begin{proof}
    The following derivations prove both directions.
    \begin{center}
      \AxiomC{$ $}
      \RightLabel{\scriptsize{$\name{N}_\Box$}}
      \UnaryInfC{$\Box \top$}

      \AxiomC{$ $}
      \RightLabel{\scriptsize{$\name{D}$}}
      \UnaryInfC{$\Box \top \to \Diamond \top$}

      \RightLabel{\scriptsize{MP}}
      \BinaryInfC{$\Diamond \top$}
      \DisplayProof
    \end{center}

    \begin{center}
      \AxiomC{$ $}
      \RightLabel{\scriptsize{$\name{D}_\Diamond$}}
      \UnaryInfC{$\Diamond \top$}

      \AxiomC{$ $}
      \UnaryInfC{$\top \to (\varphi \to \varphi)$}
      \RightLabel{\scriptsize{$\name{M}_\Diamond$}}
      \UnaryInfC{$\Diamond \top \to \Diamond (\varphi \to \varphi)$}

      \RightLabel{\scriptsize{MP}}
      \BinaryInfC{$\Diamond (\varphi \to \varphi)$}

      \AxiomC{$ $}
      \RightLabel{\scriptsize{$\name{FS1}$}}
      \UnaryInfC{$\Diamond (\varphi \to \varphi) \to (\Box \varphi \to \Diamond \varphi)$}

      \RightLabel{\scriptsize{MP}}
      \BinaryInfC{$\Box \varphi \to \Diamond \varphi$}
      \DisplayProof
    \end{center}
  \end{proof}
\end{proposition}

\begin{proposition} \label{prop:d-dbox-derivations}
  (1) $\Logic{WKD} \vdash \name{D}_\Box$.
  (2) $\Logic{LIK} + \name{D}_\Box \vdash \name{D}$.
  (3) $\Logic{CKD} \nvdash \name{D}_\Box$.\\
  (4) $\Logic{IK} + \name{D}_\Box \nvdash \name{D}$.
  \begin{proof}
    The following derivations prove (1) and (2), respectively.
    \begin{center}
      \AxiomC{$ $}
      \RightLabel{\scriptsize{$\name{D}$}}
      \UnaryInfC{$\Box \bot \to \Diamond \bot$}

      \AxiomC{$ $}
      \RightLabel{\scriptsize{$\name{N}_\Diamond$}}
      \UnaryInfC{$\Diamond \bot \to \bot$}

      \RightLabel{\scriptsize{Syll.}}
      \BinaryInfC{$\Box \bot \to \bot$}
      \DisplayProof
    \end{center}

    \begin{center}
      \AxiomC{$ $}
      \UnaryInfC{$\varphi \to (\varphi \lor \bot)$}
      \RightLabel{\scriptsize{$\name{M}_\Box$}}
      \UnaryInfC{$\Box \varphi \to \Box (\varphi \lor \bot)$}

      \AxiomC{$ $}
      \RightLabel{\scriptsize{$\name{CD}$}}
      \UnaryInfC{$\Box(\varphi \lor \bot) \to (\Diamond \varphi \lor \Box \bot)$}

      \RightLabel{\scriptsize{MP}}
      \BinaryInfC{$\Box \varphi \to (\Diamond \varphi \lor \Box \bot)$}

      \AxiomC{$ $}
      \RightLabel{\scriptsize{$\name{D}_\Box$}}
      \UnaryInfC{$\Box \bot \to \bot$}
      \UnaryInfC{$\Box \bot \to \Diamond \varphi$}

      \RightLabel{\scriptsize{Syll.}}
      \BinaryInfC{$\Box \varphi \to \Diamond \varphi$}
      \DisplayProof
    \end{center}
    (3) and (4) are left for the reader as easy exercises. See also Figure \ref{fig:d-countermodels}.
  \end{proof}
\end{proposition}

\begin{figure}[h]
  \centering
  \begin{subfigure}{.5\textwidth}
    \centering
    \tikzset{
      cross/.pic = {
      \draw[rotate = 45] (-#1,0) -- (#1,0);
      \draw[rotate = 45] (0,-#1) -- (0, #1);
      }
    }
    \begin{tikzpicture}
    \fill (0, 0) circle (2pt);
    \draw (0, 0) node[right]{$x \not\SatCK \bot$};

    \draw (2, 2) circle (2pt);
    \draw (2, 2) node[right]{$y \SatCK \bot$};
    \draw (2, 2) pic {cross=2pt};

    \draw [thick, ->, rel, squig] (0,0) to (2,2);
    \draw (1, 1) node[above left]{$\MRel$};
    \end{tikzpicture}
    \caption*{The countermodel for (3).}
  \end{subfigure}%
  \begin{subfigure}{.5\textwidth}
    \centering
    \begin{tikzpicture}
    \fill (0, 0) circle (2pt);
    \draw (0, 0) node[right]{$x$};

    \draw (0, 2) circle (2pt);
    \draw (0, 2) node[right]{$x' \SatIK p, {} \not\SatIK \bot$};

    \draw [thick, ->, rel] (0,0) to (0,2);
    \draw (0, 1) node[left]{$\IRel$};
    \end{tikzpicture}
    \caption*{The countermodel for (4).}
  \end{subfigure}%

  \caption{
    Visualization of the countermodels for Proposition \ref{prop:d-dbox-derivations}.
    $\otimes$ denotes a $\MRel$-reflexive fallible world.
    $\circ$ denotes a $\MRel$-reflexive world.
  }
  \label{fig:d-countermodels}
\end{figure}

Thus, already over $\Logic{CK}$, $\name{D}_\Diamond$ and $\name{D}$ are interderivable,
while $\name{D}_\Box$ is much weaker than (or, in $\Logic{CK}$, even incomparable to) $\name{D}$.
Despite that, $\name{D}_\Box$ is conventionally used as the deontic axiom for $\Diamond$-free logics,
which we will see later in this section.

\subsection{The modal cubes and the collapse phenomena}

Let $X \subseteq \set{ \name{D}, \name{T}, \name{B}, \name{4}, \name{5} }$,
then, for each $L \in \set{ \Logic{C}, \Logic{W}, \Logic{I} }$, there are initially $2^5$ axiom sets, and hence $2^5$ presentations of the form $L\Logic{K}X$.
However, as in the case of classical modal logic, many of them coincide,
resulting in 15 different extensions for each of $\Logic{CK}$, $\Logic{WK}$, and $\Logic{IK}$,
forming the intuitionistic and constructive versions of the modal ``cube''.

\begin{proposition} \label{prop:ck-cube} \leavevmode
  (1) $\Logic{CKT} \vdash \name{D}$.
  (2) $\Logic{CKB4} \vdash \name{5}$.
  (3) $\Logic{CKB5} \vdash \name{4}$.
  (4) $\Logic{CKB4D} \vdash \name{T}$.
  (5) $\Logic{CS5} \vdash \name{B}$.
  \begin{proof}
    (1) $\name{D}_\Diamond$ easily follows from $\name{T}_\Diamond$, then use Proposition \ref{prop:d-ddia-derivations}.
    (2) Follows from the below derivations:
    \begin{center}
      \AxiomC{$ $}
      \RightLabel{\scriptsize{$\name{4}_\Box$}}
      \UnaryInfC{$\Box \varphi \to \Box \Box \varphi$}
      \RightLabel{\scriptsize{$\name{M}_\Diamond$}}
      \UnaryInfC{$\Diamond \Box \varphi \to \Diamond \Box \Box \varphi$}

      \AxiomC{$ $}
      \RightLabel{\scriptsize{$\name{B}_\Box$}}
      \UnaryInfC{$\Diamond \Box \Box \varphi \to \Box \varphi$}

      \RightLabel{\scriptsize{Syll.}}
      \BinaryInfC{$\Diamond \Box \varphi \to \Box \varphi$}
      \DisplayProof
    \end{center}
    \begin{center}
      \AxiomC{$ $}
      \RightLabel{\scriptsize{$\name{B}_\Diamond$}}
      \UnaryInfC{$\Diamond \varphi \to \Box \Diamond \Diamond \varphi$}

      \AxiomC{$ $}
      \RightLabel{\scriptsize{$\name{4}_\Diamond$}}
      \UnaryInfC{$\Diamond \Diamond \varphi \to \Diamond \varphi$}
      \RightLabel{\scriptsize{$\name{M}_\Box$}}
      \UnaryInfC{$\Box \Diamond \Diamond \varphi \to \Box \Diamond \varphi$}

      \RightLabel{\scriptsize{Syll.}}
      \BinaryInfC{$\Diamond \varphi \to \Box \Diamond \varphi$}
      \DisplayProof
    \end{center}
    (3) Follows from the below derivations:
    \begin{center}
      \AxiomC{$ $}
      \RightLabel{\scriptsize{$\name{B}_\Diamond$}}
      \UnaryInfC{$\Box \varphi \to \Box \Diamond \Box \varphi$}

      \AxiomC{$ $}
      \RightLabel{\scriptsize{$\name{5}_\Box$}}
      \UnaryInfC{$\Diamond \Box \varphi \to \Box \varphi$}
      \RightLabel{\scriptsize{$\name{M}_\Box$}}
      \UnaryInfC{$\Box \Diamond \Box \varphi \to \Box \Box \varphi$}

      \RightLabel{\scriptsize{Syll.}}
      \BinaryInfC{$\Box \varphi \to \Box \Box \varphi$}
      \DisplayProof
    \end{center}
    \begin{center}
      \AxiomC{$ $}
      \RightLabel{\scriptsize{$\name{5}_\Diamond$}}
      \UnaryInfC{$\Diamond \varphi \to \Box \Diamond \varphi$}
      \RightLabel{\scriptsize{$\name{M}_\Diamond$}}
      \UnaryInfC{$\Diamond \Diamond \varphi \to \Diamond \Box \Diamond \varphi$}

      \AxiomC{$ $}
      \RightLabel{\scriptsize{$\name{B}_\Box$}}
      \UnaryInfC{$\Diamond \Box \Diamond \varphi \to \Diamond \varphi$}

      \RightLabel{\scriptsize{Syll.}}
      \BinaryInfC{$\Diamond \Diamond \varphi \to \Diamond \varphi$}
      \DisplayProof
    \end{center}
    (4) One can easily derive $\name{T}_\Diamond$ from $\name{B}_\Diamond$ and $\name{T}_\Box$ by syllogism,
    so it suffices to show $\Logic{CKB4D} \vdash \name{T}_\Box$.
    The following derivation proves it:
    \begin{center}
      \AxiomC{$ $}
      \RightLabel{\scriptsize{$\name{4}_\Box$}}
      \UnaryInfC{$\Box \varphi \to \Box \Box \varphi$}

      \AxiomC{$ $}
      \RightLabel{\scriptsize{$\name{D}$}}
      \UnaryInfC{$\Box \Box \varphi \to \Diamond \Box \varphi$}

      \RightLabel{\scriptsize{Syll.}}
      \BinaryInfC{$\Box \varphi \to \Diamond \Box \varphi$}

      \AxiomC{$ $}
      \RightLabel{\scriptsize{$\name{B}_\Box$}}
      \UnaryInfC{$\Diamond \Box \varphi \to \varphi$}

      \RightLabel{\scriptsize{Syll.}}
      \BinaryInfC{$\Box \varphi \to \varphi$}
      \DisplayProof
    \end{center}
    (5) Follows from the below derivations:
    \begin{center}
      \AxiomC{$ $}
      \RightLabel{\scriptsize{$\name{5}_\Box$}}
      \UnaryInfC{$\Diamond \Box \varphi \to \Box \varphi$}

      \AxiomC{$ $}
      \RightLabel{\scriptsize{$\name{T}_\Box$}}
      \UnaryInfC{$\Box \varphi \to \varphi$}

      \RightLabel{\scriptsize{Syll.}}
      \BinaryInfC{$\Diamond \Box \varphi \to \varphi$}
      \DisplayProof
    \end{center}
    \begin{center}
      \AxiomC{$ $}
      \RightLabel{\scriptsize{$\name{T}_\Diamond$}}
      \UnaryInfC{$\varphi \to \Diamond \varphi$}

      \AxiomC{$ $}
      \RightLabel{\scriptsize{$\name{5}_\Diamond$}}
      \UnaryInfC{$\Diamond \varphi \to \Box \Diamond \varphi$}

      \RightLabel{\scriptsize{Syll.}}
      \BinaryInfC{$\varphi \to \Box \Diamond \varphi$}
      \DisplayProof
    \end{center}%
  \end{proof}
\end{proposition}

\begin{corollary}[{\cite{strasburger_cut_2013,strassburger_nested_2015}}] \label{cor:modal-cubes-fifteen}
  For $L \in \set{ \Logic{C}, \Logic{W}, \Logic{I} }$, there are exactly 15 different logics in the form $L\Logic{K}X$, where $X \subseteq \set{ \name{D}, \name{T}, \name{B}, \name{4}, \name{5} }$.
  \begin{proof}
    It is well known that there are exactly 15 different classical modal logics in the form $\Logic{K}X$ (cf.~\cite{CZ97}).
    So, by $\Logic{CK} \subseteq \Logic{K}$, there are at least 15 different logics in the form $\Logic{CK}X$.
    On the other hand, by Proposition \ref{prop:ck-cube}, there are at most 15 such logics.
    The rest follows from $\Logic{CK} \subseteq \Logic{WK} \subseteq \Logic{IK} \subseteq \Logic{K}$.
  \end{proof}
\end{corollary}

\begin{figure}[htb]
  \centering
  \tikzset{up/.style={
    shorten >=8pt, shorten <=8pt,
  }}
  \tikzset{upright/.style={
    shorten >=10pt, shorten <=10pt,
  }}
  \tikzset{right/.style={
    shorten >=12pt, shorten <=12pt,
  }}

  \begin{tikzpicture}[scale=0.6]
    \draw (0, 0) node{$\Logic{K}$};
    
    \draw (0, 3) node{$\Logic{KD}$};
    \draw [thick, up] (0, 0) to (0, 3);

    \draw (0, 6) node{$\Logic{KT}$};
    \draw [thick, up] (0, 3) to (0, 6);

    \draw (2, 2) node{$\Logic{K4}$};
    \draw [upright] (0, 0) to (2, 2);
    
    \draw (2, 5) node{$\Logic{KD4}$};
    \draw [upright] (0, 3) to (2, 5);
    \draw [dashed, up] (2, 2) to (2, 5);

    \draw (2, 8) node{$\Logic{S4}$};
    \draw [thick, upright] (0, 6) to (2, 8);
    \draw [dashed, up] (2, 5) to (2, 8);

    \draw (6, 0) node{$\Logic{KB}$};
    \draw [thick, right] (0, 0) to (6, 0);
    
    \draw (6, 3) node{$\Logic{KDB}$};
    \draw [thick, up] (6, 0) to (6, 3);
    \draw [thick, right, shorten >=16pt] (0, 3) to (6, 3);

    \draw (6, 6) node{$\Logic{KTB}$};
    \draw [thick, up] (6, 3) to (6, 6);
    \draw [thick, right, shorten >=16pt] (0, 6) to (6, 6);

    \draw (3, 1) node{$\Logic{K5}$};
    \draw [right] (0, 0) to (3, 1);

    \draw (4.5, 2) node{$\Logic{K45}$};
    \draw [dashed, right] (2, 2) to (4.5, 2);
    \draw [upright] (3, 1) to (4.5, 2);

    \draw (3, 4) node{$\Logic{KD5}$};
    \draw [up] (3, 1) to (3, 4);
    \draw [upright, shorten >=16pt] (0, 3) to (3, 4);

    \draw (4.5, 5) node{$\Logic{KD45}$};
    \draw [dashed, right, shorten <=16pt, shorten >=18pt] (2, 5) to (4.5, 5);
    \draw [upright] (3, 4) to (4.5, 5);
    \draw [dashed, up] (4.5, 2) to (4.5, 5);

    \draw (8, 2) node{$\Logic{KB5}$};
    \draw [thick, upright] (6, 0) to (8, 2);
    \draw [dashed, right, shorten <=14pt, shorten >=16pt] (4.5, 2) to (8, 2);
    
    \draw (8, 8) node{$\Logic{S5}$};
    \draw [thick, up] (8, 2) to (8, 8);
    \draw [thick, upright] (6, 6) to (8, 8);
    \draw [thick, right] (2, 8) to (8, 8);
    \draw [dashed, up, shorten >=14pt] (4.5, 5) to [out=75,in=195](8, 8);
  \end{tikzpicture}
  \caption{The ``cube'' of classical modal logics. The thin and dashed lines are purely for ease of viewing.}
  \label{fig:modal-cube}
\end{figure}

Moreover, Pacheco recently proved in his preprint \cite{pacheco_collapsing_2024} that the logics $\Logic{CKB}$ and $\Logic{IKB}$ coincide,
collapsing the right half of CML and IML cubes.
In particular, $\Logic{CKB}$ proves the $\name{C}_\Diamond$ axiom.
From the viewpoint adopted here, this may call into question the ``constructiveness'' of the $\name{B}$ axiom.

\begin{theorem}[The collapse phenomena, {\cite[Lem.\,20]{pacheco_collapsing_2024}}] \label{thm:ckb-collapse}
  If $\Logic{IKB} \vdash \varphi$, then $\Logic{CKB} \vdash \varphi$.
  Consequently, $\Logic{CK}X = \Logic{WK}X = \Logic{IK}X$ for $X \supseteq \set{ \name{B} }$.
\end{theorem}

$\Logic{IK}X$ ($X \subseteq \set{ \name{D}, \name{T}, \name{B}, \name{4}, \name{5} }$) has been studied extensively by means of IK-frames and IK-models.
Fischer Servi \cite{fischer_axiomatizations_1984} introduced $\Logic{IK}$ and explicitly proved the completeness of $\Logic{IK}$, $\Logic{IKT}$, $\Logic{IKTB}$, $\Logic{IS4}$, and $\Logic{IS5}$.
Simpson \cite[Thm.\,3.3.4]{simpson_proof_1994} states the result for $\Logic{IK}X$ in general,
noting that these cases appear explicitly in Fischer Servi and that the remaining cases are straightforward.

\begin{theorem}[{\cite{fischer_axiomatizations_1984,simpson_proof_1994}}] \label{thm:ik-cube-completeness}
  Associate $\name{D}$, $\name{T}$, $\name{B}$, $\name{4}$, and $\name{5}$, respectively, with
  seriality, reflexivity, symmetry, transitivity, and Euclideanness of $\MRel$.
  Then, for $X \subseteq \set{ \name{D}, \name{T}, \name{B}, \name{4}, \name{5} }$,
  $\Logic{IK}X \vdash \varphi$ iff $\varphi$ is valid in every IK-model whose underlying frame satisfies the relational conditions associated with the members of $X$.
\end{theorem}

We shall note that, unlike in the classical case, the frame conditions in Theorem \ref{thm:ik-cube-completeness} are sufficient but are not, in general, necessary to validate the respective axioms;
in particular, there is an IK-frame that is not $\MRel$-reflexive but still validates $\name{T}$.
The following example is due to Simpson \cite{simpson_proof_1994}.

\begin{example}[{\cite{simpson_proof_1994}}] \label{ex:non-reflexive-ik-frame-that-validates-t}
  Let $\Model{F} = (\set{x, y}, \IRel, \MRel)$ be an IK-frame,
  where ${\IRel}$ is the reflexive closure of $\set{ (x, y), (y, x) }$
  and ${\MRel} = \set{ (x, y), (y, x) }$,
  then $\Model{F} \ValidIK \name{T}$. Here, the frame $\Model{F}$, visualized in Figure \ref{fig:non-reflexive-ik-frame-that-validates-t}, is clearly nowhere $\MRel$-reflexive.
  \begin{figure}[H]
    \centering
    \begin{tikzpicture}
    \fill (0, 0) circle (2pt);
    \draw (0, 0) node[left]{$x$};

    \fill (2, 2) circle (2pt);
    \draw (2, 2) node[right]{$y$};

    \draw [thick, <->, rel] (0,0) to [out=90, in=180](2,2);
    \draw (0.3, 1.7) node{$\IRel$};

    \draw [thick, <->, decorate, decoration={
      snake,
      pre=lineto,
      post=lineto,
      pre length=0.5cm,
      post length=0.2cm,
      amplitude=1
    }, shorten <=4pt, shorten >=4pt] (0,0) to [out=0, in=270](2, 2);
    \draw (1.7, 0.3) node{$\MRel$};
      
    \end{tikzpicture}
    \caption{The frame $\Model{F}$ in Example \ref{ex:non-reflexive-ik-frame-that-validates-t}.}
    \label{fig:non-reflexive-ik-frame-that-validates-t}
  \end{figure}
\end{example}

Plotkin and Stirling \cite{plotkin_framework_1986} obtained the necessary and sufficient
frame condition on IK-frames for the so-called \emph{Geach axioms}.

\begin{theorem}[{\cite[Thm.\,2.1]{plotkin_framework_1986}}] \label{thm:ik-geach-frame-condition}
  Let $\name{G}(k, l, m, n) \defeq \Diamond^k \Box^l \varphi \to \Box^m \Diamond^n \varphi$.
  An IK-frame $\Model{F} = (W, \IRel, \MRel)$ validates $\name{G}(k, l, m, n)$ iff
  \begin{equation*}
    w \MRel^k x \ \tand\  w \MRel^m y \implies \exists x' \, \exists z \, (x \IRel x' \tand x' \MRel^l z \tand y \MRel^n z).
  \end{equation*}
  Here, $w_0 \MRel^n w_n$ is an abbreviation of $\exists w_1, \ldots, w_{n-1} \, (w_0 \MRel w_1 \MRel \ldots w_{n-1} \MRel w_n)$.
  In particular, $x \MRel^0 y \defsiff x = y$.
\end{theorem}

Compared with the IML cube, the corresponding extensions of $\Logic{CK}$ and $\Logic{WK}$ seem to have received less systematic semantic treatment.
Dalmonte \cite{dalmonte_wijesekera-style_2022} introduced Gentzen-style sequent calculi and neighborhood semantics for
$\Logic{WK}$, $\Logic{WKT}$, and $\Logic{WKD}$, and for their non-normal variants in which
the basic modal axioms and rules are weakened, and proved that the calculi and semantics characterize the same logics.
In addition, he proved that these logics are decidable and enjoy CIP.
Alechina et al.\ \cite{alechina_categorical_2001} proved soundness and completeness of $\Logic{CS4}$ with respect to
the class of all $\MRel$-preordered CK-frames. They also introduced algebraic and categorical semantics that are sound and complete for $\Logic{CS4}$,
and also discussed the relationship between $\Logic{CS4}$ and the propositional lax logic $\Logic{PLL}$ \cite{fairtlough_propositional_1997}, which we will see later in this section.
Recently, Balbiani et al.\ \cite{balbiani_constructive_2021} proved that $\Logic{CS4}$ is decidable by showing the finite model property of it.

\begin{problem} \label{prob:cml-cube-decidability}
  For $L \in \set{\Logic{C}, \Logic{W}}$,
  are $L\Logic{K4}$, $L\Logic{K5}$, $L\Logic{K45}$, $L\Logic{KD4}$, and $L\Logic{KD45}$ decidable?
  Note that the remaining logics in the CML cube collapse into IMLs.
\end{problem}

\begin{problem}[\cite{pacheco_collapsing_2024}] \label{prob:ck-geach-frame-condition}
  Determine a frame condition on CK-frames that is necessary and sufficient to validate $\name{G}(k, l, m, n)$.
\end{problem}

Much of the cited results for CMLs uses Gentzen-style calculi,
whereas the IML results discussed below use labelled sequent calculi.
Simpson \cite{simpson_proof_1994} introduced a cut-free labelled calculus for $\Logic{IK}X$ where $X \subseteq \set{ \name{D}, \name{T}, \name{B} }$,
obtaining the decidability of these logics as a consequence.
Girlando et al.\ \cite{girlando_intuitionistic_2023} claimed a proof of the decidability of $\Logic{IS4}$ based on
a labelled calculus for it \cite{marin_fully_2021},
but it was later announced that their proof contains a gap\footnote{https://filipendule.github.io, https://sites.google.com/site/kuznets/publications}.

\begin{problem}[{\cite{simpson_proof_1994, girlando_intuitionistic_2023}}] \label{prob:iml-cube-decidability}
  Are $\Logic{IK4}$, $\Logic{IK5}$, $\Logic{IK45}$, $\Logic{IKD4}$, $\Logic{IKD45}$, and $\Logic{IKB5}$ decidable?
  Girlando et al.\ \cite{girlando_intuitionistic_2023} claim that their method proves the decidability of $\Logic{IK4}$.
  Also, Piazza states in his preprint \cite{piazza_kruskal_2026} that $\Logic{IK4}$ is decidable.
  The decidability of $\Logic{IS5}$ ($= \Logic{MIPQ}$) is already proven by Mints \cite{minc_on_1971}.
\end{problem}

For nested sequent calculi, Straßburger \cite{strasburger_cut_2013} treated the 15 extensions of $\Logic{IK}$,
while Arisaka et al.\ \cite{strassburger_nested_2015} treated the corresponding constructive extensions over $\Logic{CK}$.
Lyon \cite{lyon_nested_2021} discussed a method called \emph{structural refinement} to obtain nested calculi from labelled ones for a large class of IMLs.
Gao et al.\ \cite{gao_bi-nested_2025} discussed bi-nested calculi for $\Logic{IK}$ that allows one to construct a countermodel directly from a failed derivation.

\begin{problem} \label{prob:modal-cube-local-tabularity}
  Let $L$ be a logic with a language $\mathscr{L}$.
  For $\varphi, \psi \in \mathscr{L}$, let us write $\varphi \equiv_L \psi$ to mean $\varphi \leftrightarrow \psi \in L$, then $\equiv_L$ is clearly an equivalence relation.
  Then, $L$ is said to be \emph{locally tabular} if for any finite $P \subseteq \PropVar$,
  the set $T(P) = \set*{ \varphi; \Var(\varphi) \subseteq P } / \mathord{\equiv_L}$ is finite.
  For example, $\Logic{Cl}$ is locally tabular, while $\Logic{Int}$ is not.
  Here, it is known by Diego's theorem \cite{diego_sur_1966} that the implicational fragment of $\Logic{Int}$ is locally tabular.
  It is also known that $\Logic{K}X$ ($X \supseteq \set{ \name{5} }$) is locally tabular (e.g.\ \cite{CZ97}).
  Now, is the implicational fragment of $L\Logic{K}X$ ($L \in \set{ \Logic{C}, \Logic{W}, \Logic{I} }$, $X \supseteq \set{ \name{5} }$) locally tabular?
\end{problem}

For the borderline logics, Balbiani et al.\ \cite{balbiani_constructive_2021} proved the decidability of $\Logic{S4I} = \Logic{LIK} + \name{T} + \name{4}$.
Balbiani et al.\ \cite{balbiani_local_2024}
obtained completeness of $\Logic{LIK} + X$ $(X \subseteq \set{ \name{D}, \name{T}, \name{4} })$
and decidability of $\Logic{LIK} + \name{D}$ and $\Logic{LIK} + \name{T}$.
Balbiani and Gencer \cite{balbiani_intuitionistic_2026}
studied extensions of many borderline logics with the $\name{T}$, $\name{B}$, and $\name{4}$ axioms,
obtaining completeness and decidability, the latter of which is restricted to the combinations of $\name{T}$ and $\name{B}$.

\subsection{The $\Diamond$/$\Box$-free modal squares}

Similarly to the modal ``cubes'' for IMLs and CMLs, the $\Diamond$/$\Box$-free extensions up to $\Logic{iS4}_\heartsuit$ ($\heartsuit \in \set{ \Box, \Diamond }$) are well studied.
For the sake of convenience, we shall refer to them as modal ``squares''.
Došen \cite{dosen_models_1985} proved the completeness of the logics in these $\Diamond$/$\Box$-free modal squares with respect to certain classes of f-frames.

\begin{definition}[{\cite[Defs.\,1 and 2]{dosen_models_1985}}] \label{def:modal-square-frame-conditions}
  Let $(\heartsuit, \alpha) \in \set{ (\Box, \beta), (\Diamond, \delta) }$.
  Let $\Model{F} = (W, \IRel, \MRel)$ be a $\heartsuit$-p f-frame,
  and $\alpha\Model{F} = (W, \IRel, \MRel_\alpha)$ be its $\heartsuit$-condensation.
  We say $\Model{F}$ is \emph{$\MRel_\alpha$-serial} if $\MRel_\alpha$ is serial.
  \emph{$\MRel_\alpha$-reflexivity} and \emph{$\MRel_\alpha$-transitivity} are defined similarly.
\end{definition}

\begin{proposition}[{\cite[Lems.\,5 and 15]{dosen_models_1985}}] \label{prop:serial-iff-condensed-is-serial}
  Let $\Model{F} = (W, \IRel, \MRel)$ be an f-frame.
  \begin{enumerate}[label=(\arabic*)]
    \item If $\Model{F}$ is $\Box$-p, then $\MRel$ is serial iff $\MRel_\beta$ is serial.
    \item If $\Model{F}$ is $\Diamond$-p, then $\MRel$ is serial iff $\MRel_\delta$ is serial.
  \end{enumerate}
  \begin{proof}
    Both clauses follow from the $\IRel$-reflexivity of $\Model{F}$.
  \end{proof}
\end{proposition}

We shall note that the analogue of Proposition \ref{prop:serial-iff-condensed-is-serial} for $\MRel$-reflexivity or $\MRel$-transitivity
doesn't hold (Exercise).

\begin{theorem}[{\cite[Thms.\,1 and 3]{dosen_models_1985}}] \label{thm:ik-modal-square-completeness}
  Let $(\heartsuit, \alpha) \in \set{ (\Box, \beta), (\Diamond, \delta) }$.
  \begin{itemize}
    \item $\Logic{iKD}_\heartsuit \vdash \varphi$ iff $\varphi$ is e-valid on every $\heartsuit$-p f-frame that is $\MRel_\alpha$-serial.
    \item $\Logic{iKT}_\heartsuit \vdash \varphi$ iff $\varphi$ is e-valid on every $\heartsuit$-p f-frame that is $\MRel_\alpha$-reflexive.
    \item $\Logic{iK4}_\heartsuit \vdash \varphi$ iff $\varphi$ is e-valid on every $\heartsuit$-p f-frame that is $\MRel_\alpha$-transitive.
    \item $\Logic{iKD4}_\heartsuit \vdash \varphi$ iff $\varphi$ is e-valid on every $\heartsuit$-p f-frame that is $\MRel_\alpha$-serial and $\MRel_\alpha$-transitive.
    \item $\Logic{iS4}_\heartsuit \vdash \varphi$ iff $\varphi$ is e-valid on every $\heartsuit$-p f-frame that is $\MRel_\alpha$-reflexive and $\MRel_\alpha$-transitive.
  \end{itemize}
\end{theorem}

We note that Došen \cite{dosen_models_1985} did not explicitly prove the $\Logic{iK4}_\heartsuit$ and $\Logic{iKD4}_\heartsuit$ cases of Theorem \ref{thm:ik-modal-square-completeness},
but they easily follow from his proof, where he proved the $\MRel_\alpha$-transitivity of canonical f-frames for $\Logic{iS4}_\heartsuit$
without using its $\MRel_\alpha$-reflexivity.

Ono \cite{ono_intuitionistic_1977} proposed four different $\Diamond$-free weakenings of the $\name{5}$ axiom,
all of which are equivalent to $\name{5}$ over the classical $\Logic{S4}$.

\begin{definition}[{\cite{ono_intuitionistic_1977}}] \label{def:diamond-free-5s}
  Let
  $\name{5.1}_\Box \defeq \neg \Box \varphi \to \Box \neg \Box \varphi$,
  $\name{5.2}_\Box \defeq (\Box \varphi \to \Box \psi) \to \Box (\Box \varphi \to \Box \psi)$,
  $\name{5.3}_\Box \defeq \Box (\Box \varphi \lor \psi) \to (\Box \varphi \lor \Box \psi)$, and
  $\name{5.4}_\Box \defeq \Box \varphi \lor \Box \neg \Box \varphi$.
\end{definition}

He discussed the proper inclusion relationship between the extensions of $\Logic{iS4}_\Box$ with one or more axioms in Definition \ref{def:diamond-free-5s},
and proved the completeness and the finite model property of some of these logics, extending the work of Bull \cite{bull_modal_1965} on the finite model property of $\Logic{iS4}_\Box$.
Bull and Ono use the same semantics; Ono calls its structures \emph{M-frames} and \emph{M-models}.
This semantics is not the f-frame semantics used throughout this survey.
Moreover, Ono obtained the following result on the $\Diamond$-free fragment of $\Logic{IS5}$.

\begin{proposition}[{\cite[Lem.\,4.4]{ono_intuitionistic_1977}}] \label{prop:is5-box-fragment}
  $\Logic{IS5} \cap \Lang_\Box = \Logic{iS4}_\Box + \name{5.2}_\Box$.
\end{proposition}

Došen \cite{dosen_models_1985} also discussed $\Logic{iS4}_\Box + \name{5.1}_\Box$ ($\Box\Logic{HS5}$ in his paper),
$\Logic{iS4}_\Box + \name{5.2}_\Box$ ($\Box\Logic{HS5.1}$),
$\Logic{iS4}_\Diamond + \name{5.1}_\Diamond: \Diamond \neg \Diamond \varphi \to \neg \Diamond \varphi$ ($\Diamond\Logic{HS5}$), and
$\Logic{iS4}_\Diamond + \name{5.2}_\Diamond: \Diamond (\Diamond \varphi \to \Diamond \psi) \to (\Diamond \varphi \to \Diamond \psi)$ ($\Diamond\Logic{HS5.1}$),
and proved the completeness of these logics with respect to certain classes of f-frames.

\begin{problem} \label{prob:is5-diamond-fragment}
  Does $\Logic{IS5} \cap \Lang_\Diamond = \Logic{iS4}_\Diamond + \name{5.2}_\Diamond$ hold?
\end{problem}

\begin{problem} \label{prob:unary-modal-cubes}
  What would be the weaker $\Diamond$/$\Box$-free variants of the $\name{B}$ axiom?
  Moreover, can we obtain the modal ``cubes'' of $\Diamond$/$\Box$-free logics with these weaker versions of $\name{5}$ and $\name{B}$ we have mentioned?
\end{problem}

Most of the recent results discussed next concern $\Diamond$-free extensions, especially provability logics.
For recent studies on the logics in the modal ``square'',
Van der Giessen \cite[Cor.\,3.4.3--3.4.4]{giessen_uniform_admissible_2022} proved the failure of UIP in $\Logic{iS4}_\Box$ and $\Logic{iK4}_\Box$,
De Groot and Shillito \cite{de_groot_intuitionistic_2025} gave a topological semantics that is sound and complete for $\Logic{iS4}_\Box$,
and Iemhoff \cite{iemhoff_terminating_2018,iemhoff_uniform_2019} proved that $\Logic{iKD}_\Box$ is decidable and enjoys CIP and UIP.

The logic $\Logic{iS4}_\Box$ also has type-theoretic applications through the Curry--Howard correspondence.
Davies and Pfenning \cite{davies_modal_2001} use a type system based on $\Logic{iS4}_\Box$ to analyze staged computation.
Nanevski, Pfenning, and Pientka \cite{nanevski_contextual_2008} develop contextual modal type theory,
a relativized extension of $\Logic{iS4}_\Box$ in which code may depend on an explicitly recorded context.

\begin{problem} \label{prob:ikd-box-lip}
  Does $\Logic{iKD}_\Box$ enjoy LIP?
\end{problem}

\begin{problem}[\cite{iemhoff_uniform_2019}] \label{prob:ik4-is4-interpolation}
  Do $\Logic{iK4}_\Box$ and $\Logic{iS4}_\Box$ enjoy CIP and LIP?
\end{problem}

\subsection{Provability logics}

\emph{Provability logic} studies the $\Box$-formulas that are valid
when $\Box$ is read as a \emph{provability predicate} $\operatorname{Pr}_T(x)$ for an arithmetical theory $T$.
Given a formula $\varphi \in \MFb$,
an \emph{arithmetical realization} replaces each propositional variable with an arithmetical sentence
and uses $\operatorname{Pr}_T(x)$ to interpret $\Box \psi$ as ``$\psi$ is provable in $T$.''
Then $\varphi$ belongs to the provability logic of $\operatorname{Pr}_T(x)$
if the resulting arithmetical sentence is provable in $T$
regardless of the choice of realization.
This is analogous to Kripke semantics:
an arithmetical realization plays the role of a valuation,
while a provability predicate provides an arithmetical ``implementation'' of the modality,
much as an accessibility relation provides a relational one.

\begin{definition}[e.g.\ \cite{artemov_provability_2005}] \label{def:provability-logic}
  Let $T$ be an arithmetical theory with a computable set of axioms
  and let $\operatorname{Pr}_T(x)$ be a provability predicate of $T$.
  An \emph{arithmetical realization} $*$ of $\MFb$ maps each propositional variable to an arithmetical sentence,
  commutes with the propositional connectives, and satisfies
  \begin{equation*}
    (\Box \varphi)^* \defeq \operatorname{Pr}_T(\ulcorner \varphi^* \urcorner).
  \end{equation*}
  The \emph{provability logic of $\operatorname{Pr}_T(x)$} consists of the formulas $\varphi \in \MFb$
  such that $T \vdash \varphi^*$ for every arithmetical realization $*$.
\end{definition}

The best-known such logic is $\Logic{GL} = \Logic{K4} + \name{L}_\Box$,
the provability logic of the standard provability predicate for Peano arithmetic $\Logic{PA}$;
its characteristic axiom
$\name{L}_\Box: \Box(\Box \varphi \to \varphi) \to \Box \varphi$
is a modal form of L\"ob's theorem.
As $\Logic{PA}$ is based on classical first-order logic, it is natural to consider intuitionistic provability logics
when studying provability predicates in intuitionistic arithmetic.
In the standard arithmetical provability language considered here, $\Box$ is primitive
and $\Diamond$, when used, is defined as $\neg \Box \neg$ (\cite[Remark 5.1]{litak_lewis_2018}).
It is therefore natural to first consider the following $\Diamond$-free logics.

\begin{definition} \label{def:iglbox-islbox}
  We define
  \begin{equation*}
    \Logic{iGL}_\Box \defeq \Logic{iK4}_\Box + \name{L}_\Box
    \qquad\text{and}\qquad
    \Logic{iSL}_\Box \defeq \Logic{iGL}_\Box + \name{S}_\Box.
  \end{equation*}
\end{definition}

There are two reasons to mention $\Logic{iSL}_\Box$.
The axiom $\name{S}_\Box: \varphi \to \Box \varphi$ is called the \emph{completeness principle};
under an arithmetical realization, its validity requires
$T \vdash A \to \operatorname{Pr}_T(\ulcorner A \urcorner)$ for every arithmetical sentence $A$.
Thus, following Visser \cite{visser_completenes_1982}, one may say that $T$ proves its own completeness.
As such, $\name{S}_\Box$ is sometimes referred to as the $\name{CP}$ axiom (e.g.\ \cite{litak_lewis_2018}).
Moreover, the classical logic $\Logic{GL} + \name{S}_\Box$ is ``degenerate'' in that it proves $\Box \bot$,
whereas $\Logic{iSL}_\Box$ does not.

As with the other $\Diamond$-free logics, $\Logic{iGL}_\Box$ and $\Logic{iSL}_\Box$
have a Kripke-like semantics based on f-frames, f-models, and the extrinsic satisfaction relation.

\begin{definition} \label{def:iglbox-frame-conditions}
  Let $\Model{F} = (W, \IRel, \MRel)$ be an f-frame, and put ${\MRel_\beta} \defeq (\MRel \cdot \IRel)$.
  When $\Model{F}$ is $\Box$-p, this is the modal relation of its $\Box$-condensation $\beta \Model{F}$.
  \begin{itemize}
    \item We write
          $\operatorname{Up}_{\IRel}(W) \defeq \set*{ A \subseteq W; \forall x,y\, (x \in A \tand x \IRel y \imp y \in A) }$
          for the set of all $\IRel$-upsets of $W$.
    \item We say $\Model{F}$ is \emph{$\MRel_\beta$-Noetherian}
          if for any $A \in \operatorname{Up}_{\IRel}(W)$ such that $A \ne W$,
          there is $w_A \in W \setminus A$ such that $\forall v \in W\, (w_A \MRel_\beta v \imp v \in A)$.
    \item We say $\Model{F}$ is \emph{$\MRel_\beta$-strong} if ${\MRel_\beta} \subseteq {\IRel}$,
          and is \emph{$\MRel$-strong} if ${\MRel} \subseteq {\IRel}$.
  \end{itemize}
\end{definition}

Here, the Noetherian condition on $\MRel_\beta$ says that every proper $\IRel$-upset has a $\MRel_\beta$-maximal point outside it,
which is an analogue of the converse well-foundedness condition for $\Logic{GL}$.
Also, much like seriality, the strongness conditions on $\MRel_\beta$ and $\MRel$ coincide.

\begin{proposition} \label{prop:sbox-frame-condition}
  Let $\Model{F} = (W, \IRel, \MRel)$ be an f-frame.
  \begin{enumerate}[label=(\arabic*)]
    \item $\Model{F}$ is $\MRel$-strong iff it is $\MRel_\beta$-strong.
    \item If $\Model{F}$ is $\Box$-p, then $\Model{F}$ is $\MRel$-strong iff
          $\Model{F} \ValidE \varphi \to \Box \varphi$ for every $\varphi \in \MFb$.
  \end{enumerate}
  \begin{proof}
    (1)
    ($\Rightarrow$)
    If $x \MRel_\beta y$, then there is $z \in W$ such that $x \MRel z \IRel y$.
    Then $x \IRel z \IRel y$ by $\MRel$-strongness, so $x \IRel y$ by the transitivity of $\IRel$.
    Thus ${\MRel_\beta} \subseteq {\IRel}$.
    ($\Leftarrow$) Trivial from the reflexivity of $\IRel$.

    (2)
    ($\Rightarrow$) Suppose that $\Model{F}$ is $\MRel$-strong, and let $\Model{M}$ be an f-model based on $\Model{F}$.
    Take $x, x' \in W$ and $\varphi \in \MFb$ such that $x \IRel x'$ and $x' \SatE_\Model{M} \varphi$.
    For every $y \in W$ such that $x' \MRel y$, $\MRel$-strongness gives $x' \IRel y$.
    Since $\Model{F}$ is $\Box$-p, $\SatE_\Model{M}$ is $\Box$-persistent by Proposition~\ref{prop:extrinsic-persistency},
    so $y \SatE_\Model{M} \varphi$.
    Therefore, $x' \SatE_\Model{M} \Box \varphi$, and so $x \SatE_\Model{M} \varphi \to \Box \varphi$.
    ($\Leftarrow$) Suppose that $\Model{F}$ is not $\MRel$-strong,
    then there are $x, y \in W$ such that $x \MRel y$ and $x \not \IRel y$.
    Take a fresh variable $p^* \in \PropVar$, then define a valuation $V$ by
    $w \in V(p) \defsiff (p \ne p^* \tor x \IRel w)$,
    where each $V(p)$ is clearly upward closed with respect to $\IRel$,
    and let $\Model{M} = (\Model{F}, V)$.
    Then $x \SatE_\Model{M} p^*$ and $y \nVdash^e_\Model{M} p^*$.
    Since $x \MRel y$, we have $x \nVdash^e_\Model{M} \Box p^*$,
    and consequently $x \nVdash^e_\Model{M} p^* \to \Box p^*$.
  \end{proof}
\end{proposition}

\begin{theorem}[{\cite{ursini_modal_1979,litak_constructive_2014}}] \label{thm:iglbox-completeness}
  Let $\varphi \in \MFb$. Then:
  \begin{enumerate}[label=(\arabic*)]
    \item $\Logic{iGL}_\Box \vdash \varphi$ iff $\Model{F} \ValidE \varphi$ for every $\Box$-p, $\MRel_\beta$-transitive and $\MRel_\beta$-Noetherian f-frame $\Model{F}$.
    \item $\Logic{iSL}_\Box \vdash \varphi$ iff $\Model{F} \ValidE \varphi$ for every $\Box$-p, $\MRel_\beta$-transitive, $\MRel_\beta$-Noetherian, and $\MRel_\beta$-strong f-frame $\Model{F}$.
  \end{enumerate}
\end{theorem}

By Proposition~\ref{prop:sbox-frame-condition}, the $\MRel_\beta$-strongness condition in the second clause is precisely
the $\MRel$-strongness condition corresponding to $\name{S}_\Box$.

It is well known that classical $\Logic{GL}$ enjoys all of CIP, LIP, and UIP (cf.\ \cite{kurahashi_uniform_2020}),
which motivates the study of the corresponding properties for its intuitionistic variants.
As in the other $\Diamond$-free logics, proof-theoretic methods have been particularly useful for
$\Logic{iGL}_\Box$ and $\Logic{iSL}_\Box$.

\begin{theorem}[\cite{giessen_sequent_2021,shillito_new_2023,feree_mechanised_2024,giessen_igl_uniform_2026}] \label{thm:iglbox-islbox-proof-theory}
  $\Logic{iGL}_\Box$ and $\Logic{iSL}_\Box$ have cut-admissible Gentzen-style sequent calculi.
  Also, both logics enjoy CIP and UIP.
\end{theorem}

\begin{problem} \label{prob:iglbox-islbox-lip}
  Do $\Logic{iGL}_\Box$ and $\Logic{iSL}_\Box$ enjoy LIP?
\end{problem}

As in classical $\Logic{GL}$, the $\name{4}_\Box$ axiom follows from L\"ob's principle.

\begin{proposition} \label{prop:iglbox-proves-four}
  $\Logic{iK}_\Box + \name{L}_\Box \vdash \name{4}_\Box$.
  Consequently,
  $\Logic{iGL}_\Box = \Logic{iK}_\Box + \name{L}_\Box$
  and $\Logic{iK4}_\Box \subsetneq \Logic{iGL}_\Box$.
\end{proposition}

There is, however, an important difference:
although $\Logic{GL}$ is the provability logic of the standard provability predicate for $\Logic{PA}$,
$\Logic{iGL}_\Box$ is not the provability logic of the standard provability predicate for Heyting arithmetic $\Logic{HA}$,
the intuitionistic analogue of $\Logic{PA}$;
the latter logic contains Leivant's principle
$\Box(\varphi \lor \psi) \to \Box(\Box \varphi \lor \psi)$,
which is not derivable in $\Logic{iGL}_\Box$ \cite{litak_lewis_2018}.
Still, $\Logic{iGL}_\Box$ is regarded as the basic logic for provability predicates in intuitionistic arithmetic.
$\Logic{iSL}_\Box$ and its extensions are used to investigate the completeness principle (e.g.\ \cite{visser_completenes_1982})
or $\Sigma_1$-provability (e.g.\ \cite{mojtahedi_sigma_2024}).
Also, Mojtahedi recently announced in his preprint \cite{mojtahedi_provability_2026}
that a certain extension $\Logic{iGLH}$ of $\Logic{iGL}_\Box$
is the provability logic of the standard provability predicate for $\Logic{HA}$.
The reader may consult \cite[Section 4]{visser_problems_2006} for open problems on the provability logics of intuitionistic arithmetic.

Variants of $\Logic{iGL}_\Box$ and $\Logic{iSL}_\Box$ have also been adopted in type theory.
Nakano \cite{nakano_modality_2000} introduces a \emph{modality for recursion} in a modal typing system with recursive types.
Birkedal, Møgelberg, Schwinghammer, and Støvring \cite{birkedal_first_2012} use a \emph{later operator}
to guard recursive definitions of terms, predicates, and types.
In these approaches, the modality ensures the existence and uniqueness of fixed points for suitably guarded type expressions.
Litak \cite[Author's Cut, Sec.\,3]{litak_constructive_2014} and Litak and Visser \cite[Sec.\,7.2]{litak_lewis_2018} provide an entry point to this literature
and relate these applications to intuitionistic provability logics.

The logics discussed above are formulated in $\MFb$ because $\Box$ is primitive
under the provability interpretation.
One may also ask for an intuitionistic version of classical $\Logic{GL}$
in which both $\Box$ and $\Diamond$ are primitive.
The behavior of $\Diamond$ is not determined by that of $\Box$,
since their classical duality is not intuitionistically valid.
Das, van der Giessen, and Marin \cite{das_intuitionistic_2024} approach this question
by applying Simpson's methodology for IMLs to classical $\Logic{GL}$.
Recall that the standard translation regards modal logic as a fragment of first-order logic,
with $\Box$ and $\Diamond$ corresponding to universal and existential quantification over the accessibility relation.
Simpson's approach in \cite[Sections 2.2 and 5.2]{simpson_proof_1994} evaluates the resulting first-order formulas in what he calls \emph{IL-models},
that is, Kripke models for intuitionistic first-order logic.
For the language of the standard translation, such a model assigns to each intuitionistic world $w$
an ordinary first-order structure
\begin{equation*}
  \mathscr{M}_w = (D_w, \MRel_w, (V_{p,w})_{p \in \PropVar})
\end{equation*}
so that, for any $w \IRel w'$,
\begin{equation*}
  D_w \subseteq D_{w'}, \qquad
  {\MRel_w} \subseteq {\MRel_{w'}}, \qquad
  V_{p,w} \subseteq V_{p,w'} \quad (p \in \PropVar).
\end{equation*}
Thus, if $\rho$ is an assignment into $D_w$ and $y$ is fresh,
the first-order forcing clauses applied to the standard translation give
\begin{align*}
  w, \rho \Vdash \mathrm{ST}_x(\Box \varphi)
  &\tiff
    \forall w' \IRelRev w\, \forall d \in D_{w'}\,
    \bigl(\rho(x) \MRel_{w'} d \imp
    (w', \rho[y \mapsto d]) \Vdash \mathrm{ST}_y(\varphi)\bigr), \\
  w, \rho \Vdash \mathrm{ST}_x(\Diamond \varphi)
  &\tiff
    \exists d \in D_w\,
    \bigl(\rho(x) \MRel_w d \tand
    (w, \rho[y \mapsto d]) \Vdash \mathrm{ST}_y(\varphi)\bigr).
\end{align*}
The clause for $\Box$ therefore considers every later world and its enlarged domain,
whereas the witness for $\Diamond$ belongs to the domain of the current world.

Das, van der Giessen, and Marin \cite[Section 6]{das_intuitionistic_2024}
develop the Kripke predicate semantics for $\Logic{IGL}$
by requiring every $\MRel_w$ to be transitive.
Moreover, there must be no infinite path alternating between
retaining a domain element while moving to a later world
and taking an $\MRel_w$-step within the current domain, as illustrated in Figure \ref{fig:igl-il-model}.

\begin{figure}[h]
  \centering
  \begin{tikzpicture}[x=1.8cm, y=1.35cm]
    \draw[rounded corners] (-0.5,-0.35) rectangle (2.5,0.35);
    \draw (-0.5,0) node[left]{$D_{w_0}$};
    \fill (0,0) circle (2pt);
    \draw (0,0) node[below]{$d_0$};

    \draw[rounded corners] (-0.5,0.65) rectangle (2.5,1.35);
    \draw (-0.5,1) node[left]{$D_{w_1}$};
    \fill (0,1) circle (2pt);
    \draw (0,1) node[above left]{$d_0$};
    \fill (2,1) circle (2pt);
    \draw (2,1) node[above right]{$d_1$};

    \draw[rounded corners] (-0.5,1.65) rectangle (2.5,2.35);
    \draw (-0.5,2) node[left]{$D_{w_2}$};
    \fill (0,2) circle (2pt);
    \draw (0,2) node[above left]{$d_0$};
    \fill (2,2) circle (2pt);
    \draw (2,2) node[above right]{$d_1$};

    \draw[thick, ->, rel] (0,0) to (0,1);
    \draw (-0.05,0.5) node[left]{$\IRel$};
    \draw[thick, ->, rel, squig] (0,1) to (2,1);
    \draw (1,1) node[below]{$\MRel_{w_1}$};
    \draw[thick, ->, rel] (2,1) to (2,2);
    \draw (2.05,1.5) node[right]{$\IRel$};
  \end{tikzpicture}
  \caption{The beginning of an alternating path in an IL-model, where $w_0 \IRel w_1 \IRel w_2$.}
  \label{fig:igl-il-model}
\end{figure}

In the corresponding birelational model, each pair $(w,d)$ is regarded as a world,
and the two kinds of steps above become the intuitionistic and modal relations, respectively.
Thus, the condition in the birelational semantics is that
the composite of the intuitionistic and modal relations is converse well-founded
\cite[Sections 4 and 6]{das_intuitionistic_2024}.

\begin{definition}[{\cite[Def.\,4.7]{das_intuitionistic_2024}}] \label{def:igl-via-semantics}
  We say an IK-frame $\Model{F} = (W, \IRel, \MRel)$ is \emph{composite Noetherian}
  if $({\IRel} \cdot {\MRel})$ is converse well-founded,
  that is, if there are no sequences $(x_n)_{n < \omega}$ and $(y_n)_{n < \omega}$ in $W$
  such that $x_n \IRel y_n \MRel x_{n+1}$ for every $n < \omega$.
  Since $\IRel$ is reflexive, $\MRel$ is also converse well-founded on every composite Noetherian frame.
  An IK-frame is said to be an \emph{IGL-frame} if it is composite Noetherian and $\MRel$-transitive.
  We let $\mathcal{B}\Logic{IGL}$ be the set of all formulas in $\MF$ that are h-valid on every IGL-frame.
\end{definition}

The frame condition above is not first-order definable,
so the usual method of adding a first-order relational rule to a labelled calculus does not apply.
Das, van der Giessen, and Marin \cite[Sections 3 and 4]{das_intuitionistic_2024}
instead use a non-well-founded labelled calculus.
Their construction follows Shamkanov's non-well-founded proof-theoretic treatment of $\Logic{GL}$
\cite{shamkanov_circular_2014},
while its trace-based progress condition comes from Simpson's work on cyclic arithmetic
\cite{simpson_cyclic_2017}.
They first apply this method to the standard labelled calculus
$\ell\Logic{K4}$ for $\Logic{K4}$.
An \emph{$\infty$-proof} in $\ell\Logic{K4}$ is a possibly non-well-founded derivation
satisfying the trace-based progress condition:
every infinite branch must have a trace that follows relational atoms infinitely often.
They prove that, for every $\varphi \in \MF$,
$\Logic{GL} \vdash \varphi$ iff
there is an $\infty$-proof in $\ell\Logic{K4}$ of $\imp x : \varphi$ for some label $x$
\cite[Props.\,3.6 and 3.7]{das_intuitionistic_2024}.
The intuitionistic system is then obtained by restricting the succedent.

\begin{definition}[{\cite[Defs.\,4.1 and 4.2]{das_intuitionistic_2024}}]
  \label{def:igl-via-non-well-founded-calculus}
  Let $\ell\Logic{IK4}$ be the restriction of $\ell\Logic{K4}$
  to sequents with exactly one formula in the succedent.
  We let $\ell\Logic{IGL}$ be the class of all $\infty$-proofs in $\ell\Logic{IK4}$,
  and write $\ell\Logic{IGL} \vdash \varphi$ if there is an $\infty$-proof
  of the labelled sequent $\imp x : \varphi$ for some label $x$.
\end{definition}

\begin{theorem}[{\cite[Cor.\,5.6 and Thm.\,8.9]{das_intuitionistic_2024}}]
  \label{thm:igl-calculus-semantics}
  For every $\varphi \in \MF$,
  $\ell\Logic{IGL} \vdash \varphi$ iff $\varphi \in \mathcal{B}\Logic{IGL}$.
\end{theorem}

\begin{definition}[{\cite[Def.\,9.1]{das_intuitionistic_2024}}] \label{def:igl}
  We define $\Logic{IGL}$ to be the common logic in Theorem~\ref{thm:igl-calculus-semantics}, that is,
  \begin{equation*}
    \Logic{IGL} \defeq \mathcal{B}\Logic{IGL}
    = \set{ \varphi \in \MF; \ell\Logic{IGL} \vdash \varphi }.
  \end{equation*}
\end{definition}

The analogue of Theorem \ref{thm:modal-negative-embedding} also holds for $\Logic{IGL}$.

\begin{proposition}[{\cite[p.\,22:16]{das_intuitionistic_2024}}]
  For any $\varphi \in \MF$, $\Logic{GL} \vdash \varphi$ iff $\varphi^N \in \Logic{IGL}$.
\end{proposition}

Unlike in classical $\Logic{GL}$, L\"ob's axiom and its classically equivalent diamond formulation are not equivalent in $\Logic{IGL}$:
$\Logic{IGL}$ has $\Box(\Box p \to p) \to \Box p$ but not
$\Diamond p \to \Diamond(p \land \Box\neg p)$
\cite[Examples 4.3 and 4.8]{das_intuitionistic_2024}.
Also, Almeida, Bezhanishvili, and Lemal proved in a recent preprint
that $\Logic{IGL}$ does not enjoy CIP, and, consequently, neither LIP nor UIP
\cite[Thm.\,4.6]{almeida_fischer-servi_2026}.

Aguilera and Pacheco \cite{aguilera_intuitionistic_2025} introduced a cyclic labelled calculus for $\Logic{IGL}$,
whose proofs are finite trees equipped with progressing back-links.
They proved that this calculus determines the same logic as $\Logic{IGL}$ defined above
\cite[Thm.\,1]{aguilera_intuitionistic_2025}.
Since its proofs are finite objects, this also shows that $\Logic{IGL}$ is recursively enumerable
\cite[Thm.\,2]{aguilera_intuitionistic_2025}.
However, the following problem remains open.

\begin{problem}[{\cite[p.\,22:16]{das_intuitionistic_2024}; \cite[Question~3]{aguilera_intuitionistic_2025}}]
  \label{prob:igl-finite-axiomatization}
  Does $\Logic{IGL}$ have a finite Hilbert-style axiomatization?
  In particular, is $\Logic{IGL} = \Logic{IK4} + \name{L}_\Box$,
  or are further axioms involving $\Diamond$ required?
\end{problem}

\subsection{Heyting--Lewis logics}

The \emph{Lewis arrow} $\lewis$, also called \emph{strict implication},
is a binary modal connective introduced by Lewis.
The formula $\varphi \lewis \psi$ expresses that
it is impossible for $\varphi$ to be true while $\psi$ is false.
Historically, $\lewis$ predates the unary modal operator $\Box$,
but it has largely given way to the more convenient unary operator in classical modal logic,
where $\varphi \lewis \psi$ can be expressed as $\Box(\varphi \to \psi)$.
In an intuitionistic setting, however,
the satisfaction clause for $\varphi \lewis \psi$ compares $\varphi$ and $\psi$ at the modal successors themselves,
whereas $\Box(\varphi \to \psi)$ also quantifies over their intuitionistic successors;
thus, the two formulas need not be equivalent.
On the other hand, $\Box \varphi$ can still be defined as $\top \lewis \varphi$ even in an intuitionistic setting.
An intuitionistic logic with $\lewis$ as a primitive modal connective is often called \emph{Heyting--Lewis logic}.

\begin{definition} \label{def:lewis-language}
  Let $\Lang_\lewis$ be the language obtained from $\PF$ by adding the binary connective $\lewis$,
  and let $\Box \varphi$ abbreviate $\top \lewis \varphi$.
  We thereby identify $\MFb$ with a sublanguage of $\Lang_\lewis$.
\end{definition}

One important interpretation of the Lewis arrow is \emph{preservativity},
a generalization of the provability interpretation from the previous subsection.
Under the preservativity interpretation, $\varphi \lewis \psi$ says that every sufficiently simple arithmetical assumption
from which $T$ proves the realization of $\varphi$ also allows $T$ to prove the realization of $\psi$.
Here, the ``sufficiently simple'' assumptions are the $\Sigma_1$-sentences,
which may roughly be thought of as existential statements whose witnesses can be checked by a finite computation.
The modal-logical treatment of preservativity is already present in Visser \cite{visser_substitutions_2002},
while the name \emph{preservativity logic} was subsequently introduced by Iemhoff \cite{iemhoff_preservativity_2003}.

\begin{definition}[{e.g.\ \cite[Sec.\,5]{litak_lewis_2018}}] \label{def:sigma1-preservativity}
  Let $T$ be an arithmetical theory with a computable set of axioms,
  and let $A$ and $B$ be arithmetical sentences.
  We write $A \lewis_T B$ if, for every $\Sigma_1$-sentence $S$,
  \begin{equation*}
    T \vdash S \to A \quad \Longrightarrow \quad T \vdash S \to B.
  \end{equation*}
  Let $\operatorname{Pres}_T(x,y)$ be an arithmetical formula formalizing this relation.
  An \emph{arithmetical realization} $*$ of $\Lang_\lewis$ maps each propositional variable to an arithmetical sentence,
  commutes with the propositional connectives, and satisfies
  \begin{equation*}
    (\varphi \lewis \psi)^*
    \defeq
    \operatorname{Pres}_T(\ulcorner \varphi^* \urcorner,\ulcorner \psi^* \urcorner).
  \end{equation*}
  The \emph{($\Sigma_1$-)preservativity logic of $\operatorname{Pres}_T(x,y)$}
  consists of the formulas $\varphi \in \Lang_\lewis$
  such that $T \vdash \varphi^*$ for every arithmetical realization $*$.
\end{definition}

It follows immediately from Definition~\ref{def:sigma1-preservativity} that
\begin{equation*}
  \top \lewis_T A \quad \Longleftrightarrow \quad T \vdash A.
\end{equation*}
Litak and Visser show that, under their arithmetization for $\Logic{HA}$,
the derived modality $\Box \varphi = \top \lewis \varphi$ recovers the ordinary provability predicate
\cite[Sec.\,5.4.1]{litak_lewis_2018}.
Moreover, $T \vdash A \to B$ implies $A \lewis_T B$,
and $A \lewis_T B$ together with $B \lewis_T C$ implies $A \lewis_T C$.
These properties correspond to the rule $\name{N}_a$ and the axiom $\name{Tr}$ below.

\begin{definition}[{\cite[Sec.\,4.1]{litak_lewis_2018}}] \label{def:basic-lewis-logics}
  Consider the following rule and axiom schemes:
  \begin{align*}
    \name{N}_a  &: \quad \frac{\varphi \to \psi}{\varphi \lewis \psi}, \\
    \name{Tr}   &: \quad (\varphi \lewis \psi) \to ((\psi \lewis \chi) \to (\varphi \lewis \chi)), \\
    \name{K}_a  &: \quad (\varphi \lewis \psi) \to ((\varphi \lewis \chi) \to (\varphi \lewis (\psi \land \chi))), \\
    \name{Di}   &: \quad (\varphi \lewis \chi) \to ((\psi \lewis \chi) \to ((\varphi \lor \psi) \lewis \chi)), \\
    \name{Box}  &: \quad (\varphi \lewis \psi) \to \Box(\varphi \to \psi).
  \end{align*}
  We define:
  \begin{align*}
    \Logic{HLC}^\flat  &\defeq \Int + \name{N}_a + \name{Tr} + \name{K}_a, &
    \Logic{HLC}^\sharp &\defeq \Logic{HLC}^\flat + \name{Di}, \quad \\
    \Logic{i-Box}      &\defeq \Logic{HLC}^\sharp + \name{Box}.
  \end{align*}
  Following the notation in \cite{visser_lewis_2024,groot_relational_2026},
  a $\name{Di}$-free logic is generally marked with a superscript $\flat$,
  and its extension by $\name{Di}$ with a superscript $\sharp$.\footnote{
    Litak and Visser \cite{litak_lewis_2018} denote $\Logic{HLC}^\flat$ and $\Logic{HLC}^\sharp$
    by $\Logic{iA}^-$ and $\Logic{iA}$, respectively.
    Accordingly, their flat versions of the extensions below carry a superscript minus,
    while their sharp versions carry no superscript.}
\end{definition}

The axiom $\name{Di}$ states that if both $\varphi$ and $\psi$ preserve $\chi$,
then $\varphi \lor \psi$ also preserves $\chi$.
Its arithmetical validity uses a strengthened form of the disjunction property under $\Sigma_1$-assumptions.
In particular, $\name{Di}$ is valid under the preservativity interpretation over $\Logic{HA}$
but not over $\Logic{PA}$
\cite[Secs.\,4.1.2 and 5.4.1]{litak_lewis_2018}.

\begin{proposition}[{\cite[Lem.\,4.1, Fact~4.2, Lem.\,4.4, and Cor.\,4.8]{litak_lewis_2018}}] \label{prop:basic-lewis-logics} \leavevmode
  \begin{enumerate}[label=(\arabic*)]
    \item $\Logic{HLC}^\flat$ proves
          $\name{BL}: \Box(\varphi \to \psi) \to (\varphi \lewis \psi)$
          and
          $\name{LB}: (\varphi \lewis \psi) \to (\Box \varphi \to \Box \psi)$.
    \item $\Logic{HLC}^\flat \dashv\vdash \Logic{iK}_\Box + \name{Tr} + \name{K}_a$.
    \item $\Logic{HLC}^\flat + \name{Box} \vdash \name{Di}$.
    \item $\Logic{HLC}^\sharp + \varphi \lor \neg \varphi \vdash \name{Box}$.
  \end{enumerate}
\end{proposition}

It follows that $\varphi \lewis \psi$ and $\Box(\varphi \to \psi)$ are interdefinable in the presence of $\name{Box}$,
and in particular over the classical extension of $\Logic{HLC}^\sharp$.
Thus, the Lewis arrow is interdefinable with $\Box(\varphi \to \psi)$ over $\Logic{i-Box}$
and the classical extension of $\Logic{HLC}^\sharp$,
but not over $\Logic{HLC}^\sharp$ itself.
Moreover, adding $\name{Box}$ to $\Logic{HLC}^\flat$ already yields $\name{Di}$,
so flat logics collapse into sharp ones in the presence of $\name{Box}$.

\begin{definition}[{\cite[Eq.\,(4)]{litak_lewis_2018}}] \label{def:lewis-semantics}
  Let $\Model{M} = (W, \IRel, \MRel, V)$ be an f-model.
  We define a \emph{Lewisian satisfaction relation} $\Vdash^l_\Model{M}$ on $\Lang_\lewis$
  by the usual intuitionistic clauses for the propositional connectives and the clause
  \begin{equation*}
    x \Vdash^l_\Model{M} \varphi \lewis \psi
    \quad \Longleftrightarrow \quad
    \forall y\,(x \MRel y \imp (y \Vdash^l_\Model{M} \varphi \imp y \Vdash^l_\Model{M} \psi)).
  \end{equation*}
  We write $\Model{M} \vDash^l \varphi$ if $x \Vdash^l_\Model{M} \varphi$ for every $x \in W$,
  and $\Model{F} \vDash^l \varphi$ if $\Model{M} \vDash^l \varphi$ for every f-model $\Model{M}$ based on $\Model{F}$.
  We say that $\Vdash^l_\Model{M}$ is \emph{persistent} if, for every $\varphi \in \Lang_\lewis$ and $x,x' \in W$,
  \begin{equation*}
    x \Vdash^l_\Model{M} \varphi \ \tand\ x \IRel x'
    \quad \Longrightarrow \quad
    x' \Vdash^l_\Model{M} \varphi.
  \end{equation*}
\end{definition}

The $\Box$-condensed condition is also called the \emph{$\lewis$-p condition}.
Accordingly, a $\Box$-condensed f-frame is often called a \emph{$\lewis$-frame} (e.g.\ \cite{litak_lewis_2018}).

\begin{proposition}[{\cite[Facts~3.2 and 3.4]{litak_lewis_2018}}] \label{prop:lewis-frame-conditions}
  Let $\Model{F} = (W, \IRel, \MRel)$ be an f-frame.
  \begin{enumerate}[label=(\arabic*)]
    \item $\Model{F} \in \FrameClass{\Box\text{-con}}$
          iff $\Vdash^l_\Model{M}$ is persistent for every f-model $\Model{M}$ based on $\Model{F}$.
    \item If $\Model{F} \in \FrameClass{\Box\text{-con}}$, then every instance of $\name{Box}$ is valid on $\Model{F}$
          iff $\Model{F} \in \FrameClass{\Box\text{-bri}}$.
  \end{enumerate}
\end{proposition}

\begin{theorem}[{\cite[Thm.\,6.4(a)--(b)]{litak_lewis_2018}}] \label{thm:basic-lewis-completeness} \leavevmode
  \begin{enumerate}[label=(\arabic*)]
    \item $\Logic{HLC}^\sharp = \set{ \varphi \in \Lang_\lewis; \forall \Model{F} \in \FrameClass{\Box\text{-con}}\,(\Model{F} \vDash^l \varphi) }$.
    \item $\Logic{i-Box} = \set{ \varphi \in \Lang_\lewis; \forall \Model{F} \in \FrameClass{\Box\text{-con}} \cap \FrameClass{\Box\text{-bri}}\,(\Model{F} \vDash^l \varphi) }$.
  \end{enumerate}
\end{theorem}

\begin{corollary} \label{cor:lewis-box-fragment}
  $\Logic{HLC}^\sharp \cap \MFb = \Logic{i-Box} \cap \MFb = \Logic{iK}_\Box$.
  Also, let $\tau: \Lang_\lewis \to \MFb$ be the translation that commutes with the propositional connectives and satisfies
  $\tau(\psi_1 \lewis \psi_2) = \Box(\tau(\psi_1) \to \tau(\psi_2))$.
  Then, for every $\varphi \in \Lang_\lewis$,
  \begin{equation*}
    \Logic{i-Box} \vdash \varphi
    \quad \Longleftrightarrow \quad
    \Logic{iK}_\Box \vdash \tau(\varphi).
  \end{equation*}
  \begin{proof}
    The first claim follows from Theorem~\ref{thm:basic-lewis-completeness}
    and Theorem~\ref{cor:condensed-completeness},
    since the restriction of $\Vdash^l$ to $\MFb$ is the extrinsic satisfaction relation $\SatE$.
    For the second claim, $\name{BL}$ and $\name{Box}$ yield
    $\Logic{i-Box} \vdash (\psi_1 \lewis \psi_2) \leftrightarrow \Box(\psi_1 \to \psi_2)$.
    An easy induction on $\varphi$ therefore gives
    $\Logic{i-Box} \vdash \varphi \leftrightarrow \tau(\varphi)$.
    Since $\tau(\varphi) \in \MFb$, the result follows from the first claim.
  \end{proof}
\end{corollary}

Several extensions of the basic Lewis logics arise from their arithmetical interpretation.

\begin{definition}[{\cite[Sec.\,4.1.2]{litak_lewis_2018}}] \label{def:preservativity-logics}
  Consider the following axiom schemes:
  \begin{align*}
    \name{L}_a &: \quad (\Box \varphi \to \varphi) \lewis \varphi, \\
    \name{W}_a &: \quad ((\varphi \land \Box \psi) \lewis \psi) \to (\varphi \lewis \psi), \\
    \name{M}_a &: \quad (\varphi \lewis \psi) \to ((\Box \chi \to \varphi) \lewis (\Box \chi \to \psi)).
  \end{align*}
  We define:
  \begin{align*}
    \Logic{iGL}_a^\flat   &\defeq \Logic{HLC}^\flat + \name{L}_a, &
    \Logic{iGL}_a^\sharp  &\defeq \Logic{HLC}^\sharp + \name{L}_a, \\
    \Logic{iGW}_a^\flat   &\defeq \Logic{HLC}^\flat + \name{W}_a, &
    \Logic{iGW}_a^\sharp  &\defeq \Logic{HLC}^\sharp + \name{W}_a, \\
    \Logic{iPreL}^\flat   &\defeq \Logic{iGW}_a^\flat + \name{M}_a, &
    \Logic{iPreL}^\sharp  &\defeq \Logic{iGW}_a^\sharp + \name{M}_a.
  \end{align*}
\end{definition}

These systems form three pairs.
The logics $\Logic{iGL}_a^\flat$ and $\Logic{iGL}_a^\sharp$ add the Lewisian version $\name{L}_a$ of L\"ob's principle,
the logics $\Logic{iGW}_a^\flat$ and $\Logic{iGW}_a^\sharp$ use the stronger principle $\name{W}_a$,
and the logics $\Logic{iPreL}^\flat$ and $\Logic{iPreL}^\sharp$ add $\name{M}_a$ as well.
The flat--sharp distinction remains important for these extensions.
Litak and Visser prove arithmetical soundness results for both variants;
in particular, $\Logic{iPreL}^\sharp$ is arithmetically valid over $\Logic{HA}$ and several related intuitionistic theories
\cite[Sec.\,5.4.1]{litak_lewis_2018}.
They also characterize the preservativity logic of a broad class of classical theories as
$\Logic{iPreL}^\flat$ plus excluded middle \cite[Thm.\,5.10]{litak_lewis_2018}.
The intuitionistic case is less settled: $\Logic{iPreL}^\sharp$ does not capture every preservativity principle valid in $\Logic{HA}$,
and the exact preservativity logics of $\Logic{HA}$ and several of its variants remain open \cite[Open Question~5.9 (VII)--(IX)]{litak_lewis_2018}.

Litak and Visser \cite{litak_lewis_2018} provide a broad overview of Heyting--Lewis logics.
They study explicit fixed points in this setting \cite{litak_lewisian_2024},
and also discuss the Lewisian variant $\Logic{iSL}_a$ of $\Logic{iSL}_\Box$
in their preprint \cite{visser_lewis_2024}.
Iemhoff, de Jongh, and Zhou \cite{iemhoff_properties_2005} also study the fixed-point and Beth properties of related provability and preservativity logics.
For a related recent development, Mojtahedi announces axiomatizations of preservativity and relative admissibility
for several modal logics extending $\Logic{iK4}$ in his preprint \cite{mojtahedi_provability_2026}.
De Groot and Litak \cite{groot_relational_2026} give a Kripke-style relational semantics for $\Logic{HLC}^\flat$
and prove its strong completeness and finite model property.

\subsection{Epistemic and propositional lax logics}

The strongness axiom $\name{S}_\Box: \varphi \to \Box \varphi$ also occurs in intuitionistic epistemic logic.
Here, $\Box \varphi$ is read, roughly, as ``$\varphi$ has been verified,'' or more precisely, as ``it has been verified that $\varphi$ has a proof.''
Since every intuitionistically true proposition $\varphi$ has a proof, and every proof is also a verification,
Artemov and Protopopescu \cite{artemov_intuitionistic_2016} accept
$\name{S}_\Box$, which they call the principle of the \emph{constructivity of truth} or \emph{coreflection}.

\begin{definition}[{\cite[Sec.\,3]{artemov_intuitionistic_2016}}] \label{def:iel-logics}
  We define
  \begin{equation*}
    \Logic{IEL}^{-} \defeq \Logic{iK}_\Box + \name{S}_\Box
    \qquad \text{and} \qquad
    \Logic{IEL} \defeq \Logic{IEL}^{-} + (\Box \varphi \to \neg \neg \varphi).
  \end{equation*}
\end{definition}

The logic $\Logic{IEL}^{-}$ is intended as a basic logic of \emph{intuitionistic belief},
in which verifications of false propositions\footnote{
  Artemov and Protopopescu \cite[p.276]{artemov_intuitionistic_2016} give the following example:
  ``before the European discovery of Australia, all available evidence supported the proposition `all swans are white'; this turned out to be false and can be taken as an instance of verification-based belief which may be false.''
} are not ruled out,
whereas $\Logic{IEL}$ is intended as a logic of \emph{intuitionistic knowledge}.
The additional principle $\Box \varphi \to \neg \neg \varphi$ expresses \emph{weak factivity}:
a proposition known in this verification-based sense cannot be false.
Over $\Logic{IEL}^{-}$, it is equivalent to $\name{D}_\Box: \neg \Box \bot$,
which is also known as the \emph{consistency axiom} in classical doxastic logic.
In this sense, $\Logic{IEL}^{-}$ does not distinguish intuitionistic knowledge from \emph{provably consistent intuitionistic belief}.
If the propositional base is made classical, however, weak factivity implies ordinary reflection $\name{T}_\Box: \Box \varphi \to \varphi$;
together with $\name{S}_\Box$, this makes $\varphi$ and $\Box \varphi$ equivalent.
Artemov and Protopopescu also explicitly reject $\name{T}_\Box$ as a general principle of intuitionistic knowledge,
as a verification of $\varphi$ does not necessarily produce an explicit proof of $\varphi$.

\begin{proposition}[{\cite[Thms.\,3.5 and 4.7]{artemov_intuitionistic_2016}}] \label{prop:iel-properties} \leavevmode
  \begin{enumerate}[label=(\arabic*)]
    \item $\Logic{IEL} \dashv \vdash \Logic{IEL}^{-} + \name{D}_\Box$.
    \item Neither $\Logic{IEL}^{-}$ nor $\Logic{IEL}$ proves $\name{T}_\Box$.
  \end{enumerate}
\end{proposition}

Artemov and Protopopescu claim that $\Logic{IEL}$ largely avoids common criticisms of classical epistemic logic.
Classical epistemic logic is based on normal classical modal logic,
which has $\name{Nec}_\Box: \frac{\varphi}{\Box \varphi}$ as a rule.
Together with the axiom $\name{K}_\Box$, this leads to the familiar phenomenon of
\emph{logical omniscience}.
Under the classical epistemic reading of $\Box \varphi$ as ``$\varphi$ is known,''
$\name{Nec}_\Box$ entails, in particular, that all classical tautologies are known.
The principle $\name{S}_\Box$, read classically, represents an even stronger form
of logical omniscience: ``if $\varphi$ is true, then $\varphi$ is known.''
Artemov and Protopopescu assert that $\name{S}_\Box$ has a very different meaning
under their notion of intuitionistic knowledge, namely,
``constructive truth, i.e., proof, yields verification/knowledge''
\cite[Sec.\,5.1.1]{artemov_intuitionistic_2016}.

They also respond to \emph{Fitch's paradox of knowability}.
The Church--Fitch argument shows that, when classical epistemic logic is extended with a possibility modality $\blacklozenge$,
the seemingly innocent principle
$\varphi \to \blacklozenge \Box \varphi$
(``if $\varphi$ is true, then $\varphi$ is knowable'')
entails the omniscience principle $\name{S}_\Box$.
Artemov and Protopopescu argue that not only is this conclusion unproblematic
when read intuitionistically, but also
$\varphi \to \blacklozenge \Box \varphi$, when formulated in classical logic,
fails to adequately capture the intuitionistic relation between truth and knowledge.
For them, constructivity, rather than knowability, is the defining feature of
intuitionistic truth, and $\name{S}_\Box$ directly expresses this constructivity;
various principles of knowability, such as $\varphi \to \blacklozenge \Box \varphi$, may instead be regarded as consequences of it.
They therefore conclude that the knowability paradox is
``a pseudoproblem which holds only from a classical standpoint''
\cite[Sec.\,5]{artemov_intuitionistic_2016}.

As with the other $\Diamond$-free logics, the relational semantics of these logics is a special case of the extrinsic semantics.

\begin{definition}[{\cite[Defs.\,4.1 and 4.2]{artemov_intuitionistic_2016}}] \label{def:iel-frames}
  An f-frame $\Model{F} = (W, \IRel, \MRel)$ is an \emph{$\Logic{IEL}^{-}$-frame}
  if it is $\Box$-condensed and $\MRel$-strong.
  It is an \emph{$\Logic{IEL}$-frame} if, in addition, ${\MRel}$ is serial.
  An $L$-model, for $L \in \set{ \Logic{IEL}^{-}, \Logic{IEL} }$, is an f-model based on an $L$-frame.
\end{definition}

By Proposition~\ref{prop:sbox-frame-condition}, $\MRel$-strongness is equivalent to the
$\MRel_\beta$-strongness condition used for $\Logic{iSL}_\Box$, and it corresponds to coreflection.
Seriality validates $\name{D}_\Box$.
Also, the $\Box$-condensation condition ensures the $\Box$-persistency.

\begin{theorem}[{\cite[Thms.\,4.4 and 4.6]{artemov_intuitionistic_2016}}] \label{thm:iel-completeness}
  Let $L \in \set{ \Logic{IEL}^{-}, \Logic{IEL} }$ and $\varphi \in \MFb$.
  Then $L \vdash \varphi$ iff $\Model{F} \ValidE \varphi$ for every $L$-frame $\Model{F}$.
\end{theorem}

Krupski and Yatmanov \cite{krupski_sequent_2016} give a cut-free sequent calculus for $\Logic{IEL}$ and prove that its decision problem is
$\textsc{PSPACE}$-complete.
Hagemeier and Kirst \cite{hagemeier_constructive_2022} give constructive proofs, mechanized in Coq, of completeness, decidability,
semantic cut elimination, and the finite model property for $\Logic{IEL}$.
Fiorino \cite{fiorino_linear_2023} later gives sequent calculi with the subformula property whose proof search terminates in linear depth;
failed proof search also yields countermodels of minimal depth.

Van der Giessen studies admissible rules for $\Logic{IEL}^{-}$, under the name $\Logic{iCK4}$,
and projectivity for $\Logic{IEL}$, but does not establish interpolation for either logic
\cite[Sec.\,1.3.1 and Chs.\,6--7]{giessen_uniform_admissible_2022}.
The author is not aware of any interpolation results for these logics.

\begin{problem} \label{prob:iel-interpolation}
  Do $\Logic{IEL}^{-}$ and $\Logic{IEL}$ enjoy CIP, LIP, and UIP?
\end{problem}

$\name{S}_\Box$ has a different interpretation in \emph{propositional lax logic}.
The literature generally writes the modality as $\bigcirc$ \cite{fairtlough_propositional_1997,iemhoff_proof_2024}.
This notation distinguishes it from an ordinary box or diamond, since it has both $\Box$-like and $\Diamond$-like features.

\begin{definition}[{\cite{fairtlough_propositional_1997,iemhoff_proof_2024}}] \label{def:pll}
  Let $\Lang_{\bigcirc}$ and $\Logic{iK}_{\bigcirc}$ be obtained from $\MFb$ and $\Logic{iK}_\Box$,
  respectively, by replacing $\Box$ with $\bigcirc$.
  We define
  \begin{equation*}
    \Logic{PLL} \defeq
    \Logic{iK}_{\bigcirc}
      + (\varphi \to \bigcirc \varphi)
      + (\bigcirc \bigcirc \varphi \to \bigcirc \varphi).
  \end{equation*}
\end{definition}

In particular, the lax modality satisfies
\begin{equation*}
  \varphi \to \bigcirc \varphi,
  \qquad
  \bigcirc \bigcirc \varphi \to \bigcirc \varphi,
  \qquad
  \bigcirc \varphi \land \bigcirc \psi \to \bigcirc (\varphi \land \psi).
\end{equation*}
The first two principles make $\bigcirc$ inflationary and idempotent.
Together with the monotonicity and preservation of $\top$ inherited from $\Logic{iK}_{\bigcirc}$,
the third makes $\bigcirc$ preserve finite meets;
algebraically, $\bigcirc$ is therefore a nucleus.
Fairtlough and Mendler \cite{fairtlough_propositional_1997} use $\bigcirc \varphi$ to express that $\varphi$ holds subject to a constraint,
with applications to hardware verification.
For example, in the timing analysis of combinational circuits, the constraint is a bound on signal-propagation delay:
a wire contributes zero delay, sequential composition adds delays, and parallel composition takes their maximum
\cite[Sec.\,1]{fairtlough_propositional_1997}.
They also proposed a sound and complete semantics for $\Logic{PLL}$ called \emph{Kripke constraint models}.
After identifying $\bigcirc$ with the constructive diamond,
these models are precisely the $\MRel$-preordered and $\MRel$-strong CK-models.

\begin{definition}[{\cite[Defs.\,3.1, 3.2]{fairtlough_propositional_1997}}] \label{def:pll-constraint-model}
  $\Model{M} = (W, W_\bot, \IRel, \MRel, V)$ is a \emph{Kripke constraint model}
  if $(W, \IRel, \MRel, V)$ is a $\MRel$-preordered and $\MRel$-strong f-model,
  $W_\bot \subseteq W$ is upward closed with respect to $\IRel$,
  and $W_\bot \subseteq V(p)$ for every $p \in \PropVar$.
  A satisfaction relation ${\Vdash_\Model{M}} \subseteq W \times \Lang_{\bigcirc}$ for $\Model{M}$ is defined as follows:
  \begin{itemize}
    \item $x \Vdash_\Model{M} \bot \defiff x \in W_\bot$;
    \item $x \Vdash_\Model{M} \psi_1 \to \psi_2
      \defiff \forall x' \IRelRev x\, (x' \Vdash_\Model{M} \psi_1 \imp x' \Vdash_\Model{M} \psi_2)$;
    \item $x \Vdash_\Model{M} \bigcirc \varphi
      \defiff \forall x' \IRelRev x\, \exists y' \MRelRev x'\, (y' \Vdash_\Model{M} \varphi)$;
    \item The other cases are identical to the classical one.
  \end{itemize}
  We write $\Model{M} \vDash \varphi$ iff $x \Vdash_\Model{M} \varphi$ for every $x \in W$.
\end{definition}

\begin{theorem}[{\cite[Thms.\,3.3, 4.4]{fairtlough_propositional_1997}}] \label{thm:pll-constraint-completeness}
  Let $\varphi \in \Lang_{\bigcirc}$.
  $\Logic{PLL} \vdash \varphi$ iff $\Model{M} \vDash \varphi$ for every Kripke constraint model $\Model{M}$.
\end{theorem}

Iemhoff \cite{iemhoff_proof_2024} introduced a cut-admissible Gentzen-style sequent calculus and a terminating Gentzen-style calculus for $\Logic{PLL}$.
She also claimed in the same paper that $\Logic{PLL}$ enjoys UIP,
but an updated preprint\footnote{\url{https://arxiv.org/abs/2209.08976v2}} notes a mistake in the proof,
so the claimed result remains unproved.

\begin{problem} \label{prob:pll-interpolation}
  Does $\Logic{PLL}$ enjoy CIP, LIP, and UIP?
\end{problem}

Benton, Bierman, and de Paiva \cite[Secs.\,3.1 and 3.3]{benton_computational_1998}
derive their propositional CL-logic from Moggi's computational lambda calculus via the Curry--Howard correspondence
and observe that an equivalent axiomatization had independently been proposed by Fairtlough and Mendler under the name $\Logic{PLL}$.
Valliappan \cite{valliappan_lax_2026} also develops typed lambda calculi for sublogics of $\Logic{PLL}$,
together with possible-world and categorical semantics and constructive results on normalization and equational completeness.

Up to the change of modal symbol, $\Logic{IEL}^{-}$ is the common base of $\Logic{IEL}$ and $\Logic{PLL}$.
The former adds weak factivity, whereas the latter adds $\bigcirc \bigcirc \varphi \to \bigcirc \varphi$.

\begin{definition} \label{def:pll-extrinsic-frames}
  An f-frame $\Model{F} = (W, \IRel, \MRel)$ is \emph{$\MRel$-semi-dense} if
  ${\MRel} \subseteq (\MRel \cdot \MRel \cdot \IRel)$,
  that is, for any $x \MRel y$, there are $u, v \in W$ such that
  $x \MRel u \MRel v \IRel y$.
  We say $\Model{F}$ is a \emph{$\Logic{PLL}$-frame} if it is $\Box$-p, $\MRel$-strong,
  and $\MRel$-semi-dense.
\end{definition}

\begin{figure}[htb]
  \centering
  \begin{tikzpicture}
    \fill (0, 0) circle (2pt);
    \draw (0, 0) node[above]{$x$};

    \fill (3, 0) circle (2pt);
    \draw (3, 0) node[above]{$y$};

    \fill (1, -2) circle (2pt);
    \draw (1, -2) node[below]{$\exists u$};

    \fill (3, -2) circle (2pt);
    \draw (3, -2) node[below]{$\exists v$};

    \draw [thick, ->, rel, squig] (0,0) to (3,0);
    \draw (1.5,0) node[above]{$\MRel$};

    \draw [thick, ->, rel, squig, dashed] (0,0) to (1,-2);
    \draw (0.5,-1) node[below left]{$\MRel$};

    \draw [thick, ->, rel, squig, dashed] (1,-2) to (3,-2);
    \draw (2,-2) node[below]{$\MRel$};

    \draw [thick, ->, rel, dashed] (3,-2) to (3,0);
    \draw (3,-1) node[right]{$\IRel$};
  \end{tikzpicture}
  \caption{Visualization of $\MRel$-semi-density.}
\end{figure}

\begin{proposition} \label{prop:pll-converse-four-frame-condition}
  Let $\Model{F} = (W, \IRel, \MRel)$ be a $\Box$-p f-frame.
  Then $\Model{F} \ValidE \Box \Box \varphi \to \Box \varphi$ for every $\varphi \in \MFb$
  iff $\Model{F}$ is $\MRel$-semi-dense.
  \begin{proof}
    ($\Leftarrow$) Routine.
    ($\Rightarrow$) Suppose that $\Model{F}$ is not $\MRel$-semi-dense,
    then there are $x, y \in W$ such that $x \MRel y$ and $x \MRel u \MRel v \IRel y$ holds for no $u, v \in W$.
    Take a fresh variable $p^* \in \PropVar$, then define a valuation $V$ by
    \begin{equation*}
      w \in V(p) \defiff (p \ne p^* \tor \exists u,v \in W\, (x \MRel u \MRel v \IRel w)).
    \end{equation*}
    It is easy to check that each $V(p)$ is upward closed with respect to $\IRel$, so
    $\Model{M} = (\Model{F},V)$ is an f-model.
    The definition of $V(p)$ gives $x \SatE_\Model{M} \Box \Box p$,
    while the choice of $x,y$ gives $y \nVdash^e_\Model{M} p$ and hence
    $x \nVdash^e_\Model{M} \Box p$.
    By the reflexivity of $\IRel$,
    $x \nVdash^e_\Model{M} \Box \Box p \to \Box p$.
    Therefore, $\Model{F} \not\ValidE \Box \Box p \to \Box p$.
  \end{proof}
\end{proposition}

By Theorem~\ref{thm:ikb-ikd-completeness}, Proposition~\ref{prop:sbox-frame-condition},
and Proposition~\ref{prop:pll-converse-four-frame-condition},
$\Logic{IEL}^{-} + (\Box \Box \psi \to \Box \psi)$ is sound with respect to
$\Logic{PLL}$-frames.

\begin{problem} \label{prob:pll-extrinsic-completeness}
  Is this semantics complete?
  More precisely, for every $\varphi \in \MFb$, is
  $\Logic{IEL}^{-} + (\Box \Box \psi \to \Box \psi) \vdash \varphi$
  iff 
  $\Model{F} \ValidE \varphi$
  for every $\Logic{PLL}$-frame $\Model{F}$?
  An affirmative answer would give an f-frame semantics for $\Logic{PLL}$ after replacing
  $\bigcirc$ with $\Box$.
\end{problem}

There are also different proposals for an epistemic possibility connective.

\begin{definition}[{\cite{pacheco_epistemic_2024}}] \label{def:iel-possibility}
  On an $\Logic{IEL}$-model $\Model{M}$, Pacheco interprets $\widehat K$ by
  \begin{equation*}
    x \SatE_\Model{M} \widehat K \varphi
    \defiff
    \forall x' \IRelRev x \, \exists y \,
      (x' \MRel y \tand y \SatE_\Model{M} \varphi).
  \end{equation*}
\end{definition}

This is the intrinsic, or Wijesekera-style, diamond clause.
Pacheco proves that $\widehat K \varphi$ and $\neg \neg \varphi$ have the same truth value at every world of every
$\Logic{IEL}$-model \cite[Thm.\,5]{pacheco_epistemic_2024}; thus this possibility connective is not independent.

\begin{definition}[{\cite[p.\,49]{simpson_proof_1994}}] \label{def:prenosil-diamond}
  On an f-model $\Model{M} = (W, \IRel, \MRel, V)$, consider the diamond clause
  \begin{equation*}
    x \Vdash^{\mathrm p}_\Model{M} \Diamond \varphi
    \defiff
    \exists x', y \in W \,
      (x' \IRel x \tand x' \MRel y \tand y \Vdash^{\mathrm p}_\Model{M} \varphi).
  \end{equation*}
  Thus $\Diamond \varphi$ holds at $x$ when $\varphi$ holds at an ${\MRel}$-successor
  of some ${\IRel}$-predecessor of $x$.
  Simpson attributes this clause to unpublished work of Plotkin and Stirling.
  P\v{r}enosil later developed the corresponding operation on upsets \cite[p.\,428]{prenosil_duality_2014}.
\end{definition}

Using this clause, Balbiani \cite{balbiani_intuitionistic_2025} obtains complete two-modal doxastic and epistemic logics
$\Logic{L}_{\mathrm{dox}}$ and $\Logic{L}_{\mathrm{epi}}$, in which the box and diamond are independent.

For a different two-modal approach, Kurz and Palmigiano \cite{kurz_epistemic_2013}
develop an $\Logic{IK}$-based dynamic multimodal epistemic logic.

\section*{Acknowledgements}

I would like to express my deep gratitude to
Taishi Kurahashi, my supervisor, for their extensive support and advice throughout this research,
and to Leonardo Pacheco for keeping me informed of the latest developments in intuitionistic and constructive modal logic,
including some of his own recent results,
and for providing thoughtful feedback on earlier versions of this paper.
I would like to thank Anupam Das and Ian Shillito for
helpful discussions at \emph{Logic Colloquium 2025},
particularly concerning some of the observations made in this survey, and for telling me about their ongoing work and emerging results.
I would also like to thank Mashu Noguchi, my colleague, for various feedback.

During the preparation of this manuscript, the author used OpenAI Codex to assist with language proofreading and to check mathematical arguments for potential errors.
All suggestions made by the tool were independently reviewed and verified by the author, who takes full responsibility for the content of the manuscript.

\appendix

\section{Useful lists, tables, and figures} \label{appendix:tables}

\begin{table}[H]
  \centering
  \setlength{\extrarowheight}{3pt}
  \caption{List of modal axioms.}
  \label{tab:axioms}
  \begin{tabular}{rl|rl}
    $\name{M}_\Box$     & \AxiomC{$\varphi \to \psi$} \UnaryInfC{$\Box \varphi \to \Box \psi$} \DisplayProof &
    $\name{M}_\Diamond$ & \AxiomC{$\varphi \to \psi$} \UnaryInfC{$\Diamond \varphi \to \Diamond \psi$} \DisplayProof \\
    $\name{N}_\Box$      & $\Box \top$ &
    $\name{N}_\Diamond$ & $\neg \Diamond \bot$ \\
    $\name{C}_\Box$     & $(\Box \varphi \land \Box \psi) \to \Box (\varphi \land \psi)$ &
    $\name{C}_\Diamond$ & $\Diamond (\varphi \lor \psi) \to (\Diamond \varphi \lor \Diamond \psi)$ \\
    $\name{K}_\Diamond$ & $\Box (\varphi \to \psi) \to (\Diamond \varphi \to \Diamond \psi)$ & & \\
    $\name{FS1}$ & $\Diamond (\varphi \to \psi) \to (\Box \varphi \to \Diamond \psi)$ &
    $\name{FS2}$ & $(\Diamond \varphi \to \Box \psi) \to \Box (\varphi \to \psi)$ \\
    & & $\name{wCD}$ & $\Box (\varphi \lor \psi) \to ((\Diamond \varphi \to \Box \psi) \to \Box \psi)$ \\
    & & $\name{CD}$ & $\Box (\varphi \lor \psi) \to (\Diamond \varphi \lor \Box \psi)$ \\
  \end{tabular}%
  \vspace{2ex}
  \begin{tabular}{rl|rl|rl}
    $\name{D}_\Box$ & $\neg \Box \bot$ & $\name{D}_\Diamond$ & $\Diamond \top$ & $\name{D}$ & $\Box \varphi \to \Diamond \varphi$ \\
    $\name{T}_\Box$ & $\Box \varphi \to \varphi$ & $\name{T}_\Diamond$ & $\varphi \to \Diamond \varphi$ & $\name{T}$ & $\name{T}_\Box \land \name{T}_\Diamond$ \\
    $\name{B}_\Box$ & $\Diamond \Box \varphi \to \varphi$ & $\name{B}_\Diamond$ & $\varphi \to \Box \Diamond \varphi$ & $\name{B}$ & $\name{B}_\Box \land \name{B}_\Diamond$ \\
    $\name{4}_\Box$ & $\Box \varphi \to \Box \Box \varphi$ & $\name{4}_\Diamond$ & $\Diamond \Diamond \varphi \to \Diamond \varphi$ & $\name{4}$ & $\name{4}_\Box \land \name{4}_\Diamond$ \\
    $\name{5}_\Box$ & $\Diamond \Box \varphi \to \Box \varphi$ & $\name{5}_\Diamond$ & $\Diamond \varphi \to \Box \Diamond \varphi$ & $\name{5}$ & $\name{5}_\Box \land \name{5}_\Diamond$ \\
    $\name{S}_\Box$ & $\varphi \to \Box \varphi$ \\
    $\name{L}_\Box$ & $\Box (\Box \varphi \to \varphi) \to \Box \varphi$ \\
  \end{tabular}%
\end{table}

\begin{table}[H]
  \caption{List of basic logics.}
  \vspace{-3.0ex}
  \begin{align*}
    \Logic{iK}_\Box     & \defeq  \Logic{Int} + \name{M}_\Box + \name{N}_\Box + \name{C}_\Box \\
    \Logic{iK}_\Diamond & \defeq  \Logic{Int} + \name{M}_\Diamond + \name{N}_\Diamond + \name{C}_\Diamond \\
    \Logic{CK} & \defeq \Logic{iK}_\Box + \name{K}_\Diamond = \Logic{iK}_\Box + \name{M}_\Diamond + \name{FS1} \\
    \Logic{WK} & \defeq \Logic{CK} + \name{N}_\Diamond \\
    \Logic{IK}^- & \defeq \Logic{WK} + \name{C}_\Diamond \\
    \Logic{FIK} & \defeq \Logic{IK}^- + \name{wCD} \\
    \Logic{LIK} & \defeq \Logic{IK}^- + \name{CD} \\
    \Logic{IK} & \defeq \Logic{IK}^- + \name{FS2} \\
    \Logic{IK}^c & \defeq \Logic{LIK} + \name{FS2} = \Logic{IK} + \name{CD} \\
  \end{align*}
\end{table}

\begin{table}[H]
  \centering
  \caption{Interaction and persistence conditions.}
  \label{tab:principal-frame-conditions}
  \scriptsize
  \setlength{\tabcolsep}{8pt}
  \renewcommand{\arraystretch}{1.5}
  \begin{tabular}{c c c}
    \hline
    Condition & Visualization & Relational form \\
    \hline
    $\Box$-p
      & \begin{tikzpicture}[baseline=(current bounding box.center),scale=.75]
          \fill (0,0) circle (2pt); \node[below] at (0,0) {$x$};
          \fill (0,2) circle (2pt); \node[above] at (0,2) {$x'$};
          \fill (2,0) circle (2pt); \node[below] at (2,0) {$\exists y$};
          \fill (2,2) circle (2pt); \node[above] at (2,2) {$y'$};
          \draw[thick,->,rel] (0,0) -- (0,2); \node[left] at (0,1) {$\IRel$};
          \draw[thick,->,rel,squig,dashed] (0,0) -- (2,0); \node[below] at (1,0) {$\MRel$};
          \draw[thick,->,rel,dashed] (2,0) -- (2,2); \node[right] at (2,1) {$\IRel$};
          \draw[thick,->,rel,squig] (0,2) -- (2,2); \node[above] at (1,2) {$\MRel$};
        \end{tikzpicture}
      & $(\IRel \mathbin{\cdot} \MRel) \subseteq (\MRel \mathbin{\cdot} \IRel)$ \\
    $\Diamond$-p
      & \begin{tikzpicture}[baseline=(current bounding box.center),scale=.75]
          \fill (0,0) circle (2pt); \node[below] at (0,0) {$x$};
          \fill (0,2) circle (2pt); \node[above] at (0,2) {$x'$};
          \fill (2,0) circle (2pt); \node[below] at (2,0) {$y$};
          \fill (2,2) circle (2pt); \node[above] at (2,2) {$\exists y'$};
          \draw[thick,->,rel] (0,0) -- (0,2); \node[left] at (0,1) {$\IRel$};
          \draw[thick,->,rel,squig] (0,0) -- (2,0); \node[below] at (1,0) {$\MRel$};
          \draw[thick,->,rel,dashed] (2,0) -- (2,2); \node[right] at (2,1) {$\IRel$};
          \draw[thick,->,rel,squig,dashed] (0,2) -- (2,2); \node[above] at (1,2) {$\MRel$};
        \end{tikzpicture}
      & $(\IRel^{-1} \mathbin{\cdot} \MRel) \subseteq (\MRel \mathbin{\cdot} \IRel^{-1})$ \\
    FS2
      & \begin{tikzpicture}[baseline=(current bounding box.center),scale=.75]
          \fill (0,0) circle (2pt); \node[below] at (0,0) {$x$};
          \fill (0,2) circle (2pt); \node[above] at (0,2) {$\exists x'$};
          \fill (2,0) circle (2pt); \node[below] at (2,0) {$y$};
          \fill (2,2) circle (2pt); \node[above] at (2,2) {$y'$};
          \draw[thick,->,rel,dashed] (0,0) -- (0,2); \node[left] at (0,1) {$\IRel$};
          \draw[thick,->,rel,squig] (0,0) -- (2,0); \node[below] at (1,0) {$\MRel$};
          \draw[thick,->,rel] (2,0) -- (2,2); \node[right] at (2,1) {$\IRel$};
          \draw[thick,->,rel,squig,dashed] (0,2) -- (2,2); \node[above] at (1,2) {$\MRel$};
        \end{tikzpicture}
      & $(\MRel \mathbin{\cdot} \IRel) \subseteq (\IRel \mathbin{\cdot} \MRel)$ \\
    \hline
  \end{tabular}
\end{table}

\begin{table}[H]
  \centering
  \caption{Condensation and brilliance conditions.}
  \label{tab:condensation-frame-conditions}
  \scriptsize
  \setlength{\tabcolsep}{8pt}
  \renewcommand{\arraystretch}{1.5}
  \begin{tabular}{c c c}
    \hline
    Condition & Visualization & Relational form \\
    \hline
    $\Box$-condensed
      & \begin{tikzpicture}[baseline=(current bounding box.center),scale=.75]
          \fill (0,0) circle (2pt); \node[left] at (0,0) {$x$};
          \fill (0,2) circle (2pt); \node[left] at (0,2) {$x'$};
          \fill (2,2) circle (2pt); \node[right] at (2,2) {$y$};
          \draw[thick,->,rel] (0,0) -- (0,2); \node[left] at (0,1) {$\IRel$};
          \draw[thick,->,rel,squig] (0,2) -- (2,2); \node[below right] at (1,1) {$\MRel$};
          \draw[thick,->,rel,squig,dashed] (0,0) -- (2,2); \node[above] at (1,2) {$\MRel$};
        \end{tikzpicture}
      & $(\IRel \mathbin{\cdot} \MRel) \subseteq {\MRel}$ \\
    $\Box$-brilliant
      & \begin{tikzpicture}[baseline=(current bounding box.center),scale=.75]
          \fill (0,0) circle (2pt); \node[left] at (0,0) {$x$};
          \fill (2,0) circle (2pt); \node[right] at (2,0) {$y$};
          \fill (2,2) circle (2pt); \node[right] at (2,2) {$y'$};
          \draw[thick,->,rel,squig] (0,0) -- (2,0); \node[below] at (1,0) {$\MRel$};
          \draw[thick,->,rel] (2,0) -- (2,2); \node[right] at (2,1) {$\IRel$};
          \draw[thick,->,rel,squig,dashed] (0,0) -- (2,2); \node[above left] at (1,1) {$\MRel$};
        \end{tikzpicture}
      & $(\MRel \mathbin{\cdot} \IRel) \subseteq {\MRel}$ \\
    $\Diamond$-condensed
      & \begin{tikzpicture}[baseline=(current bounding box.center),scale=.75]
          \fill (0,0) circle (2pt); \node[left] at (0,0) {$x$};
          \fill (0,2) circle (2pt); \node[left] at (0,2) {$x'$};
          \fill (2,0) circle (2pt); \node[right] at (2,0) {$y$};
          \draw[thick,->,rel] (0,0) -- (0,2); \node[left] at (0,1) {$\IRel$};
          \draw[thick,->,rel,squig,dashed] (0,2) -- (2,0); \node[above] at (1,1) {$\MRel$};
          \draw[thick,->,rel,squig] (0,0) -- (2,0); \node[below] at (1,0) {$\MRel$};
        \end{tikzpicture}
      & $(\IRel^{-1} \mathbin{\cdot} \MRel) \subseteq {\MRel}$ \\
    $\Diamond$-brilliant
      & \begin{tikzpicture}[baseline=(current bounding box.center),scale=.75]
          \fill (0,2) circle (2pt); \node[left] at (0,2) {$x'$};
          \fill (2,0) circle (2pt); \node[right] at (2,0) {$y$};
          \fill (2,2) circle (2pt); \node[right] at (2,2) {$y'$};
          \draw[thick,->,rel,squig,dashed] (0,2) -- (2,0); \node[below left] at (1,1) {$\MRel$};
          \draw[thick,->,rel] (2,0) -- (2,2); \node[right] at (2,1) {$\IRel$};
          \draw[thick,->,rel,squig] (0,2) -- (2,2); \node[above] at (1,2) {$\MRel$};
        \end{tikzpicture}
      & $(\MRel \mathbin{\cdot} \IRel^{-1}) \subseteq {\MRel}$ \\
    \hline
  \end{tabular}
\end{table}

\begin{table}[H]
  \centering
  \caption{Interpolation status of the logics discussed in the survey.}
  \label{tab:interpolation-status}
  \renewcommand{\arraystretch}{1.1}
  \begin{tabular*}{.9\textwidth}{@{\extracolsep{\fill}}lccc}
    \hline
    Logic & CIP & LIP & UIP \\
    \hline
    $\Logic{iK}_\Box$
      & $\checkmark$ \cite{iemhoff_uniform_2019}
      & $\checkmark^{\dagger}$
      & $\checkmark$ \cite{iemhoff_uniform_2019} \\
    $\Logic{iK}_\Diamond$
      & $?$ & $?$ & $?$ \\
    $\Logic{CK},\ \Logic{WK}$
      & $\checkmark$ \cite{wijesekera_constructive_1990,giessen_uniform_2026}
      & $\checkmark^{\dagger}$
      & $\checkmark$ \cite{giessen_uniform_2026} \\
    $\Logic{IK}$
      & $\times$ \cite{almeida_fischer-servi_2026} & $\times$ & $\times$ \\
    $\Logic{IK}^-,\ \Logic{FIK},\ \Logic{LIK},\ \Logic{IK}^c$
      & $?$ & $?$ & $?$ \\
    \hline
    $\Logic{WKT},\ \Logic{WKD}$
      & $\checkmark$ \cite{dalmonte_wijesekera-style_2022} & $?$ & $?$ \\
    $\Logic{iKD}_\Box$
      & $\checkmark$ & $?$ & $\checkmark$ \cite{iemhoff_uniform_2019} \\
    $\Logic{iK4}_\Box,\ \Logic{iS4}_\Box$
      & $?$ & $?$ & $\times$ \cite{giessen_uniform_admissible_2022} \\
    $\Logic{iGL}_\Box,\ \Logic{iSL}_\Box$
      & $\checkmark$ \cite{giessen_uniform_admissible_2022}
      & $?$
      & $\checkmark$ \cite{feree_mechanised_2024,giessen_igl_uniform_2026} \\
    $\Logic{IGL}$
      & $\times$ \cite{almeida_fischer-servi_2026} & $\times$ & $\times$ \\
    $\Logic{IEL}^{-},\ \Logic{IEL}$
      & $?$ & $?$ & $?$ \\
    $\Logic{PLL}$
      & $?$ & $?$ & $?$ \cite{iemhoff_proof_2024} \\
    \hline
  \end{tabular*}
  \par \smallskip
  \begin{minipage}{.9\textwidth}
    \footnotesize
    $\checkmark$: established; $\times$: fails; $?$: presumably open;
    $\dagger$: proved in this paper.
    For $\Logic{PLL}$, the UIP proof was retracted in the updated preprint.
  \end{minipage}
\end{table}

\section{Lyndon interpolation for $\Logic{CK}$ and $\Logic{WK}$} \label{appendix:lip}

\def\fCenter{\ \imp\ }

In this appendix, we will prove that $\Logic{CK}$ and $\Logic{WK}$ enjoy CIP and LIP.
Note that, although some of the results here do not seem to have been stated explicitly in the literature
(at least until recently, for CIP of $\Logic{CK}$ in the preprint of van der Giessen and Shillito \cite{giessen_uniform_2026}),
they follow by the same arguments as in \cite{wijesekera_constructive_1990}.

\begin{definition}[{\cite{wijesekera_constructive_1990}, \cite{dalmonte_minimal_2025}}]
  $\mathrm{G}_\Logic{CK}$ is obtained from $\Logic{LJ}$ by adding the following rules:
  \begin{center}
    \Axiom$\Gamma \fCenter \varphi$
    \RightLabel{($\Box$)}
    \UnaryInf$\Box \Gamma \fCenter \Box \varphi$
    \DisplayProof
    \quad
    \Axiom$\Gamma, \varphi \fCenter \psi$
    \RightLabel{($\Diamond_1$)}
    \UnaryInf$\Box \Gamma, \Diamond \varphi \fCenter \Diamond \psi$
    \DisplayProof
  \end{center}

  $\mathrm{G}_\Logic{WK}$ is obtained from $\mathrm{G}_\Logic{CK}$ by adding the following rule:
  \begin{center}
    \Axiom$\Gamma, \varphi \fCenter {}$
    \RightLabel{($\Diamond_0$)}
    \UnaryInf$\Box \Gamma, \Diamond \varphi \fCenter {}$
    \DisplayProof
  \end{center}
\end{definition}

\begin{proposition}[{\cite{wijesekera_constructive_1990}, \cite{dalmonte_minimal_2025}}]
  For $L \in \set{\Logic{CK}, \Logic{WK}}$,
  $\mathrm{G}_L$ admits cut elimination, and
  $\mathrm{G}_L \vdash \Gamma \imp \varphi$ iff $L \vdash \bigwedge \Gamma \to \varphi$.
\end{proposition}

\begin{definition}
  We say $\Gamma_1; \Gamma_2$ is a \emph{partition} of $\Gamma$ if $\Gamma_1 \cup \Gamma_2 = \Gamma$ and $\Gamma_1 \cap \Gamma_2 = \emptyset$.
\end{definition}

\begin{lemma}[{\cite[Lem.\,2.1.5]{wijesekera_constructive_1990}}] \label{lemma:maehara}
  Let $L \in \set{\Logic{CK}, \Logic{WK}}$.
  Let $\Gamma \imp \Delta$ be provable in $\mathrm{G}_L$,
  then for any partition $\Gamma_1; \Gamma_2$ of $\Gamma$, there is $\chi \in \MF$ such that:
  \begin{enumerate}[label=(\alph*)]
    \item $\Gamma_1 \imp \chi$ and $\chi, \Gamma_2 \imp \Delta$ are provable in $\mathrm{G}_L$;
    \item $\VarAny(\chi) \subseteq \VarAny(\Gamma_1) \cap (\VarAnother(\Gamma_2) \cup \VarAny(\Delta))$,
          for each $(\bullet, \circ) \in \set{ (+, -), (-, +) }$.
  \end{enumerate}
  \begin{proof}
    Induction on the height of the cut-free proof of $\Gamma \imp \Delta$.

    Take any partition $\Gamma_1; \Gamma_2$ of $\Gamma$.
    If $\Gamma \imp \Delta$ is an initial sequent or is obtained by a rule of $\Logic{LJ}$,
    then it is proven by the same way one would apply Maehara's method to $\Logic{LJ}$ (e.g. \cite{ono_proof_2019}).
    So we shall only consider the cases where $\Gamma \imp \Delta$ is obtained by ($\Box$), ($\Diamond_1$), or ($\Diamond_0$).

    Suppose that $(\Gamma \imp \Delta) = (\Box \Pi \imp \Box \psi)$ is obtained from $\Pi \imp \psi$ by ($\Box$).
    Let $\Pi_1 = \Box^{-1} \Gamma_1$ and $\Pi_2 = \Box^{-1} \Gamma_2$, then $\Pi_1; \Pi_2$ is a partition of $\Pi$.
    By the induction hypothesis, there is $\chi$ that satisfies (a) and (b) for $\Pi_1; \Pi_2$.
    We shall show that $\Box \chi$ satisfies (a) and (b) for $\Gamma_1;\Gamma_2 = \Box \Pi_1; \Box \Pi_2$:
    \begin{enumerate}[label=(\alph*)]
      \item From $\Pi_1 \imp \chi$, we obtain $\Box \Pi_1 \imp \Box \chi$ by ($\Box$).
            Also from $\chi, \Pi_2 \imp \psi$, we obtain $\Box \chi, \Box \Pi_2 \imp \Box \psi$ by ($\Box$).
      \item By definition, we have $\VarAny(\Box \chi) = \VarAny(\chi)$, $\VarAny(\Box \Pi_i) = \VarAny(\Pi_i)$ ($i = 1,2$), and $\VarAny(\Box \psi) = \VarAny(\psi)$,
            for $\bullet \in \set{+, -}$.
            So it follows from $\chi$ satisfying (b).
    \end{enumerate}

    Suppose that $(\Gamma \imp \Delta) = (\Box \Pi, \Diamond \psi_1 \imp \Diamond \psi_2)$ is obtained from $\Pi, \psi_1 \imp \psi_2$ by ($\Diamond_1$).
    Here, for some separation $\Pi_1;\Pi_2$ of $\Pi$,
    either $\Gamma_1; \Gamma_2 = \Box \Pi_1, \Diamond \psi_1; \Box \Pi_2$ 
    or $\Gamma_1; \Gamma_2 = \Box \Pi_1; \Box \Pi_2, \Diamond \psi_1$.
    \begin{itemize}
      \item Suppose the former, then $\Pi_1, \psi_1; \Pi_2$ is a separation of $\Pi, \psi_1$,
            so by the induction hypothesis, there is $\chi$ that satisfies (a) and (b) for $\Pi_1, \psi_1; \Pi_2$.
            We shall show that $\Diamond \chi$ satisfies (a) and (b) for $\Gamma_1; \Gamma_2 = \Box \Pi_1, \Diamond \psi_1; \Box \Pi_2$:
            \begin{enumerate}[label=(\alph*)]
              \item From $\Pi_1, \psi_1 \imp \chi$, we obtain $\Box \Pi_1, \Diamond \psi_1 \imp \Diamond \chi$ by ($\Diamond_1$).
                    Also from $\Pi_2, \chi \imp \psi_2$, we obtain $\Box \Pi_2, \Diamond \chi \imp \Diamond \psi_2$ by ($\Diamond_1$).
              \item Follows from $\chi$ satisfying (b).
            \end{enumerate}

      \item Suppose the latter, then $\Pi_1; \Pi_2, \psi_1$ is a separation of $\Pi, \psi_1$,
            so by the induction hypothesis, there is $\chi$ that satisfies (a) and (b) for $\Pi_1; \Pi_2, \psi_1$.
            We shall show that $\Box \chi$ satisfies (a) and (b) for $\Gamma_1; \Gamma_2 = \Box \Pi_1; \Box \Pi_2, \Diamond \psi_1$:
            \begin{enumerate}[label=(\alph*)]
              \item From $\Pi_1 \imp \chi$, we obtain $\Box \Pi_1 \imp \Box \chi$ by ($\Box$).
                    Also from $\Pi_2, \chi, \psi_1 \imp \psi_2$, we obtain $\Box \Pi_2, \Box \chi, \Diamond \psi_1 \imp \Diamond \psi_2$ by ($\Diamond_1$).
              \item Follows from $\chi$ satisfying (b).
            \end{enumerate}
    \end{itemize}

    Suppose that $(\Gamma \imp \Delta) = (\Box \Pi, \Diamond \psi \imp {})$ is obtained from $\Pi, \psi \imp {}$ by ($\Diamond_0$),
    then the rest of the proof is similar to the previous case.
  \end{proof}
\end{lemma}

\begin{theorem} \label{thm:ck-and-wk-enjoy-lip}
  Both $\Logic{CK}$ and $\Logic{WK}$ enjoy LIP.
  \begin{proof}
    Let $L \in \set{\Logic{CK}, \Logic{WK}}$ and suppose $L \vdash \varphi \to \psi$,
    then $\varphi \imp \psi$ is provable in $\mathrm{G}_L$.
    Consider the partition $\varphi;\emptyset$ of $\varphi$,
    then by Lemma \ref{lemma:maehara}, there is $\chi$ such that
    $\varphi \imp \chi$ and $\chi \imp \psi$ are provable in $\mathrm{G}_L$,
    which imply $L \vdash \varphi \to \chi$ and $L \vdash \chi \to \psi$, respectively,
    and $\VarAny(\chi) \subseteq \VarAny(\varphi) \cap \VarAny(\psi)$ for each $\bullet \in \set{+, -}$.
  \end{proof}
\end{theorem}

\begin{corollary}[{\cite{wijesekera_constructive_1990}, \cite{giessen_uniform_2026}}]
  Both $\Logic{CK}$ and $\Logic{WK}$ enjoy CIP.
\end{corollary}

\begin{corollary} \label{cor:ikb-enjoys-lip}
  $\Logic{iK}_\Box$ enjoys LIP.
  \begin{proof}[Proof (sketch)]
    By the admissibility of cut in $\mathrm{G}_\Logic{CK}$ and $\Logic{iK}_\Box = \Logic{CK} \cap \MFb$,
    we can show that $\Logic{LJ} + (\Box)$ is exactly a sequent calculus for $\Logic{iK}_\Box$.
    Then repeat the above proof for $\Logic{iK}_\Box$ and $\Logic{LJ} + (\Box)$.
  \end{proof}
\end{corollary}

\begin{problem}
  Does $\Logic{iK}_\Diamond$ enjoy CIP, LIP, and UIP?
  Note that we can't use the same technique as in $\Logic{iK}_\Box$ since $\Logic{WK} \cap \MFd \subsetneq \Logic{iK}_\Diamond$.
\end{problem}

\bibliographystyle{plain}
\bibliography{refs}

\end{document}